\documentclass[a4paper,11pt]{article}
\usepackage[latin1]{inputenc} 
\usepackage[T1]{fontenc}
\usepackage{lmodern}

\usepackage{amsthm,amsmath,amsfonts,amssymb,stmaryrd,bbm,mathrsfs}
\usepackage{mathtools}
\usepackage{enumitem}
\usepackage[colorlinks=true, linkcolor=blue, citecolor=blue,pdfstartview=FitH]{hyperref}

\usepackage[top=2.7cm, bottom=2.7cm, left=2.5cm, right=2.5cm]{geometry}

\usepackage[french,english]{babel}
\usepackage{caption,tikz,color}
\usetikzlibrary{shapes}
\usetikzlibrary{patterns}

\makeatletter

\@addtoreset{equation}{section}
\makeatother

\setlist[enumerate,1]{label=(\roman*), font = \normalfont} 

\let\originalleft\left
\let\originalright\right
\renewcommand{\left}{\mathopen{}\mathclose\bgroup\originalleft}
\renewcommand{\right}{\aftergroup\egroup\originalright}

\newlength{\bibitemsep}
\newlength{\bibparskip}
\let\oldthebibliography\thebibliography
\renewcommand\thebibliography[1]{\oldthebibliography{#1}
  \setlength{\parskip}{\bibitemsep}
  \setlength{\itemsep}{\bibparskip}}

\newcommand{\N}{\mathbb{N}}
\newcommand{\Z}{\mathbb{Z}}
\newcommand{\Q}{\mathbb{Q}}
\newcommand{\R}{\mathbb{R}}

\newcommand{\T}{\mathbb{T}}
\renewcommand{\P}{\mathbb{P}}
\newcommand{\E}{\mathbb{E}}

\newcommand{\Ec}[1]{\mathbb{E} \left[#1\right]}
\newcommand{\Pp}[1]{\mathbb{P} \left(#1\right)}
\newcommand{\Ecsq}[2]{\mathbb{E} \left[#1\mathrel{}\middle|\mathrel{}#2\right]}
\newcommand{\Ppsq}[2]{\mathbb{P} \left(#1\mathrel{}\middle|\mathrel{}#2\right)}

\newcommand{\Eci}[2]{\mathbb{E}_{#1} \left[#2\right]}
\newcommand{\Ppi}[2]{\mathbb{P}_{#1} \left(#2\right)}

\newcommand{\Ecsqi}[3]{\mathbb{E}_{#1} \left[#2\mathrel{}\middle|\mathrel{}#3\right]}

\newcommand{\1}{\mathbbm{1}}

\newcommand{\cD}{\mathcal{D}}

\newcommand{\cC}{\mathcal{C}}

\newcommand{\cM}{\mathcal{M}}
\newcommand{\cL}{\mathcal{L}}

\newcommand{\cR}{\mathcal{R}}
\newcommand{\cZ}{\mathcal{Z}}

\newcommand{\sF}{\mathscr{F}}

\newcommand{\sG}{\mathscr{G}}

\newcommand{\bfS}{\mathbf{S}}
\newcommand{\bfV}{\mathbf{V}}

\newcommand{\e}{\mathrm{e}}
\newcommand{\diff}{\mathop{}\mathopen{}\mathrm{d}}

\newcommand{\supp}{\mathrm{supp}}

\newcommand{\Exists}{\exists\,}
\newcommand{\Forall}{\forall\,}

\newcommand{\abs}[1]{\left\lvert#1\right\rvert}
\newcommand{\norme}[1]{\left\lVert#1\right\rVert}

\newcommand{\petito}[1]{o\mathopen{}\left(#1\right)}
\newcommand{\grandO}[1]{O\mathopen{}\left(#1\right)}

\newcommand{\enstq}[2]{\left\{#1\mathrel{}\middle|\mathrel{}#2\right\}}
\newcommand{\restreinta}{\mathclose{}|\mathopen{}}

\newcommand\relphantom[1]{\mathrel{\phantom{#1}}}

\title{\vspace{-0.3cm}The near-critical Gibbs measure of the branching random walk}

\author{Michel \bsc{Pain}\footnote{DMA, \'Ecole Normale Supérieure and LPMA, Sorbonne Universités.}}

\theoremstyle{plain} 
\newtheorem{thm}{Theorem}[section]
\newtheorem{prop}[thm]{Proposition}
\newtheorem{lem}[thm]{Lemma}
\newtheorem{cor}[thm]{Corollary}

\theoremstyle{definition}

\theoremstyle{remark}
\newtheorem{rem}[thm]{Remark}

\begin{document}

\maketitle

\vspace{-0.5cm}

\begin{abstract}
\noindent Consider the supercritical branching random walk on the real line in the boundary case and the associated Gibbs measure $\nu_{n,\beta}$ on the $n$\textsuperscript{th} generation, which is also the polymer measure on a disordered tree with inverse temperature $\beta$.
The convergence of the partition function $W_{n,\beta}$, 
after rescaling, towards a nontrivial limit has been proved by A\"{\i}dékon and Shi \cite{aidekonshi2014} in the critical case $\beta = 1$ and by Madaule \cite{madaule2015} when $\beta >1$. 
We study here the near-critical case, where $\beta_n \to 1$, and prove the convergence of $W_{n,\beta_n}$, after rescaling, towards a constant multiple of the limit of the derivative martingale.
Moreover, trajectories of particles chosen according to the Gibbs measure $\nu_{n,\beta}$ have been studied by Madaule \cite{madaule2016} in the critical case, with convergence towards the Brownian meander, and by Chen, Madaule and Mallein \cite{cmm2015arxiv} in the strong disorder regime, with convergence towards the normalized Brownian excursion.
We prove here the convergence for trajectories of particles chosen according to the near-critical Gibbs measure and display continuous families of processes from the meander to the excursion or to the Brownian motion.
\end{abstract}

\vspace{0.3cm}

{\small 

}

\section{Introduction and main results}

\subsection{Definitions and assumptions}

The branching random walk on the real line is a natural extension of the Galton-Watson process, with addition of a position to each individual, and is defined as follows.
Initially, there is a single particle at the origin, forming the 0\textsuperscript{th} generation.
It gives birth to children, scattered in $\R$ according to some point process $\cL$ and forming the 1\textsuperscript{st} generation. 
Then, each particle of the 1\textsuperscript{st} generation produces its own children disposed around its position according to the law of $\cL$ independently of others: this set of children forms the 2\textsuperscript{nd} generation.
The system goes on indefinitely, unless there is no particle at some generation.
The genealogical tree of the branching random walk, denoted by $\T$, is a Galton-Watson tree (where an individual can have an infinity of children).
For $z \in \T$, we denote by $\abs{z}$ the generation of the particle $z$ and by $V(z)$ its position in $\R$.
We denote by $\Psi$ the log-Laplace transform of $\cL$: we set, for each $\beta \in \R_+$,
\begin{align*}
\Psi (\beta) 
\coloneqq 
\log \E \Biggl[ \sum_{\abs{z} = 1} \e^{-\beta V(z)} \Biggr]
\in (-\infty, \infty],
\end{align*}
noting that $\cL$ has the same law as $(V(z), \abs{z} = 1)$.
Throughout the paper, we assume the following integrability conditions on the reproduction law $\cL$.
First of all, we need to assume that the Galton-Watson tree $\T$ is \textit{supercritical}, that is
\begin{equation} \label{hypothese 1}
\E \Biggl[ \sum_{\abs{z} = 1} 1 \Biggr] >1,
\end{equation}
so that the survival event $S$ has positive probability and, thus, we can introduce the new probability $\P^* \coloneqq \P (\cdot \mid S)$.
Moreover, we work in the \textit{boundary case} (Biggins and Kyprianou \cite{bigginskyprianou2005}) by assuming
\begin{equation} \label{hypothese 2}
\E \Biggl[ \sum_{\abs{z} = 1} \e^{-V(z)} \Biggr] = 1
\quad \text{ and } \quad
\E \Biggl[ \sum_{\abs{z} = 1} V(z) \e^{-V(z)} \Biggr] = 0,
\end{equation}
which means $\Psi(1) = 0$ and $\Psi'(1) = 0$.
See the arXiv version of Jaffuel \cite{jaffuel2012} for discussion on the cases where the branching random walk can be reduced to assumption \eqref{hypothese 2}.
We assume also that
\begin{equation} \label{hypothese 3}
\sigma^2
\coloneqq
\E \Biggl[ \sum_{\abs{z} = 1} V(z)^2 \e^{-V(z)} \Biggr] \in (0, \infty)
\quad \text{ and } \quad
\Ec{X (\log_+ X)^2 + \widetilde{X} \log_+ \widetilde{X}} < \infty,
\end{equation}
where we set, for $y \geq 0$, $\log_+ y \coloneqq \max (0, \log y)$ and
\begin{equation} \label{equation def X et tildeX}
X 
\coloneqq
\sum_{\abs{z} = 1} \e^{-V(z)}
\quad \text{ and } \quad
\widetilde{X}
\coloneqq
\sum_{\abs{z} = 1} V(z)_+ \e^{-V(z)}
\end{equation}
with $V(z)_+ \coloneqq \max(0, V(z))$.
The first part of \eqref{hypothese 3} gives $\Psi''(1) = \sigma^2$.
We will say that $\cL$ is $(h,a)$\textit{-lattice} if $h > 0$ is the largest real number such that the support of $\cL$ is contained by $a + h\Z$ and, then, $h$ is called the \textit{span} of $\cL$. 
In this paper, we work in both lattice and nonlattice cases, but we will need sometime to distinguish these cases.
Finally, we set two last assumptions that are not supposed to hold in the whole paper, but only in specific cases of the results.
The following assumption,
\begin{align}
\Exists 0< \delta_0 < 1 : {}
& \E \Biggl[ 
\sum_{\abs{z} = 1} \e^{- (1- \delta_0) V(z)}
\Biggr] < \infty, \label{hypothese 4}
\end{align}
means that $\Psi$ is finite on $[1-\delta_0,1]$ and, thus, analytic on $(1-\delta_0,1)$ and, using \eqref{hypothese 2} and \eqref{hypothese 3}, we have $\Psi(\beta) = \frac{\sigma^2}{2} (\beta-1)^2 + o((\beta-1)^2)$ as $\beta \uparrow 1$ by a Taylor expansion.
The second one,
\begin{align}
\Exists 0< \delta_1 < \delta_2 < 1/4 :  {}
& \E \Biggl[ 
\Bigl( \sum_{\abs{z} = 1} \e^{- (1- \delta_1) V(z)} \Bigr)^{1+2\delta_2} 
\Biggr] < \infty, \label{hypothese 5}
\end{align}
comes from Madaule \cite{madaule2016} and is probably not optimal for the results where it is used.
Under assumption \eqref{hypothese 5}, $\Psi$ is finite on $[1-\delta_1,1+ \delta_2/2]$ and, therefore, analytic on a neighbourhood of 1, so we can improve the Taylor expansion: $\Psi(\beta) = \frac{\sigma^2}{2} (\beta-1)^2 + O((\beta-1)^3)$ as $\beta \to 1$.

For $n \in \N$ and $\beta \in \R_+$, we set
\begin{align*}
W_{n,\beta}
\coloneqq
\sum_{\abs{z} = n} \e^{-\beta V(z)}
\quad \text{ and } \quad
\nu_{n,\beta}
\coloneqq
\frac{1}{W_{n,\beta}}
\sum_{\abs{z} = n} 
\e^{-\beta V(z)} 
\delta_z,
\end{align*}
as soon as $W_{n,\beta} < \infty$, which holds a.\@s.\@ for $\beta \geq 1$ under assumption \eqref{hypothese 2}.
Then, $\nu_{n,\beta}$ is a random probability measure on $\{z \in \T : \abs{z} = n \}$,
which is called the \textit{Gibbs measure} of parameter $\beta$ on the $n$\textsuperscript{th} generation of the branching random walk.
It is also the law of a directed polymer on the tree $\T$ in a random environment, introduced by Derrida and Spohn \cite{derridaspohn88} as a mean field limit for directed polymer on a lattice as dimension goes to infinity: with this terminology, $V(z)$ is the energy of the path leading from the root to particle $z$, $\beta$ is the inverse temperature and $W_{n,\beta}$ is the \textit{partition function}.
In the case $\beta = \infty$, we can define $\nu_{n,\infty}$ as the uniform measure on the random set $\{ \abs{x} = n : V(x) = \min_{\abs{z} = n} V(z) \}$.
According to Derrida and Spohn \cite{derridaspohn88}, there is a critical parameter $\beta_c > 0$ for the directed polymer on a disordered tree (with our setting $\beta_c = 1$, see Subsection \ref{subsection partition function} for more details) and our aim in this paper is to study the near-critical case, where $\beta$ depends on $n$ and tends from above and below to $\beta_c = 1$ as $n\to \infty$.
The near-critical case has been recently studied for the directed polymer on the lattice in dimension $1+1$ and $1+2$ by Alberts, Khanin and Quastel \cite{akq2014-2} and Caravenna, Sun and Zygouras \cite{csz2014arxiv,csz2016arxiv}, with the emergence of the so-called intermediate disorder regime.
For the polymer on a tree, some work near criticality has been done by Alberts and Ortgiese \cite{albertsortgiese2013} and Madaule \cite{madaule2016}, mostly on the partition function.
Before stating our results, we recall some well-known properties of the branching random walk, that hold under assumptions \eqref{hypothese 1}, \eqref{hypothese 2} and \eqref{hypothese 3}.
First, the sequence
\begin{align*}
D_n
\coloneqq
\sum_{\abs{z} = n} V(z) \e^{-V(z)},
\quad 
n\in\N,
\end{align*}
is a martingale, called the \textit{derivative martingale}, and
Biggins and Kyprianou \cite{bigginskyprianou2004} (under slightly stronger assumptions) and A\"{\i}dékon \cite{aidekon2013} showed that we have
\begin{align} \label{eq CV martingale derivee}
D_n
\xrightarrow[n\to\infty]{}
D_\infty > 0, 
\quad
\P^*\text{-a.\@s.}
\end{align}
Moreover, Chen \cite{chen2015-1} proved that these assumptions are optimal for the nontriviality of $D_\infty$.
Furthermore, A\"{\i}dékon \cite{aidekon2013} also showed that, in the nonlattice case, $\min_{\abs{z}=n} V(z) - \frac{3}{2} \log n$ converges in law under $\P^*$ and described the limit as a random shift (depending on $D_\infty$) of a Gumbel distribution.
In the lattice case, we do not have this convergence, but the tightness still holds (see Equation (4.20) of Chen \cite{chen2015-2} or Mallein \cite{mallein2016arxiv-1}): for each $\varepsilon >0$, it exists $C > 0$ such that
\begin{align} \label{eq tension position la plus basse}
\Pp{
\min_{\abs{z}=n} V(z)
- 
\frac{3}{2} \log n
\in
[-C,C] }
\geq
1 - \varepsilon,
\end{align}
for $n$ large enough.

\subsection{The partition function}
\label{subsection partition function}

The process $(W_{n,\beta})_{n\in\N}$ for some fixed $\beta \in\R_+$ has been intensively studied because,
if $\Psi (\beta)$ is finite, then the renormalized process $(\widetilde{W}_{n,\beta})_{n\in\N} \coloneqq (\e^{-n\Psi(\beta)} W_{n,\beta})_{n\in\N}$ is a nonnegative martingale, called \textit{additive martingale}, and, therefore, converges a.\@s.\@ to some limit $\widetilde{W}_{\infty,\beta}$.
Kahane and Peyrière \cite{kahanepeyriere76}, Biggins \cite{biggins77-1} and Lyons \cite{lyons95} have determined when this limit is nontrivial:
under the additional assumption that the expectation $\E[W_{1,\beta} \log_+ W_{1,\beta}]$ is finite, we have the following dichotomy
\begin{align} \label{eq cv additive martingale}
\left\{
\begin{array}{ll}
\widetilde{W}_{\infty,\beta} > 0 \quad \P^*\text{-a.\@s.\@} 
& \text{if } \beta \in [0,1), \\
\widetilde{W}_{\infty,\beta}= 0 \quad \P^*\text{-a.\@s.\@} 
& \text{if } \beta \geq 1.
\end{array}
\right.
\end{align}
With the terminology of polymers' literature (see \cite{cometsyoshida2006}), the region $\beta \in [0,1)$ is thus called the \textit{weak disorder regime} and the region $\beta \geq 1$ the \textit{strong disorder regime}.
In the strong disorder regime $\beta \geq 1$, it is natural to seek a proper renormalization of $W_{n,\beta}$, so that it converges towards a nontrivial limit.
This question has already been answered when $\beta$ does not depend on $n$.
In the critical case $\beta = 1$, A\"{\i}dékon and Shi \cite{aidekonshi2014} proved that we have
\begin{align} \label{eq CV aidekon shi}
\sqrt{n}
W_{n,1}
\xrightarrow[n\to\infty]{}
\frac{1}{\sigma}
\sqrt{\frac{2}{\pi}} D_\infty,
\quad
\text{in }
\P^*
\text{-probability},
\end{align}
and that the particles that contribute mainly to $W_{n,1}$ are those of order $\sqrt{n}$.
On the other hand, Theorem 2.3 of Madaule \cite{madaule2015} shows (in the nonlattice case) that, in the case $\beta > 1$, we have
\begin{align*}
n^{3\beta/2}
W_{n,\beta}
\xrightarrow[n\to\infty]{}
Z_{\beta} D_\infty,
\quad
\text{in law},
\end{align*}
where $Z_\beta$ is a random variable independent of $D_\infty$, whose law have been described by Barral, Rhodes and Vargas \cite{brv2012}.
Moreover, the particles that contribute mainly to $W_{n,\beta}$ are in $[\frac{3}{2} \log n - C, \frac{3}{2} \log n + C]$ with $C$ some large constant and, thus, are close to the lowest particle at time $n$ (see \eqref{eq tension position la plus basse}).
Since there is a discontinuity in the size of the partition function $W_{n,\beta}$ as $\beta \downarrow 1$ between $n^{-3\beta /2}$ and $n^{-1/2}$, and also as $\beta \uparrow 1$ between exponential and polynomial size, we try to understand this transition by considering the near-critical case where $\beta_n \to 1$ as $n\to\infty$.
Some work has been done to this end by Alberts and Ortgiese \cite{albertsortgiese2013}, who considered the case where $\beta_n = 1 \pm n^{-\delta}$ for $\delta > 0$ and proved, under stronger assumptions, that
\begin{align*} 
W_{n,\beta_n}
=
\left\{
\begin{array}{ll}
n^{2\delta - 3/2 + \petito{1}}
	& \text{if } \beta_n = 1 + n^{-\delta} \text{ with } 0 <\delta < 1/2, \\
n^{-1/2 + \petito{1}}
	& \text{if } \beta_n = 1 \pm n^{-\delta} \text{ with } \delta \geq 1/2, \\
\exp \left( \frac{\sigma^2}{2} n^{1-2\delta} (1 + \petito{1}) \right)
	& \text{if } \beta_n = 1 - n^{-\delta} \text{ with } 0 <\delta < 1/2.
\end{array}
\right.
\quad \text{in } \P^* \text{-probability}.
\end{align*}
Moreover, the behavior of the limit $\widetilde{W}_{\infty,\beta}$ of the additive martingale has been studied near criticality by Madaule \cite{madaule2016}, who showed under assumption \eqref{hypothese 5} the following convergence 
\begin{align} \label{eq Madaule equiv additive martingale}
\frac{\widetilde{W}_{\infty,\beta}}{1-\beta}
\xrightarrow[\beta \uparrow 1]{}
2 D_\infty,
\quad
\text{in } \P^*\text{-probability},
\end{align}
and, although this is not exactly our setting where $n\to\infty$ and $\beta \to 1$ simultaneously, this result and its proof are useful in this paper.
The following theorem improves Alberts and Ortgiese's result, showing convergence to a nontrivial limit of $W_{n,\beta_n}$ after rescaling, for every sequence $\beta_n \coloneqq 1  \pm 1/\alpha_n$.
\begin{thm} \label{theorem W}
Assume \eqref{hypothese 1}, \eqref{hypothese 2}, \eqref{hypothese 3} and that $\alpha_n \to \infty$ as $n\to\infty$.
Let $\cM$ denotes the Brownian meander of length 1.
\begin{enumerate}
\item If $\beta_n \coloneqq 1 + 1/\alpha_n$ and $\sqrt{n}/\alpha_n \to \infty$ as $n\to \infty$, then we have
\begin{align*}
\frac{n^{3 \beta_n/2}}{\alpha_n^2}
W_{n,\beta_n}
\xrightarrow[n\to\infty]{}
\frac{1}{\sigma^3}
\sqrt{\frac{2}{\pi}} D_\infty,
\quad
\text{in }
\P^*
\text{-probability}.
\end{align*}
\item If $\beta_n \coloneqq 1 + 1/\alpha_n$ and $\sqrt{n}/\alpha_n \to \gamma \in [0,\infty)$ as $n\to \infty$, then we have
\begin{align*}
\sqrt{n}
W_{n,\beta_n}
\xrightarrow[n\to\infty]{}
\frac{1}{\sigma}
\sqrt{\frac{2}{\pi}} 
\Ec{\e^{-\sigma \gamma \cM(1)}}
D_\infty,
\quad
\text{in }
\P^*
\text{-probability}.
\end{align*}
\item If \eqref{hypothese 4} holds, $\beta_n \coloneqq 1 - 1/\alpha_n$ and $\sqrt{n} / \alpha_n \to \gamma \in [0,\infty)$ as $n\to \infty$, then we have
\begin{align*}
\sqrt{n}
W_{n,\beta_n}
\xrightarrow[n\to\infty]{}
\frac{1}{\sigma}
\sqrt{\frac{2}{\pi}} 
\Ec{\e^{\sigma \gamma \cM(1)}}
D_\infty,
\quad
\text{in }
\P^*
\text{-probability}.
\end{align*}
\item If \eqref{hypothese 5} holds, $\beta_n \coloneqq 1 - 1/\alpha_n$ and $\sqrt{n} / \alpha_n \to \infty$ as $n\to \infty$, then we have
\begin{align*}
\alpha_n \e^{-n \Psi(\beta_n)} 
W_{n,\beta_n}
\xrightarrow[n\to\infty]{}
2 D_\infty,
\quad
\text{in }
\P^*
\text{-probability}.
\end{align*}
\end{enumerate}
\end{thm}
Note that, in case (i), the size of $W_{n,\beta_n}$ can be a $o(n^{-3/2})$ as soon as $\alpha_n \log \alpha_n \ll \log n$: this possibility does not appear in Alberts and Ortgiese's result.

\subsection{Trajectory of particles under the Gibbs measure}

The second main result of this paper concerns the trajectory of particles chosen according to the Gibbs measure.
We first need to introduce some additional notation.
For a particle $z$ at the $n$\textsuperscript{th} generation and $0 \leq i \leq n$, we denote by $z_i$ its ancestor at the $i$\textsuperscript{th} generation and we set
\begin{align*}
\mathbf{V} (z)
\coloneqq
\left( \frac{V(z_{\lfloor tn \rfloor})}{\sigma\sqrt{n}}, t\in [0,1] \right)
\end{align*}
the rescaled trajectory of $z$'s lineage.
We work here in the set $\cD([0,1])$ of the càdlàg functions from $[0,1]$ to $\R$, with the Skorokhod distance (see Section \ref{section appendix D([0,1])}).
For $n \in \N$ and $\beta \geq 1$, we denote by $\mu_{n,\beta}$ the image measure of the Gibbs measure $\nu_{n,\beta}$ by $\bfV$, that is the random measure on $\cD([0,1])$ such that, for each $F \in \cC_b(\cD([0,1]))$, we have
\begin{align*}
\mu_{n,\beta} (F)
\coloneqq
\frac{1}{W_{n,\beta}}
\sum_{\abs{z} = n} 
\e^{-\beta V(z)} 
F \left( \bfV (z) \right),
\end{align*}
where $\cC_b(\cD([0,1]))$ denotes the set of continuous bounded functions from $\cD([0,1])$ to $\R$.
Convergence of $\mu_{n,\beta}$ has already been studied in the strong disorder regime when $\beta$ does not depend on $n$, under assumptions \eqref{hypothese 1}, \eqref{hypothese 2} and \eqref{hypothese 3}.
In the critical case $\beta =1$, Theorem 1.2 of Madaule \cite{madaule2016} shows%
\footnote{Actually, Madaule \cite{madaule2016} considers the linear interpolation of the trajectory, instead of $\bfV$, and the convergence on $\cC([0,1])$, instead of $\cD([0,1])$. 
But the convergence \eqref{eq CV Madaule} follows 
from Madaule's result.}
that, for all $F \in \cC_b(\cD([0,1]))$,
\begin{align} \label{eq CV Madaule}
\mu_{n,1} (F)
\xrightarrow[n\to\infty]{}
\Ec{F(\cM)},
\quad
\text{in }
\P^*
\text{-probability}.
\end{align}
On the other hand, in the case $\beta >1$, Chen, Madaule and Mallein \cite{cmm2015arxiv} proved (in the nonlattice case) that, under $\P^*$, we have, for all uniformly continuous $F \in \cC_b(\cD([0,1]))$,
\begin{align} \label{eq CV CMM}
\mu_{n,\beta} (F)
\xrightarrow[n\to\infty]{}
\sum_{k \in \N} p_k F(\mathfrak{e}_k),
\quad
\text{in law},
\end{align}
where $(\mathfrak{e}_k)_{k\in\N}$ is a sequence of i.\@i.\@d.\@ normalized Brownian excursions and $(p_k)_{k\in\N}$ follows an independent Poisson-Dirichlet distribution with parameter $(\beta^{-1},0)$.
The convergence in \eqref{eq CV CMM} is believed to hold for all $F \in \cC_b(\cD([0,1]))$.
Moreover, Chen \cite{chen2015-2} considered the case $\beta = \infty$ and showed that, for all $F \in \cC_b(\cD([0,1]))$, we have $\E^*[\mu_{n,\infty}(F)] \to \Ec{F(\mathfrak{e})}$ as $n \to \infty$, where $\mathfrak{e}$ denotes the normalized Brownian excursion.
Finally, in the weak disorder regime $\beta <1$, if there is some $p > 1$ such that $\E[W_{1,\beta}^p] < \infty$, we have the following convergence, with $\sigma_\beta^2 \coloneqq \Psi''(\beta)$, for all $F \in \cC_b(\cD([0,1]))$,
\begin{align} \label{eq cv trajectories weak disorder regime}
\frac{1}{W_{n,\beta}}
\sum_{\abs{z} = n} 
\e^{-\beta V(z)} 
F \left( 
\frac{V(z_{\lfloor tn \rfloor}) + t n\Psi'(\beta)}{\sigma_\beta \sqrt{n}}, 
t\in [0,1] \right)
\xrightarrow[n\to\infty]{}
\Ec{F(B)},
\end{align}
in $\P^*$-probability%
\footnote{
To our knowledge, this result has not been proved yet (except when $F$ only depends on the final position, see \cite{biggins79}), so a proof is given in Section \ref{section appendix weak disorder} of the appendix.}.
It means that the trajectory is a straight line of slope $-\Psi'(\beta) > 0$ at first order and around which Brownian fluctuations occur at second order.

Our aim is to prove the convergence for trajectories of particles chosen according to the Gibbs measure in the near-critical case, in order to explain how happens the transition between the Brownian excursion, the Brownian meander and the straight line with Brownian fluctuations.
\begin{thm} \label{theorem trajectory}
Assume \eqref{hypothese 1}, \eqref{hypothese 2}, \eqref{hypothese 3} and that $\alpha_n \to \infty$ as $n\to\infty$.
Let $\mathfrak{e}$ denotes the normalized Brownian excursion,
$\cM$ the Brownian meander of length 1
and $B$ the Brownian motion.
\begin{enumerate}
\item If $\beta_n \coloneqq 1 + 1/\alpha_n$ and $\sqrt{n}/\alpha_n \to \infty$ as $n\to \infty$, then we have, for all $F \in \cC_b(\cD([0,1]))$,
\begin{align*}
\mu_{n,\beta_n} (F)
\xrightarrow[n\to\infty]{}
\Ec{F(\mathfrak{e})},
\quad
\text{in } \P^*\text{-probability}.
\end{align*}
\item If $\beta_n \coloneqq 1 + 1/\alpha_n$ and $\sqrt{n}/\alpha_n \to \gamma \in [0,\infty)$ as $n\to \infty$, then we have, for all $F \in \cC_b(\cD([0,1]))$,
\begin{align*}
\mu_{n,\beta_n} (F)
\xrightarrow[n\to\infty]{}
\frac{1}{\Ec{\e^{-\sigma \gamma \cM(1)}}}
\Ec{\e^{-\sigma \gamma \cM(1)} F(\cM)},
\quad
\text{in } \P^*\text{-probability}.
\end{align*}
\item If \eqref{hypothese 4} holds, $\beta_n \coloneqq 1 - 1/\alpha_n$ and $\sqrt{n} / \alpha_n \to \gamma \in [0,\infty)$ as $n\to \infty$, then we have, for all $F \in \cC_b(\cD([0,1]))$,
\begin{align*}
\mu_{n,\beta_n} (F)
\xrightarrow[n\to\infty]{}
\frac{1}{\Ec{\e^{\sigma \gamma \cM(1)}}}
\Ec{\e^{\sigma \gamma \cM(1)} F(\cM)},
\quad
\text{in } \P^*\text{-probability}.
\end{align*}
\item If \eqref{hypothese 5} holds, $\beta_n \coloneqq 1 - 1/\alpha_n$ and $\sqrt{n} / \alpha_n \to \infty$ as $n\to \infty$, then we have, for all $F \in \cC_b(\cD([0,1]))$, in $\P^*$-probability,
\begin{align*}
\frac{1}{W_{n,\beta_n}}
\sum_{\abs{z} = n} 
\e^{-\beta_n V(z)} 
F \left( 
\frac{V(z_{\lfloor tn \rfloor}) + t n\Psi'(\beta_n)}{\sigma\sqrt{n}}, 
t\in [0,1] \right)
\xrightarrow[n\to\infty]{}
\Ec{F(B)}.
\end{align*}
\end{enumerate}
\end{thm}

We now state a corollary of this theorem, concerning the location of the mass of the Gibbs measure. 
Note that, for the terminology of the polymers' literature, this position is the typical energy of the polymer in the near-critical case.
\begin{cor} \label{cor typical energy}
Assume \eqref{hypothese 1}, \eqref{hypothese 2}, \eqref{hypothese 3} and that $\alpha_n \to \infty$ as $n\to\infty$.
\begin{enumerate}
\item If $\beta_n \coloneqq 1 + 1/\alpha_n$ and $\sqrt{n}/\alpha_n \to \infty$ as $n\to \infty$, then, for all $\varepsilon > 0$, it exists $C > 0$ such that for $n$ large enough
\begin{align*}
\P^* \left(
	\nu_{n,\beta_n}
	\left( \left[
	\frac{3}{2} \log n + C^{-1} \alpha_n, \frac{3}{2} \log n + C \alpha_n
	\right] \right)
	\geq 1-\varepsilon
	\right)
\geq
1 - \varepsilon.
\end{align*}
\item If $\beta_n \coloneqq 1 + 1/\alpha_n$ and $\sqrt{n}/\alpha_n \to \gamma \in [0,\infty)$ as $n\to \infty$, then, for all $\varepsilon > 0$, it exists $C > 0$ such that for $n$ large enough
\begin{align*}
\P^* \left(
	\nu_{n,\beta_n}
	\left( \left[ C^{-1} \sqrt{n}, C \sqrt{n} \right] \right)
	\geq 1-\varepsilon
	\right)
\geq
1 - \varepsilon.
\end{align*}
\item If \eqref{hypothese 4} holds, $\beta_n \coloneqq 1 - 1/\alpha_n$ and $\sqrt{n} / \alpha_n \to \gamma \in [0,\infty)$ as $n\to \infty$, then the same property as in case (ii) holds.
\item If \eqref{hypothese 5} holds, $\beta_n \coloneqq 1 - 1/\alpha_n$ and $\sqrt{n} / \alpha_n \to \infty$ as $n\to \infty$, then, for all $\varepsilon > 0$, it exists $C > 0$ such that for $n$ large enough
\begin{align*}
\P^* \left(
	\nu_{n,\beta_n}
	\left( \left[
	- \Psi'(\beta_n) n - C \sqrt{n}, - \Psi'(\beta_n) n + C \sqrt{n}
	\right] \right)
	\geq 1-\varepsilon
	\right)
\geq
1 - \varepsilon.
\end{align*}
\end{enumerate}
\end{cor}
\begin{proof}
In cases (ii), (iii) and (iv), it is a direct consequence of Theorem \ref{theorem trajectory}.
For case (i), the assertion will be proved at the end of Subsection \ref{subsection proof of proposition 4.1 and of part (i) of Theorem 1.3}.
\end{proof}

\subsection{Genealogy under the Gibbs measure}
\label{subsection genealogy}

We state here a direct consequence of Theorem \ref{theorem trajectory} concerning the overlap in the branching random walk, introduced by Derrida and Spohn \cite{derridaspohn88} in the context of polymers on trees.
We set, for $x, y \in \T$,
\begin{align*}
\abs{x\wedge y} \coloneqq \max \{ k \leq \min (\abs{x},\abs{y}) : x_k = y_k \},
\end{align*}
that is the generation of the most recent common ancestor of $x$ and $y$.
For some couple of particles $(x,y)$ chosen according to $\nu_{n,\beta}^{\otimes 2}$, we are interested in the \textit{overlap} between $x$ and $y$ defined by $\abs{x \wedge y} / n$.
Thus, we set, for a Borel set $A \subset [0,1]$,
\begin{align*}
\omega_{n,\beta} (A)
\coloneqq
\nu_{n,\beta}^{\otimes 2}
\left(
\frac{\abs{x \wedge y}}{n}
\in A
\right)
\end{align*}
so that $\omega_{n,\beta}$ is a random probability measure on $[0,1]$.
Madaule \cite{madaule2016} gives the following consequence of \eqref{eq CV Madaule} in the case $\beta = 1$:
\begin{align*}
\omega_{n,1}
\xrightarrow[n\to\infty]{}
\delta_0,
\quad
\text{in }
\P^* \text{-probability}.
\end{align*}
For the other extremal case $\beta = \infty$, one can prove in the nonlattice case only%
\footnote{For example, using the stopping line $\cZ [A] \coloneqq \{z \in \T : V(z) \geq A > \max_{k < \abs{z}} V(z_k) \}$ for large $A$, it is possible to show that, with high probability for $n$ large, all particles under $\nu_{n,\infty}$ have the same ancestor in $\cZ [A]$. 
Then \eqref{eq genealogy beta = infinity} follows from Lemma 3.3 of Chen \cite{chen2015-2}.
In the lattice case, it is clear that \eqref{eq genealogy beta = infinity} cannot hold.}
that, 
\begin{align} \label{eq genealogy beta = infinity}
\omega_{n,\infty}
\xrightarrow[n\to\infty]{}
\delta_1,
\quad
\text{in }
\P^* \text{-probability}.
\end{align}
The transition between this two cases appears with case $\beta \in (1,\infty)$, with which Chen, Madaule and Mallein \cite{cmm2015arxiv} deal, but their result \eqref{eq CV CMM} is only proved for $F \in \cC_b^u(\cD([0,1]))$ and, thus, the convergence in law of $\omega_{n,\beta}$ cannot be obtained as a corollary. 
However, Mallein \cite{mallein2016arxiv-2} shows that, under $\P^*$, we have
\begin{align*}
\omega_{n,\beta}
\xrightarrow[n\to\infty]{}
(1-\pi_\beta) \delta_0
+
\pi_\beta \delta_1,
\quad
\text{in law},
\end{align*}
where $\pi_\beta \coloneqq \sum_{k\in\N} p_k^2$ and $(p_k)_{k\in\N}$ follows a Poisson-Dirichlet distribution with parameter $(\beta^{-1},0)$.
It confirms a conjecture of Derrida and Spohn \cite{derridaspohn88}.
In the near critical case, we can state the following consequence of Theorem \ref{theorem trajectory}.
\begin{cor} 
In each case of Theorem \ref{theorem W}, under the same assumptions, we have
\begin{align*}
\omega_{n,\beta_n}
\xrightarrow[n\to\infty]{}
\delta_0,
\quad
\text{in }
\P^* \text{-probability}.
\end{align*}
\end{cor}
\begin{proof}
We give the proof for the case (i), but it is exactly the same for other cases (only the limiting trajectory changes).
First note that, for all $F \in \cC_b(\cD([0,1])^2)$, we have, with $\mathfrak{e}$ and $\mathfrak{e}'$ independent normalized Brownian excursions,
\begin{align} \label{pb}
\mu_{n,\beta_n}^{\otimes 2} (F)
=
\frac{1}{W_{n,\beta_n}^2}
\sum_{\abs{x} = \abs{y} = n} 
\e^{-\beta_n V(x) - \beta_n V(y)} 
F \left( \bfV (x), \bfV (y) \right)
\xrightarrow[n\to\infty]{}
\Ec{F(\mathfrak{e}, \mathfrak{e}')},
\end{align}
in $\P^*$-probability and therefore in $L^1$, because $\mu_{n,\beta_n}^{\otimes 2} (F)$ is bounded.
Indeed, by Theorem \ref{theorem trajectory}, \eqref{pb} holds when $F$ is of the form $F(x,y) = G_1(x) G_2(y)$ for some $G_1, G_2 \in \cC_b(\cD([0,1]))$ and, the general case follows.
Then, we consider some $\varepsilon>0$ and the closed set $A \coloneqq \{ (x,y) \in \cD([0,1])^2 : \Forall t \leq \varepsilon, x_t = y_t \}$ of $\cD([0,1])^2$.
We have
\begin{align*}
\Ec{\omega_{n,\beta_n} ([\varepsilon, 1]) }
\leq 
\Ec{\mu_{n,\beta_n}^{\otimes 2} (A)}
\xrightarrow[n\to\infty]{}
\Pp{\Forall t \leq \varepsilon, \mathfrak{e}_t = \mathfrak{e}'_t}
=
0,
\end{align*}
using \eqref{pb} and the Portmanteau theorem.
\end{proof}

\subsection{Comments on the results}

Theorem \ref{theorem W} fully describes the transition from the size $n^{-3\beta / 2}$ when $\beta > 1$ to the size $\e^{n\Psi(\beta)}$ when $\beta < 1$ and
Corollary \ref{cor typical energy} shows the transition in the location of the mass of the Gibbs measure from $[\frac{3}{2} \log n - C, \frac{3}{2} \log n + C]$ when $\beta > 1$, to $[- \Psi'(\beta)n - C \sqrt{n}, - \Psi'(\beta)n + C \sqrt{n}]$ when $\beta < 1$ (this follows from \eqref{eq cv trajectories weak disorder regime}).
Secondly, Theorem \ref{theorem trajectory} describes the transition between the Brownian excursion in case (i) and the straight line with Brownian fluctuations in case (iv).
Note that, since $\min_{z\in \T} V(z) > - \infty$ a.\@s.\@, it is natural that the limiting trajectory stays nonnegative on $[0,1]$ in cases (i) to (iii).
In case (iv), this constraint disappears in the limit due to the drift. 
Indeed, staying above a constant for a random walk with drift approximately $\sigma \sqrt{n} / \alpha_n$ needs effort until times of order $\alpha_n^2$, so it disappears in the trajectory after scaling by $n$. 
However, this effort has a cost of order $1/\alpha_n$, which explains the presence of this factor in the size $\e^{n\Psi(\beta_n)}/\alpha_n$ of the partition function. 
In cases (ii) and (iii), when $\alpha_n$ is of order $\sqrt{n}$ or larger, the perturbation is too small to change radically the behavior of the Gibbs measure in regards to the critical case $\beta = 1$: 
the size of the partition function is still $n^{-1/2}$, the typical energy is of order $\sqrt{n}$ and the limiting trajectory has a density w.r.t.\@ the Brownian meander.
Therefore, cases (ii) and (iii) are called the \textit{critical window} by Alberts and Ortgiese \cite{albertsortgiese2013}.
It brings to light a family of laws $(P^\gamma)_{\gamma \in \R}$ on $\cD([0,1])$ defined by 
$P^\gamma (F)
\coloneqq \E[\e^{-\sigma \gamma \cM(1)} F(\cM)] / \E[\e^{-\sigma \gamma \cM(1)}]$ 
for all $F \in \cC_b(\cD([0,1]))$, including the law of the Brownian meander of length 1 for $\gamma = 0$ and such that $P^\gamma \Rightarrow \cL(\mathfrak{e})$ as $\gamma \to \infty$, where $\cL(\mathfrak{e})$ denotes the law of $\mathfrak{e}$.
But there is no convergence as $\gamma \to -\infty$, because the trajectory is sent to infinity.
However, we can consider another family $(Q^\gamma)_{\gamma \in \R}$ defined by
$Q^\gamma (F)
\coloneqq \E[\e^{\sigma \gamma \cM(1)} F(\cM_t - \sigma \gamma t, t\in [0,1])] / \E[\e^{\sigma \gamma \cM(1)}]$, so that, in case (iii), we have 
\begin{align*}
\frac{1}{W_{n,\beta_n}}
\sum_{\abs{z} = n} 
\e^{-\beta_n V(z)} 
F \left( 
\frac{V(z_{\lfloor tn \rfloor}) + t n\Psi'(\beta_n)}{\sigma\sqrt{n}}, 
t\in [0,1] \right)
\xrightarrow[n\to\infty]{}
Q^\gamma(F)
\quad
\text{ in }
\P^*
\text{-probability},
\end{align*}
This family includes also the Brownian meander's law and we have $Q^\gamma \Rightarrow \cL(B)$ as $\gamma \to \infty$.
As opposed to this, cases (i) and (iv) are called the \textit{near-critical window} and highlight some new behaviors.
In case (i), the transition between the critical size $n^{-1/2}$ and the strong disorder size $n^{-3\beta /2}$ appears.
The factor $n^{3\beta_n /2}$ starts to behave differently than $n^{3/2}$ when $\alpha_n = O(\log n)$ and this is also the region where the particles mainly contributing to $W_{n,\beta_n}$ are not simply those in $[C^{-1} \alpha_n, C \alpha_n]$: the fact that the lowest particle at time $n$ is around $\frac{3}{2} \log n$ plays a role only in this region (see Lemma \ref{lemma addition of the barrier}).
Nevertheless, in the near-critical regime, the lowest particle at time $n$ never has a positive weight in $\nu_{n,\beta_n}$ in the limit, unlike in the case $\beta >1$.
Since the particles mainly contributing to $W_{n,\beta_n}$ are far below $\sqrt{n}$, the endpoint of the limiting trajectory has to be 0 and, therefore, the excursion appears. 
In case (iv), since $\beta_n$ tends to $1$ slowly enough, we find the same asymptotic behavior for the partition function as in \eqref{eq Madaule equiv additive martingale} when we first take $n\to \infty$ and then $\beta \uparrow 1$.
For the limiting trajectory, the result is also similar to \eqref{eq cv trajectories weak disorder regime} in the case $\beta > 1$.
Moreover, when $\alpha_n$ is not too small, the different results in case (iv) can be rewritten only in terms of $\sigma^2$: 
if $\alpha_n \gg n^{1/3}$, the size of the partition function is $\e^{\sigma^2n/2\alpha_n^2} /\alpha_n$ and, if $n^{1/4} = O(\alpha_n)$, the location of the mass is in $[\sigma^2 n /\alpha_n -C\sqrt{n}, \sigma^2 n /\alpha_n +C\sqrt{n}]$.
But, on the contrary, if $\alpha_n$ is too small, there is a break of universality.
Finally, we stress that there is no discontinuity between the different cases of the results.
Indeed, using that $\E[\e^{-\sigma \gamma \cM(1)}] \sim 1/(\sigma \gamma)^2$
and $\E[\e^{\sigma \gamma \cM(1)}] \sim \sqrt{2\pi} \sigma \gamma \e^{(\sigma \gamma)^2 /2}$ as $\gamma \to \infty$, all cases of Theorem \ref{theorem W} (requiring $\alpha_n \gg n^{1/3}$ in case (iv)) can be written
\begin{align} \label{eq reecriture theorem W}
n^{3(1-\beta_n)/2} \frac{\sqrt{n}}{\Ec{\e^{(1-\beta_n)\sigma \cM(1) \sqrt{n}}}}
W_{n,\beta_n}
\xrightarrow[n\to\infty]{}
\frac{1}{\sigma}
\sqrt{\frac{2}{\pi}} 
D_\infty,
\quad
\text{in }
\P^*
\text{-probability},
\end{align}
noting that $n^{3(1-\beta_n)/2} \to 1$ as soon as $\alpha_n \gg \log n$.
For Theorem \ref{theorem trajectory}, the continuity between the different cases follows from the convergences $P^\gamma \Rightarrow \cL(\mathfrak{e})$ and $Q^\gamma \Rightarrow \cL(B)$ as $\gamma \to \infty$.

\subsection{Organization of the paper}

Sections of this paper correspond to the different cases in the results: case (i) is treated in Section \ref{section near-critical window in the strong disorder regime}, cases (ii) and (iii) in Section \ref{section critical window} and case (iv) in Section \ref{section near-critical window in the weak disorder regime}.
The behavior in the critical window (cases (ii) and (iii)) is a direct consequence of the analogue results \eqref{eq CV aidekon shi} and \eqref{eq CV Madaule} in the critical case (apart from some technical details).
The near-critical window in the weak disorder regime (case (iv)) needs slightly more work, but relies mainly on Madaule's \cite{madaule2016} results and on $L^p$ inequalities techniques (see \cite{biggins92}).
Finally, the main part of this paper is dedicated to the proof of case (ii), which follows some ideas of A\"{\i}dékon and Shi \cite{aidekonshi2014}, with change of measure and spine decomposition techniques. 
One main difference with the previous literature on the branching random walk is that we need here to consider particles that are far below $\sqrt{n}$ but also far above the lowest particle.
Note that we prove in Subsection \ref{subsection peeling lemma} a new version of the so-called peeling lemma (see Shi \cite{shi2015}) with a more general setting than what is needed for the aim of this paper and that could be of independent interest.
On the other hand, Section \ref{section preliminary results} regroups some well-known results on the branching random walk and on classical random walk.
Some new results are stated in this section and proved in the appendix.
Note that none of the results of Section \ref{section preliminary results} are needed for the proof of cases (ii) and (iii) and only a few of them for case (iv).
The appendix contains some other technical results.
Throughout the paper, the $c_i$'s denote positive constants, 
we set $\N \coloneqq \{0,1,2,\dots \}$ 
and, for $a, b \in \N$, $\llbracket a, b \rrbracket \coloneqq [a,b] \cap \N$.
For two sequences $(u_n)_{n\in \N}$ and $(v_n)_{n\in \N}$ of positive real numbers, 
we say that $u_n \sim v_n$ as $n\to\infty$ if $\lim_{n\to\infty} u_n / v_n = 1$,
that $u_n = \grandO{v_n}$ as $n\to\infty$ if $\limsup_{n\to\infty} u_n / v_n < \infty$,
and that $u_n = \petito{v_n}$ or $u_n \ll v_n$ if $\lim_{n\to\infty} u_n/v_n =0$.
For $(S,d)$ a metric space, let $\cC_b(S)$ be the set of bounded continuous functions from $S \to \R$ and $\cC_b^u(S)$ be its subset of uniformly continuous functions.
For $F \in \cC_b(S)$, we set $\norme{F} \coloneqq \sup_{x\in S} \abs{f(x)}$.
For $F \in \cC^u_b(S)$, we will denote by $\omega_F$ a modulus of continuity for function $F$: $\omega_F$ is a continuous bounded nondecreasing function from $\R_+ \to \R_+$ such that $\omega_F (0) = 0$ and $\Forall x,y \in S$, $\abs{F(x) - F(y)} \leq \omega_F(d(x,y))$.

\section{Preliminary results}
\label{section preliminary results}

In this section, we state some preliminary results that are mostly needed in Section \ref{section near-critical window in the strong disorder regime}.
In Subsection \ref{subsection many-to-one lemma and changes of probabilities}, we present some well-known tools to study the branching random walk and the next subsections contain results concerning one-dimensional random walk.

\subsection{Many-to-one lemma and changes of probabilities}
\label{subsection many-to-one lemma and changes of probabilities}

For $a\in\R$, let $\P_a$ denote a probability measure under which $(V(z), z \in \T)$ is the branching random walk starting from $a$, and $\E_a$ the associated expectation (for brevity we will write $\P$ and $\E$ instead of $\P_0$ and $\E_0$).
We define a random walk $(S_n)_{n\geq0}$ associated to the branching random walk:  under $\P_a$, $S_0 = a$ a.\@s.\@ and the law of the increments is given by 
\begin{align*}
\Eci{a}{h(S_1-S_0)}
=
\E \Biggl[ \sum_{\abs{z} = 1} h(V(z)) \e^{-V(z)} \Biggr],
\end{align*}
for all measurable $h \colon \R \to \R_+$.
This random walk is well-defined and centred thanks to assumption \eqref{hypothese 2}.
Moreover, by assumption \eqref{hypothese 3}, we have $\E[S_1^2] = \sigma^2 \in (0,\infty)$.
Then, by induction, one gets the following result (see Biggins and Kyprianou \cite{bigginskyprianou97}).
It is also a corollary of the forthcoming Proposition \ref{prop lyon's change of measure}.

\begin{lem}[Many-to-one lemma]
For all $n \geq 1$, $a \in \R$ and all measurable function $g \colon \R^{n+1} \to \R_+$, we have
\begin{align*}
\E_a \Biggl[ \sum_{\abs{z} = n} g(V(z_0), \dots, V(z_n)) \Biggr]
=
\Eci{a}{\e^{S_n-a} g(S_0,\dots, S_n)}.
\end{align*}
\end{lem}
We now state some well-known change of probabilities and spinal decomposition results. 
This method dates back at least to Kahane and Peyrière \cite{kahanepeyriere76}, Rouault \cite{rouault81} and Chauvin and Rouault \cite{chauvinrouault88}.
See also Biggins and Kyprianou \cite{bigginskyprianou2004} for spinal decomposition in more general type of branching structures and Shi \cite{shi2015} for a survey on this topic.
Let $\sF_n$ denote the $\sigma$-algebra generated by $(V(z), \abs{z} \leq n)$ and $\sF_\infty \coloneqq \sigma ( \bigvee_{n\in \N} \sF_n )$.
We first introduce Lyons' change of measure \cite{lyons95}: since $(W_{n,1})_{n\in\N}$ is a nonnegative martingale of mean $\e^{-a}$ under $\P_a$, we can define a new probability measure $\Q_a$ on $\sF_\infty$, by letting for all $n\in \N$,
\begin{align*}
\Q_a \restreinta_{\sF_n}
\coloneqq
\e^{a} W_{n,1} \bullet \P_a \restreinta_{\sF_n}.
\end{align*}
We will denote by $\E_{\Q_a}$ the associated expectation and we will write $\Q$ and $\E_\Q$ instead of $\Q_0$ and $\E_{\Q_0}$.
Let $\hat{\cL}$ be a point process on $\R$ which has the law of $(V(z), \abs{z} = 1)$ under $\Q$.
Lyons \cite{lyons95} proved the following description for the branching random walk under $\Q_a$, with a decomposition along a \textit{spine} $(w_n)_{n\in\N}$ which is a marked ray in the the genealogical tree $\T$ (in order to be mathematically rigorous, one should enlarge the probability space and work on a product space, see Lyons \cite{lyons95}).
The system starts with one particle $w_0$ at position $a$, forming the $0$\textsuperscript{th} generation.
For each $n \in \N$, individuals of the $n$\textsuperscript{th} generation give birth independently of each other and of the foregoing. 
Individuals other than $w_n$ generate offspring around their position according to the law of $\cL$ and $w_n$ breeds according to the law of $\hat{\cL}$.
Then, $w_{n+1}$ is chosen independently among $w_n$'s children, with probability proportional to $\e^{-V(z)}$ for each child $z$.
Moreover, Lyons showed the following result concerning the spine $(w_n)_{n\in\N}$ under $\Q_a$.
\begin{prop} \label{prop lyon's change of measure}
Let $a \in \R$.
\begin{enumerate}
\item For each $n\in \N$ and $\abs{z} = n$, we have
\begin{align*}
\Q_a \left( w_n = z \mathrel{}\middle|\mathrel{} \sF_n \right) 
= 
\frac{\e^{-V(z)}}{W_{n,1}}.
\end{align*}
\item The process $(V(w_n))_{n\in\N}$ under $\Q_a$ has the same law as $(S_n)_{n\in\N}$ under $\P_a$.
\end{enumerate}
\end{prop}
Now, we present another change of measure, that was first introduced by Biggins and Kyprianou \cite{bigginskyprianou2004}.
For this, we need to define $R$ the renewal function in the first strict descending ladder height process of the random walk $(S_n)_{n\in\N}$.
For $u \geq 0$,
\begin{align*}
R(u) 
\coloneqq
\sum_{k = 0}^\infty
\Pp{H_k \leq u},
\end{align*}
where $(H_k)_{k\in\N}$ is the first strict descending ladder height process:
we set $\tau_0 \coloneqq 0$, $H_0 \coloneqq 0$ and, for $k \geq 1$, 
$\tau_k \coloneqq \inf \left\{ n > \tau_{k-1} : S_n < S_{\tau_{k-1}} \right\}$
and $H_k \coloneqq - S_{\tau_k}$.
Then, we introduce the truncated derivative martingale: for $L \geq 0$ and $n \in \N$,
\begin{align*}
D_n^{(L)}
\coloneqq
\sum_{\abs{z} = n} 
R_L(V(z)) \e^{-V(z)} \1_{\underline{V}(z) \geq -L},
\end{align*}
where, for $u \geq -L$, $R_L(u) \coloneqq R(L+u)$ and, for $\abs{z} = n$, $\underline{V}(z) \coloneqq \min_{0\leq i \leq n} V(z_i)$.
Fix now some $L \geq 0$.
For $a \geq -L$, under $\P_a$, $(D^{(L)}_n)_{n\in\N}$ is a nonnegative martingale of mean $R_L(a)$ and therefore we can define another probability measure $\Q_a^{(L)}$ by
\begin{align*}
\Q_a^{(L)} \restreinta_{\sF_n}
\coloneqq
\frac{D_n^{(L)}}{R_L(a)} \bullet \P_a \restreinta_{\sF_n}.
\end{align*}
We will denote by $\E_{\Q_a^{(L)}}$ the associated expectation and write $\Q^{(L)}$ and $\E_\Q^{(L)}$ instead of $\Q^{(L)}_0$ and $\E_{\Q^{(L)}_0}$.
For $a \geq -L$, let $\hat{\cL}^{(L)}_a$ be a point process on $\R$ with the law of $(V(z), \abs{z} = 1)$ under $\Q^{(L)}_a$.
Biggins and Kyprianou \cite{bigginskyprianou2004} proved the following spinal decomposition description for the branching random walk under $\Q_a^{(L)}$, where the spine is denoted by $(w_n^{(L)})_{n\in\N}$.
The description is similar to the previous one, but here $w_n^{(L)}$ have offspring according to $\hat{\cL}^{(L)}_{V(w_n^{(L)})}$ and $w_{n+1}^{(L)}$ is chosen among these children, with probability proportional to $R_L(V(z)) \e^{-V(z)} \1_{\underline{V}(z) \geq -L}$ for each child $z$.
Moreover, we get the following analogue of Proposition \ref{prop lyon's change of measure}.
\begin{prop} \label{prop biggins and kyprianou's change of measure}
Let $L \geq 0$ and $a \geq -L$.
\begin{enumerate}
\item For each $n\in \N$ and $\abs{z} = n$, we have
\begin{align*}
\Q^{(L)} \left( w_n^{(L)} = z \mathrel{}\middle|\mathrel{} \sF_n \right) 
= 
\frac{R_L(V(z)) \e^{-V(z)} \1_{\underline{V}(z) \geq -L}}{D_n^{(L)}}.
\end{align*}
\item The process $(V(w_n^{(L)}))_{n\in\N}$ under $\Q_a^{(L)}$ has the same law as $(S_n)_{n\in\N}$ under $\P_a$ conditioned to stay in $[-L,\infty)$: for all $n \in \N$ and all measurable function $g \colon \R^{n+1} \to \R_+$,
\begin{align*}
\Eci{\Q^{(L)}_a}{g \left(V(w_0^{(L)}), \dots, V(w_n^{(L)}) \right)}
=
\frac{1}{R_L(a)}
\Eci{a}{g(S_0,\dots, S_n) R_L(S_n) \1_{\underline{S}_n \geq -L}}.
\end{align*}
\end{enumerate}
\end{prop}

\subsection{One-dimensional random walks}

Up to the end of this section, we consider a centred random walk $(S_n)_{n\in\N}$ with finite variance $\E[S_1^2] = \sigma^2 \in (0,\infty)$.
In this subsection, we state various results concerning this one-dimensional random walk and the associated renewal function $R$.
For $n \in \N$, we set $\underline{S}_n \coloneqq \min_{0\leq i \leq n} S_i$.
Recall that $R$ is the renewal function associated to the first strict descending ladder height process $(H_k)_{k\in\N}$.
Since $\E[S_1] = 0$ and $\E[S_1^2] < \infty$, by Feller \cite[Theorem XVIII.5.1 (5.2)]{feller71}, we have $\E[H_1] < \infty$.
Thus, it follows from Feller's \cite[p.\@ 360]{feller71} renewal theorem that it exists a constant $c_0 > 0$ such that
\begin{equation} \label{equation R 1}
\frac{R(u)}{u} 
\xrightarrow[u \to \infty]{} 
c_0
\end{equation}
and so there exist also $c_1, c_2 >0$ such that, for all $u \geq 0$,
\begin{equation} \label{equation R 2}
c_1 (1+u) \leq R(u) \leq c_2 (1+u).
\end{equation}
We are interested in the behavior of random walks staying above a barrier.
First, we recall the following estimate for the probability of a random walk to stay above $-a$:
by Kozlov \cite[Theorem A]{kozlov76}, it exists a constant $\theta >0$ such that for all $u \geq 0$,
\begin{equation} \label{equation equivalent proba min S geq -a}
\Pp{\underline{S}_n \geq -u} 
\underset{n \to \infty}{\sim}
\frac{\theta R(u)}{\sqrt{n}},
\end{equation}
and it exists $c_3 > 0$ such that, for all $n\in \N$ and $u \geq 0$, we have the uniform bound
\begin{equation} \label{equation majoration proba min S geq -a}
\Pp{\underline{S}_n \geq -u} 
\leq
\frac{c_3 (1+u)}{\sqrt{n}}.
\end{equation}
Constants $c_0$ and $\theta$ will appear all along the paper and they are related by the following equation, from A\"idékon and Shi \cite[Lemma 2.1]{aidekonshi2014},
\begin{equation} \label{equation lien theta c_0}
\theta c_0 = \left( \frac{2}{\pi \sigma^2} \right)^{1/2}.
\end{equation}
The following result states the convergence in law of $S_n / \sigma \sqrt{n}$ conditioned to stay above $-u$ towards a Rayleigh distribution, with uniformity in the position of the barrier $-u$: 
by A\"idékon and Jaffuel \cite[Lemma 2.2]{aidekonjaffuel2011}, if $(\gamma_n)_{n\in \N}$ is a sequence of positive numbers such that $\gamma_n \ll \sqrt{n}$ as $n\to \infty$, then we have, for all continuous bounded function $g \colon \R_+ \to \R$,
\begin{equation} \label{equation convergence vers la loi Rayleigh}
\Ec{ g \left( \frac{S_n + u}{\sigma \sqrt{n}} \right) 
	\1_{\underline{S}_n \geq -u}}
= 
\frac{\theta R(u)}{\sqrt{n}}
\left(
\int_0^\infty g(t) t \e^{-t^2/2} \diff t
+
\petito{1}
\right),
\end{equation}
uniformly in $u\in [0, \gamma_n]$.
From Lemmas 2.2 and 2.4 of A\"idékon and Shi \cite{aidekonshi2014}, we have the following inequalities, sometimes called ballot theorems: 
it exists $c_4 >0$ such that, for all $b > a \geq 0$, $u \geq 0$ and $n \geq 1$,
\begin{align} \label{eq ballot theorem}
\Pp{S_n \in [a-u, b-u], \underline{S}_n \geq -u} 
\leq
c_4 \frac{(u+1)(b+1)(b-a+1)}{n^{3/2}},
\end{align}
and, for $\lambda \in (0,1)$,
it exists $c_5 = c_5(\lambda) >0$ such that for all $b > a \geq 0$, $u \geq 0$, $v \in\R$ and $n \geq 1$, we have
\begin{align} \label{eq ballot theorem and barrier}
\Pp{S_n  \in [a+v,b+v], \underline{S}_{\lfloor \lambda n \rfloor} \geq -u, 
	\min_{\lfloor \lambda n \rfloor \leq j \leq n} S_j \geq v} 
\leq
c_5 \frac{(u+1)(b+1)(b-a+1)}{n^{3/2}},
\end{align}
where we added a second barrier between times $\lfloor \lambda n \rfloor$ and $n$.
From the previous results, it follows (see A\"idékon \cite[Lemma B.2]{aidekon2013}) that it exists $c_6 > 0$ such that, for all $a,u \geq 0$,
\begin{equation} \label{equation MA somme de probas}
\sum_{i\geq 0}
\Pp{\underline{S}_i \geq -u, S_i \leq a-u}
\leq
c_6 (1+a) (1+ \min(a,u)).
\end{equation}

Finally, we state a last result that is used for the proof of the peeling lemma in Subsection \ref{subsection peeling lemma} and is proved in Subsection \ref{subsection lower envelope}.
\begin{lem} \label{lemme trajectoire supérieure à i^alpha}
Let $(r_n)_{n\in \N}$ be an increasing sequence of positive numbers such that we have $\sum_{n\in\N} r_n n^{-3/2} < \infty$.
We set, for $\ell, i \in \llbracket 0, n \rrbracket$, $u, \mu \geq 0$ and $v \in \R$,
\begin{align*}
m_i^{(n,\ell)}
\coloneqq
\left\{
\begin{array}{ll}
- u + r_i - \mu & \text{if } 0 \leq i < \ell, \\
v + r_{n-i} - \mu & \text{if } \ell \leq i \leq n.
\end{array}
\right.
\end{align*}
For any $\varepsilon > 0$ and $\lambda \in (0,1/2)$, it exists $\mu = \mu(\varepsilon,\lambda) >0$ such that for all $b,u \geq 0$, $v \in \R$, $n\in \N$ and $\ell \in [\lambda n, (1-\lambda)n]$,
\begin{align*}
\Pp{\underline{S}_\ell \geq -u,
	\min_{\ell \leq j \leq n} S_j \geq v,
	S_n \in [b+v, b+v+1],
	\Exists i \in \llbracket 0,n \rrbracket:
	S_i \leq m_i^{(n,\ell)}}
\leq \frac{\varepsilon (1+u) (1+b)}{n^{3/2}}.
\end{align*}
\end{lem}
\begin{rem}
This kind of lemma is useful in the proof of peeling lemmas: see Lemma B.3 of \cite{aidekon2013}, Lemma 6.1 of \cite{madaule2015} and Lemma A.6 of \cite{shi2015}.
A slight difference here is that the terminal interval is $[b+v, b+v+1]$ instead of $[b,b+v]$ in the previous results.
Furthermore, in these papers, they take $r_i = i^\alpha$ with $\alpha \in (0,1/6)$. 
This would have been sufficient for the proof of our peeling lemma, but an anonymous referee asked us whether the result holds for any $\alpha \in (0,1/2)$ and this lemma answers in the affirmative. 
\end{rem}

\subsection{Convergence towards the Brownian meander}
\label{subsection convergence towards the meander}

We define the rescaled trajectory of the random walk until time $n$: for each $n \in \N^*$,
\begin{align*}
\bfS^{(n)}
\coloneqq
\left( 
	\frac{S_{\lfloor nt \rfloor}}{\sigma \sqrt{n}},
	t \in [0,1]
\right).
\end{align*}
We state the following convergence result for the trajectory $\bfS^{(n)}$ conditioned to stay nonnegative, with uniformity in the starting point of the random walk.
\begin{prop} \label{prop CV vers le meandre}
Let $(\gamma_n)_{n\in\N}$ be a sequence of positive numbers such that $\gamma_n \ll \sqrt{n}$ as $n\to\infty$.
Then, for all $F \in \cC_b(\cD([0,1]))$, we have
\begin{align*}
\Eci{u}{F(\bfS^{(n)})
	\1_{\underline{S}_n \geq 0}}
& =
\frac{\theta R(u)}{\sqrt{n}} 
\left( \Ec{F(\cM)} + \petito{1} \right),
\end{align*}
as $n\to\infty$, uniformly in $u \in [0,\gamma_n]$, where $\cM$ denotes the Brownian meander of length $1$.
\end{prop}
This invariance principle has been proved in the case $u = 0$ by Iglehart \cite{iglehart74}, Bolthausen \cite{bolthausen76} and Doney \cite{doney85}.
The case where $F$ is a function of the terminal value of the trajectory is already showed in \eqref{equation convergence vers la loi Rayleigh}.
We give a short proof of this generalization in Subsection \ref{subsection appendix convergence towards the bessel and the meander}, that relies on the invariance principle by Caravenna and Chaumont \cite{caravennachaumont2008}, for random walk conditioned to stay nonnegative for all time.

We present also a corollary of Proposition \ref{prop CV vers le meandre}, which holds under an additional assumption on the random walk $S$ that is equivalent to assumption \eqref{hypothese 4} by the many-to-one lemma. 
It will be used in Section \ref{section near-critical window in the weak disorder regime} and its proof is postponed to Subsection \ref{subsection appendix convergence towards the bessel and the meander}.
\begin{cor} \label{cor convergence to the meander with extension}
Assume that it exists $\delta_0 >0$ such that $\E[\e^{\delta_0 S_1}] < \infty$.
Then, for all $C,L>0$ and $F \in \cC_b(\cD([0,1]))$, we have
\begin{align*}
\Ec{\e^{C S_n / \sqrt{n}} F(\bfS^{(n)}) \1_{\underline{S}_n \geq -L}}
& =
\frac{\theta R(L)}{\sqrt{n}} 
\left( \Ec{\e^{C \sigma \cM(1)} F(\cM)} + \petito{1} \right),
\end{align*}
as $n\to\infty$, where $\cM$ denotes the Brownian meander of length $1$.
\end{cor}

\subsection{Convergence towards the Brownian excursion}
\label{subsection convergence towards the Brownian excursion}

In this subsection, our interest is the convergence of $\bfS^{(n)}$, conditioned to stay above two successive barriers and to end up in a small interval, towards the normalized Brownian excursion, with uniformity with respect to the barriers' positions and to the endpoint as long as they are much smaller than $\sqrt{n}$.
This result will be used repetitively in Section \ref{section near-critical window in the strong disorder regime} for the proof of part (ii) of Theorems \ref{theorem W} and \ref{theorem trajectory}.
At this scale for the endpoint $S_n$, the random walk behaves differently in the lattice and nonlattice cases, so they have to be distinguished.
Moreover, we need some new notation: let $(S_n^-)_{n\in\N}$ be a random walk such that, under $\P_a$, $S^-_0 = a$ a.\@s.\@ and $S^-_1 - S^-_0$ has the same law as $S_0 - S_1$. 
All objects referring to $S$ ($R$, $c_0$, $\theta$, \dots) have their analogue for $S^-$ denoted with a $-$ superscript ($R^-$, $c_0^-$, $\theta^-$, \dots). 

We can now state our result, which generalizes Lemma 2.6 of Chen, Madaule and Mallein \cite{cmm2015arxiv} and is proved in Subsection \ref{subsection appendix convergence towards the Brownian excursion}.
\begin{prop} \label{prop convergence to the excursion}
Let $(\gamma_n)_{n\in \N}$ be a sequence of positive numbers such that $\gamma_n \ll \sqrt{n}$ as $n\to \infty$ and $\mathfrak{e}$ be the normalized Brownian excursion.
\begin{enumerate}
\item If the law of $S_1$ is nonlattice, then for all $h>0$, $\lambda \in (0,1)$ and $F \in \cC_b(\cD([0,1]))$,
\begin{align*}
& \Ec{F(\bfS^{(n)})
	\1_{\underline{S}_{\lfloor \lambda n \rfloor} \geq -u, 
	\min_{\lfloor \lambda n \rfloor \leq i \leq n} S_i \geq v,
	S_n \in [v + b, v+b+h)}} \\
& =
\sqrt{\frac{\pi}{2}} \frac{\theta \theta^-}{\sigma}
\frac{R(u)}{n^{3/2}}
(\Ec{F(\mathfrak{e})} + \petito{1})
\int_b^{b+h} R^-(t) \diff t,
\end{align*}
as $n \to \infty$, uniformly in $b,u \in [0,\gamma_n]$ and $v \in [-\gamma_n, \gamma_n]$.
\item If the law of $S_1$ is $(h,a)$-lattice, then for all $\lambda \in (0,1)$ and $F \in \cC_b^u(\cD([0,1]))$,
\begin{align*}
\Ec{F(\bfS^{(n)})
	\1_{\underline{S}_{\lfloor \lambda n \rfloor} \geq -u, 
	\min_{\lfloor \lambda n \rfloor \leq i \leq n} S_i \geq v,
	S_n = v + b}} 
=
\sqrt{\frac{\pi}{2}} \frac{\theta \theta^-}{\sigma}
\frac{R(u)}{n^{3/2}}
(\Ec{F(\mathfrak{e})} + \petito{1})
h R^-(b),
\end{align*}
as $n \to \infty$, uniformly in $u \in [0,\gamma_n]$, $v \in [-\gamma_n, \gamma_n]$ and $b \in [0, \gamma_n] \cap (-v + an + h\Z)$.
\end{enumerate}
\end{prop}

\section{The critical window}
\label{section critical window}

We prove here cases (ii) and (iii) of Theorems \ref{theorem W} and \ref{theorem trajectory}, where $\beta_n = 1 \pm 1/\alpha_n$ with $\sqrt{n} / \alpha_n \to \gamma \in [0,\infty)$.
This proof is based on Theorem 1.2 of Madaule \cite{madaule2016}, recalled in \eqref{eq CV Madaule}.
\begin{proof}[Proof of part (ii) of Theorems \ref{theorem W} and \ref{theorem trajectory}]
Our aim is to show that, for all $F \in \cC_b(\cD([0,1]))$, we have
\begin{align}
\frac{1}{W_{n,1}}
\sum_{\abs{z} = n} 
\e^{-\beta_n V(z)} 
F \left( \bfV (z) \right)
\xrightarrow[n\to\infty]{}
\Ec{\e^{-\sigma \gamma \cM(1)} F(\cM)}, \label{oa}
\quad
\text{in }
\P^*\text{-probability}.
\end{align}
Then, using \eqref{eq CV aidekon shi} and the case $F \equiv 1$, part (ii) of Theorems \ref{theorem W} and \ref{theorem trajectory} follows.
Note also that it is sufficient to show \eqref{oa} for $F$ nonnegative.

We first deal with the case $\gamma \in (0,\infty)$.
We consider some nonnegative function $F \in \cC_b(\cD([0,1]))$ and $\varepsilon > 0$.
By dominated convergence, we have that $\E [\e^{-\sigma \gamma' \cM(1)} F(\cM)]$ tends towards $\E[\e^{-\sigma \gamma \cM(1)} F(\cM)]$ as $\gamma' \to \gamma$, so we can choose $0 < \gamma^- < \gamma < \gamma^+$ such that
\begin{align}
\begin{split} \label{ob}
\Ec{\e^{-\sigma \gamma \cM(1)} F(\cM)} - \frac{\varepsilon}{2}
& \leq
\Ec{\e^{-\sigma \gamma^\pm \cM(1)} F(\cM)}
\leq
\Ec{\e^{-\sigma \gamma \cM(1)} F(\cM)} + \frac{\varepsilon}{2}.
\end{split}
\end{align}
Since $\sqrt{n} / \alpha_n \to \gamma$ and, under $\P^*$, $\min_{\abs{z}=n} V(z) \to \infty$ a.\@s.\@, it exists $n_0 \in \N$ such that, for all $n \geq n_0$, we have $\sqrt{n} / \alpha_n \in [\gamma^-, \gamma^+]$ and also $\P^* (\min_{\abs{z}=n} V(z) < 0) \leq \varepsilon$.
Moreover, we set, for $x \in \cD([0,1])$, $G^+(x) \coloneqq F(x) (\e^{- \sigma \gamma^+ x_1} \wedge 1)$ and $G^-(x) \coloneqq F(x) (\e^{- \sigma \gamma^- x_1} \wedge 1)$: 
if $x_1 \geq 0$, then we have $G^+(x) \leq \e^{(1-\beta_n) x_1} F(x) \leq G^-(x)$.
Thus, we get, for all $n \geq n_0$,
\begin{align*}
\P^* \Biggl(
\frac{1}{W_{n,1}}
\sum_{\abs{z} = n} 
\e^{-\beta_n V(z)} 
F \left( \bfV (z) \right)
\notin
\left[
	\mu_{n,1} (G^+),
	\mu_{n,1} (G^-)
\right]
\Biggr)
\leq 
\P^* \left(\min_{\abs{z}=n} V(z) < 0 \right)
\leq
\varepsilon.
\end{align*}
Therefore, we have
\begin{align}
& \P^* \Biggl(
\Bigg\lvert
\frac{1}{W_{n,1}}
\sum_{\abs{z} = n} 
\e^{-\beta_n V(z)} 
F \left( \bfV (z) \right)
- 
\Ec{\e^{-\sigma \gamma \cM(1)} F(\cM)}
\Bigg\rvert
\leq \varepsilon
\Biggr) \nonumber \\
& \leq
\varepsilon
+
\P^* \left(
\mu_{n,1} (G^-)
\geq 
\Ec{\e^{-\sigma \gamma \cM(1)} F(\cM)} + \varepsilon
\right)
+
\P^* \left(
\mu_{n,1} (G^+)
\leq 
\Ec{\e^{-\sigma \gamma \cM(1)} F(\cM)} - \varepsilon
\right) \nonumber \\
& \leq
\varepsilon
+
\P^* \left(
\mu_{n,1} (G^-)
\geq 
\Ec{G^+(\cM)} + \frac{\varepsilon}{2}
\right)
+
\P^* \left(
\mu_{n,1} (G^+)
\leq 
\Ec{G^-(\cM)} - \frac{\varepsilon}{2}
\right), \label{oc}
\end{align}
using \eqref{ob}.
Then, applying \eqref{eq CV Madaule}, we get that both probabilities in \eqref{oc} tends to 0, because $G^-, G^+ \in \cC_b(\cD([0,1]))$ and it concludes the proof of \eqref{oa} in the case $\gamma \in (0,\infty)$.
Finally, for the case $\gamma = 0$, we proceed in the same way as for $\gamma \in (0,\infty)$, taking here $\gamma^+ > 0$ such that $\E [F(\cM)] - \frac{\varepsilon}{2} \leq \E [\e^{-\sigma \gamma^+ \cM(1)} F(\cM)]$, $G^+$ defined as before and $G^- \coloneqq F$.
Then, the same inequalities hold.
\end{proof}
\begin{rem} \label{rem assumption madaule}
In case (iii), we work under assumption \eqref{hypothese 4} and we will use Proposition 3.8 of Madaule \cite{madaule2016}, whereas Madaule works in \cite{madaule2016} under the stronger assumption \eqref{hypothese 5}.
But, for the proof of his Proposition 3.8, he only uses Assumption \eqref{hypothese 5} in the proof of his Lemma A.2, in order to have that $\Psi$ is finite on a left-neighbourhood of 1 and $\Psi(\beta) = \frac{\sigma^2}{2} (\beta-1)^2 + o((\beta-1)^2)$ as $\beta \uparrow 1$, and this holds also under our assumption \eqref{hypothese 4}.
\end{rem}
\begin{proof}[Proof of part (iii) of Theorems \ref{theorem W} and \ref{theorem trajectory}]
Applying Proposition 3.8 of Madaule \cite{madaule2016} (see Remark \ref{rem assumption madaule} above), we get that $\mu_{n,1}(G) \to \E[G(\cM)]$ in $\P^*$-probability with $G \colon x \in \cD([0,1]) \mapsto \e^{C x_1}$ for any $C > 0$, although $G$ is not bounded.
Combining this with \eqref{eq CV Madaule}, it is straightforward to extend this convergence to functions of the type $G \colon x \in \cD([0,1]) \mapsto \e^{C x_1} F(x)$ with $C > 0$ and $F \in \cC_b(\cD([0,1]))$.
Then, we prove that, for all $F \in \cC_b(\cD([0,1]))$ nonnegative,
\begin{align}
\frac{1}{W_{n,1}}
\sum_{\abs{z} = n} 
\e^{-\beta_n V(z)} 
F \left( \bfV (z) \right)
\xrightarrow[n\to\infty]{}
\Ec{\e^{-\sigma \gamma \cM(1)} F(\cM)}, \label{oe}
\quad
\text{in }
\P^*\text{-probability},
\end{align}
using the same method as for the proof of case (ii): we approach function $x \mapsto \e^{\sigma x_1 \sqrt{n}/ \alpha_n} F(x)$ from above and from below,
by considering here
$G^+(x) \coloneqq \e^{\sigma \gamma^+ x_1} F(x)$ 
and $G^-(x) \coloneqq \e^{\sigma \gamma^- x_1} F(x)$ when $\gamma \in (0,\infty)$
and the same function $G^+$ but with $G^- \coloneqq F$ when $\gamma = 0$.
Finally, part (iii) of Theorems \ref{theorem W} and \ref{theorem trajectory} follows from \eqref{oe}.
\end{proof}

\section{The near-critical window in the weak disorder regime}
\label{section near-critical window in the weak disorder regime}

In this section, we deal with the case where $\beta_n = 1 - 1/\alpha_n$ and $\alpha_n \ll \sqrt{n}$ and prove successively convergence of the rescaled partition function and then convergence of the trajectories.
We work here under assumption \eqref{hypothese 5} so $\Psi$ is analytic on an open interval $I$ containing 1.
Moreover, by Lemma 4.3 of Madaule \cite{madaule2016}, assumption \eqref{hypothese 5} implies that there exist $c_7 > 0$ and $\eta_0 < 1$ such that $\E[(\widetilde{W}_{\infty,\beta})^{1+\frac{\eta}{2}}] \leq c_7$ for any $0 < \eta < \eta_0$ and $\beta = 1 - \eta$.
Thus, for any $n \in \N$ and $0 < \eta < \eta_0$, we have, with $\beta = 1 - \eta$,
\begin{align} \label{eq controle L^p}
\Ec{(\widetilde{W}_{n,\beta})^{1+\frac{\eta}{2}}}
= \Ec{ \Ecsq{\widetilde{W}_{\infty,\beta}}{\sF_n}^{1+\frac{\eta}{2}} }
\leq \Ec{(\widetilde{W}_{\infty,\beta})^{1+\frac{\eta}{2}}}
\leq c_7,
\end{align}
by using Jensen's inequality.
We need to introduce a more general statement of the many-to-one lemma. 
For $\beta \in I$, we define another random walk $(S_{n,\beta})_{n\in\N}$ starting at $0$ under $\P$ and such that
\begin{align*}
\Ec{h(S_{1,\beta})}
=
\e^{-\Psi(\beta)}
\E \Biggl[ \sum_{\abs{z} = 1} h(V(z) + \Psi'(\beta)) \e^{-\beta V(z)} \Biggr].
\end{align*}
Then, $S_{1,\beta}$ is centred and has variance $\sigma_\beta^2 \coloneqq \E[S_{1,\beta}^2] = \Psi''(\beta) \in (0,\infty)$.
Moreover, we have the following analogue of the many-to-one lemma (see Shi \cite{shi2015}): for all $n \geq 1$ and all measurable function $g \colon \R^{n+1} \to \R_+$, we have
\begin{align} \label{eq many-to-one beta < 1}
\E \Biggl[ \sum_{\abs{z} = n} g(V(z_i), i \in \llbracket 0,n \rrbracket) \Biggr]
=
\e^{n\Psi(\beta)} 
\Ec{\e^{\beta(S_{n,\beta}-n\Psi'(\beta))} 
	g(S_{i,\beta}-i\Psi'(\beta), i \in \llbracket 0,n \rrbracket)}.
\end{align}
One can see $(S_{n,\beta})_{n \in \N}$ as a discrete Girsanov transform of $(S_n)_{n\in\N}$.

We first establish a preliminary lemma.
\begin{lem} \label{lem weak disorder regime}
Let $(\beta'_n)_{n\in\N} \in (0,1)^\N$ and $(k_n)_{n\in\N} \in \N^\N$ be sequences such that $\beta'_n \to 1$ and $k_n \geq 1/(1-\beta'_n)^2$ for any $n \in \N$.
Then, for any $L >0$, 
\[
\limsup_{n\to\infty} \frac{1}{1-\beta'_n} 
\Ec{\widetilde{W}_{k_n,\beta'_n} \1_{\inf_{x \in \T} V(x) \geq -L}}
\leq \theta R(L) \Ec{\e^{\sigma \cM(1)}}.
\]
\end{lem}
\begin{proof}
We first apply the many-to-one lemma to get that
\begin{align*}
\frac{1}{1-\beta'_n} 
\Ec{\widetilde{W}_{k_n,\beta'_n} \1_{\inf_{x \in \T} V(x) \geq -L}} 
& \leq
\frac{1}{1-\beta'_n} 
\e^{-k_n \Psi(\beta'_n)}
\Ec{\e^{(1-\beta'_n) S_{k_n}} \1_{\underline{S}_{k_n} \geq -L}}. 
\end{align*}
Then, we set $m_n \coloneqq \lfloor (1-\beta'_n)^{-2} \rfloor \leq k_n$ and, applying the Markov property at time $m_n$, we have
\begin{align} \label{oi}
\Ec{\e^{(1-\beta'_n) S_{k_n}} \1_{\underline{S}_{k_n} \geq -L}}
& \leq
\Ec{\e^{(1-\beta'_n) S_{m_n}} \1_{\underline{S}_{m_n} \geq -L}}
\Ec{\e^{(1-\beta'_n) S_{k_n-m_n}}}  \nonumber \\
& =
\frac{\theta R(L)}{\sqrt{m_n}} 
\left( \Ec{\e^{\sigma \cM(1)}} + \petito{1} \right)
\e^{(k_n-m_n) \Psi(\beta'_n)},
\end{align}
using Corollary \ref{cor convergence to the meander with extension} to bound the first expectation in the middle part of \eqref{oi} and the many-to-one lemma for the second expectation.
Using that $\e^{-m_n \Psi(\beta'_n)} \leq 1$, it proves the lemma.
\end{proof}
\begin{proof}[Proof of part (iv) of Theorem \ref{theorem W}]
We set $\xi_n \coloneqq \alpha_n \lvert \widetilde{W}_{\infty,\beta_n} - \widetilde{W}_{n,\beta_n} \rvert$ and want to show that $\xi_n \to 0$ in $\P^*$-probability.
It will prove part (iv) of Theorem \ref{theorem W}, since $\alpha_n \widetilde{W}_{\infty,\beta_n} \to 2 D_\infty$ in $\P^*$-probability by \eqref{eq Madaule equiv additive martingale}.
We first follow the proof of Lemma 4.2 of Madaule \cite{madaule2016}.
We set $p_n \coloneqq 1 + 1/2\alpha_n$ and $\xi'_n \coloneqq \E [ \xi_n^{p_n} | \sF_n]$.
For $\varepsilon > 0$, we have
\begin{align} \label{or}
\P^*(\xi_n \geq \varepsilon)
\leq \frac{1}{\P(S)} \Ec{\1_{\xi_n^{p_n} \geq \varepsilon^{p_n}}
	\1_{\xi'_n < \varepsilon^{1+p_n}}}
+ \P^*\left(\xi'_n \geq \varepsilon^{1+p_n} \right)
\leq \frac{\varepsilon}{\P(S)} + \P^*\left(\xi'_n \geq \varepsilon^{1+p_n} \right),
\end{align}
using that $\Ppsq{\xi_n^{p_n} \geq \varepsilon^{p_n}}{\sF_n} \leq \varepsilon^{-p_n} \E [ \xi_n^{p_n} | \sF_n] = \varepsilon^{-p_n}\xi'_n$.
By the branching property at time $n$, we have
\begin{align*}
\xi_n
= \alpha_n
\sum_{\abs{x} = n} \e^{-\beta_n V(x)- n \Psi(\beta_n)} 
\left( \widetilde{W}_{\infty,\beta_n}^{(x)} - 1 \right),
\end{align*}
where, conditionally on $\sF_n$, the $\widetilde{W}_{\infty,\beta_n}^{(x)}$ for $\abs{x} = n$ are independent variables with the same law as $\widetilde{W}_{\infty,\beta_n}$.
Then, using that for any sequence $(X_i)_{i\in\N}$ of independent centred variables and any $\gamma \in [1,2]$ we have $\E[\lvert\sum X_i \rvert^\gamma] \leq 2 \sum \E[\lvert X_i \rvert^\gamma]$ (see \cite{vonbahresseen65}), we get
\begin{align}
\xi'_n =
\Ecsq{\xi_n^{p_n}}{\sF_n}
& \leq 
2\alpha_n^{p_n}
\sum_{\abs{x} = n} \e^{-p_n \beta_n V(x)- n p_n\Psi(\beta_n)} 
\Ec{\abs{\widetilde{W}_{\infty,\beta_n} - 1}^{p_n}} \nonumber \\
& \leq
2\alpha_n^{p_n}
\widetilde{W}_{n,p_n \beta_n}
\e^{n[\Psi(p_n\beta_n) - p_n\Psi(\beta_n) ]}
2^{p_n}(c_7+1), \label{of}
\end{align}
by using \eqref{eq controle L^p} for $n$ large enough such that $1-\eta_0 < 1-\beta_n < 1$.
Now, we choose $L>0$ such that $\P^* (\inf_{x \in \T} V(x) < -L) \leq \varepsilon$ and, by \eqref{of} and Markov's inequality, we get
\begin{align*}
\P^*\left(\xi'_n \geq \varepsilon^{1+p_n} \right)
& \leq
\varepsilon
+
\varepsilon^{-(1+p_n)}
\frac{c_8}{\P(S)}
(2\alpha_n)^{p_n}
\e^{n[\Psi(p_n\beta_n) - p_n\Psi(\beta_n) ]}
\Ec{\widetilde{W}_{n,p_n \beta_n} \1_{\inf_{x \in \T} V(x) \geq -L}}.
\end{align*}
As $n \to \infty$, we have $(2\alpha_n)^{p_n} \sim 2\alpha_n$, $1-p_n \beta_n \sim 1/2\alpha_n$ and, by a Taylor expansion, $n[\Psi(p_n\beta_n) - p_n\Psi(\beta_n)] \sim -3 \sigma^2 n/ 8 \alpha_n^2 \to -\infty$.
Thus, applying Lemma \ref{lem weak disorder regime} with $k_n = n$ and $\beta'_n = p_n\beta_n$, we showed that 
$\limsup_{n\to\infty} \P^*(\xi'_n \geq \varepsilon^{1+p_n}) \leq \varepsilon$.
Coming back to \eqref{or}, it concludes the proof.
\end{proof}

\begin{proof}[Proof of part (iv) of Theorem \ref{theorem trajectory}]
By Lemma \ref{lem weak convergence uc to c}, we can reduce the proof to the case $F \in \cC_b^u(\cD([0,1]))$.
Moreover, by considering $F - \Ec{F(B)}$ instead of $F$, we can assume that $\Ec{F(B)} =0$.
By Theorem \ref{theorem W}, we have $\alpha_n \e^{-n \Psi(\beta_n)} W_{n,\beta_n} \to 2D_\infty$ in $\P^*$-probability with $D_\infty > 0$ $\P^*$-a.\@s.\@, so it is sufficient to prove that
\begin{align} \label{oz}
U_n(F)
\coloneqq
\alpha_n \e^{-n \Psi(\beta_n)}
\sum_{\abs{z} = n} 
\e^{-\beta_n V(z)} 
	F \left( \widetilde{\bfV}^{(n)}(z) \right)
\xrightarrow[n\to\infty]{} 2 D_\infty \Ec{F(B)} = 0,
\end{align}
in $\P^*$-probability, where $\widetilde{\bfV}^{(n)}_t(z) \coloneqq [V(z_{\lfloor tn \rfloor}) + t n\Psi'(\beta_n)]/\sigma\sqrt{n}$ for $t \in [0,1]$.
For some $C>0$, we set $k_n \coloneqq \lfloor C \alpha_n \rfloor^2$
and, for each $x \in \cD([0,1])$, $F_n(x) \coloneqq F( x_{((n-k_n)t + k_n)/n} - x_{k_n / n}, t\in [0,1])$.
Let $\varepsilon >0$.
In order to prove \eqref{oz}, it is sufficient to prove that we can choose $C$ such that
\begin{align}
\limsup_{n\to\infty} \P^*(\abs{U_n(F)-U_n(F_n)} \geq \varepsilon) 
& \leq 2 \varepsilon, \label{oy} \\
\limsup_{n\to\infty} \P^*(\abs{U_n(F_n)-\Ecsq{U_n(F_n)}{\sF_{k_n}}}\geq\varepsilon) 
& \leq (2 + \P(S)^{-1}) \varepsilon, \label{ox} \\
\limsup_{n\to\infty} \P^*(\abs{\Ecsq{U_n(F_n)}{\sF_{k_n}}} \geq \varepsilon) 
& \leq \varepsilon. \label{ow}
\end{align}
The assumption that $\Ec{F(B)}=0$ is only needed for \eqref{ow}.
For the sequel, we fix some $L > 0$ such that $\P^* (\inf_{x \in \T} V(x) < -L) \leq \varepsilon$.

We first prove \eqref{ox}. 
We set $\zeta_n \coloneqq U_n(F_n)- \Ecsq{U_n(F_n)}{\sF_{k_n}}$ and proceed in a similar way as for the proof of part (iv) of Theorem \ref{theorem W}, by setting $p_n \coloneqq 1 + 1/2\alpha_n$ and $\zeta'_n \coloneqq \E [ \abs{\zeta_n}^{p_n} | \sF_n]$.
By \eqref{or}, we have 
$\P^*(\abs{\zeta_n} \geq \varepsilon)
\leq \varepsilon \P(S)^{-1} + \P^*(\zeta'_n \geq \varepsilon^{1+p_n})$.
By the branching property at time $k_n$, we have
\begin{align} \label{os}
\zeta_n
= \alpha_n
\sum_{\abs{x} = k_n} \e^{-\beta_n V(x)- k_n \Psi(\beta_n)} 
\left( \Upsilon_n^{(x)} - \Ec{\Upsilon_n} \right),
\end{align}
where, conditionally on $\sF_{k_n}$, the $\Upsilon_n^{(x)}$ for $\abs{x} = k_n$ are independent variables with the same law as $\Upsilon_n$ defined by
\begin{align} \label{ot}
\Upsilon_n
& \coloneqq 
\sum_{\abs{z} = n-k_n} \e^{-\beta_n V(z) - (n-k_n) \Psi(\beta_n)} 
F \left( \frac{V(z_{\lfloor t(n-k_n) \rfloor})+ t(n-k_n) \Psi'(\beta_n)}{\sigma\sqrt{n}}, 
	t\in [0,1] \right).
\end{align}
Since the $\Upsilon_n^{(x)} - \Ec{\Upsilon_n}$ are also centred, we get, in the same way as for \eqref{of},
\begin{align}
\zeta'_n 
= \Ecsq{\abs{\zeta_n}^{p_n}}{\sF_{k_n}}
& \leq 2 \alpha_n^{p_n}
\sum_{\abs{x} = k_n} \e^{-p_n \beta_n V(x)- k_n p_n \Psi(\beta_n)} 
\Ec{\abs{\Upsilon_n - \Ec{\Upsilon_n}}^{p_n}} \nonumber \\
& \leq 4 c_7 (2\norme{F} \alpha_n)^{p_n}
\widetilde{W}_{k_n,p_n \beta_n} \e^{k_n (\Psi(p_n \beta_n) - p_n \Psi(\beta_n))}, \label{oq}
\end{align}
by using the following bound
\begin{align*}
\Ec{\abs{\Upsilon_n - \Ec{\Upsilon_n}}^{p_n}}
\leq 2^{p_n+1} \Ec{\abs{\Upsilon_n}^{p_n}}
\leq 2^{p_n+1} \norme{F}^{p_n} \Ec{\abs{\widetilde{W}_{n-k_n,\beta_n}}^{p_n}}
\leq 2^{p_n+1} \norme{F}^{p_n} c_7,
\end{align*}
where we used \eqref{eq controle L^p} for $n$ large enough such that $1-\eta_0 < 1-\beta_n < 1$.
Now, using that $\P^* (\inf_{x \in \T} V(x) < -L) \leq \varepsilon$ and Markov's inequality, we get
\begin{align*}
\P^*\left(\zeta'_n \geq \varepsilon^{1+p_n} \right)
& \leq
\varepsilon
+
\varepsilon^{-(1+p_n)}
\frac{c_9}{\P(S)}
(2\norme{F}\alpha_n)^{p_n}
\e^{k_n[\Psi(p_n\beta_n) - p_n\Psi(\beta_n) ]}
\Ec{\widetilde{W}_{k_n,p_n \beta_n} \1_{\inf_{x \in \T} V(x) \geq -L}}
\end{align*}
so, using that, as $n \to \infty$, we have $(2\norme{F}\alpha_n/\varepsilon)^{p_n} \sim 2\norme{F}\alpha_n/\varepsilon$, $1-p_n \beta_n \sim 1/2\alpha_n$ and $k_n[\Psi(p_n\beta_n) - p_n\Psi(\beta_n)] \sim -3 \sigma^2 k_n/ 8 \alpha_n^2 \to -3 \sigma^2 C^2/ 8$, and applying Lemma \ref{lem weak disorder regime} with $\beta'_n = p_n\beta_n$, we get
\begin{align*}
\limsup_{n\to\infty} \P^*\left(\zeta'_n \geq \varepsilon^{1+p_n} \right)
& \leq
\varepsilon
+
c_{10}
\frac{R(L)}{\varepsilon^2}
\e^{-3 \sigma^2 C^2/ 8}
\leq 2 \varepsilon,
\end{align*}
by choosing $C$ large enough. 
This proves \eqref{ox}.
The constant $C$ is now fixed.

We now prove \eqref{oy}. Using that $\P^* (\inf_{x \in \T} V(x) < -L) \leq \varepsilon$ and the Markov inequality, we get, with $c_{11} \coloneqq 1/\P(S)$,
\begin{align} \label{ov}
\P^*(\abs{U_n(F)-U_n(F_n)} \geq \varepsilon) 
\leq \varepsilon
+
\frac{c_{11} \alpha_n}{\varepsilon \e^{n \Psi(\beta_n)}}
\E \Biggl[
\sum_{\abs{z} = n} 
\e^{-\beta_n V(z)} 
\1_{\underline{V}(z) \geq -L}
\abs{F-F_n} \left( \widetilde{\bfV}^{(n)}(z) \right)
\Biggr].
\end{align}
For $t\in [0,1]$, we define $\widetilde{\bfS}^{(n)}_t \coloneqq [S_{\lfloor tn \rfloor} + t n\Psi'(\beta_n)]/ \sigma\sqrt{n}$.
Then, using the many-to-one lemma and the triangle inequality, we get that, for any $M >0$, the expectation in \eqref{ov} is smaller than
\begin{align} \label{ou}
\Ec{
\e^{S_n / \alpha_n} 
\1_{\underline{S}_n \geq -L, \max_{0\leq k\leq k_n} S_k \leq M \alpha_n}
\abs{F-F_n} (\widetilde{\bfS}^{(n)})}
+
2 \norme{F}
\Ec{
\e^{S_n / \alpha_n} 
\1_{\underline{S}_n \geq -L, \max_{0\leq k\leq k_n} S_k > M \alpha_n}
}.
\end{align}
Note that, using \eqref{vl} with here $\kappa_n = k_n / n$, we have, for any $x \in \cD([0,1])$, 
\begin{align*}
\abs{F-F_n}(x)
\leq 
\omega_F 
\left(
\frac{k_n}{n}
\vee
3 \max_{[0,k_n/n]} \abs{x}
\right)
\end{align*}
and, thus, the first term in \eqref{ou} is smaller than
\begin{align*} 
\omega_F 
\left(
\frac{\lfloor C \alpha_n \rfloor^2}{n}
\vee
\frac{3M \alpha_n}{\sigma \sqrt{n}}
\right)
\Ec{
\e^{S_n / \alpha_n} 
\1_{\underline{S}_n \geq -L}}
=
\petito{\frac{\e^{n \Psi(\beta_n)}}{\alpha_n}},
\end{align*}
using \eqref{oi} to bound the expectation and recalling that $\alpha_n \ll \sqrt{n}$.
On the other hand, using Markov property at time $k_n$, the second term in \eqref{ou} is smaller than
\begin{align*}
& 2 \norme{F} 
\Ec{\e^{S_{k_n} / \alpha_n} \1_{\underline{S}_{k_n} \geq -L, \max_{0\leq k\leq k_n} S_k > M \alpha_n}}
\Ec{\e^{S_{n-k_n} / \alpha_n}}  \\
& = 2 \norme{F}
\frac{\theta R(L)}{\sqrt{k_n}} 
\left( 
\Ec{\e^{C \sigma \cM(1)} 
	\1_{\max \cM > M / (C\sigma)}} 
+
\petito{1} 
\right)
\e^{(n-k_n) \Psi(\beta_n)}
\end{align*}
applying Corollary \ref{cor convergence to the meander with extension} and recalling that $\sqrt{k_n} = \lfloor C \alpha_n \rfloor$.
Coming back to \eqref{ov}, we finally get
\begin{align} 
\limsup_{n\to\infty} \P^*(\abs{U_n(F)-U_n(F_n)} \geq \varepsilon) 
\leq \varepsilon
+
c_{12} \frac{R(L)}{\varepsilon C}
\Ec{\e^{C \sigma \cM(1)} 
	\1_{\max \cM > M / (C\sigma)}} 
\leq 2 \varepsilon,
\end{align}
by choosing $M$ large enough ($L$ and $C$ being fixed).
It proves \eqref{oy}.

Finally, we prove \eqref{ow}. 
Using the branching property in the same way as for \eqref{os}, we have $\Ecsq{U_n(F_n)}{\sF_{k_n}} = \alpha_n\widetilde{W}_{k_n,\beta_n} \Ec{\Upsilon_n}$,  where $\Upsilon_n$ is defined in \eqref{ot}.
By \eqref{eq many-to-one beta < 1}, we get
\begin{align*}
\Ec{\Upsilon_n}
& = \Ec{F \left( \frac{S_{\lfloor t(n-k_n) \rfloor,\beta_n}
	+ (t(n-k_n)-\lfloor t(n-k_n) \rfloor) \Psi'(\beta_n)}{\sigma\sqrt{n}}, 
	t\in [0,1] \right)} \\
& = \Ec{F \left( u_n \bfS^{(n-k_n,\beta_n)} + v_n \right)},
\end{align*}
where we introduced $\bfS^{(n,\beta)} \coloneqq (S_{\lfloor tn \rfloor,\beta} / \sigma_{\beta} \sqrt{n})_{t \in [0,1]}$ and where 
$(u_n)_{n \in \N} \in (\R_+^*)^\N$ and $(v_n)_{n \in \N} \in \cD ([0,1])^\N$ satisfy $u_n \to 1$ and $\lVert v_n \rVert_\infty \to 0$, using that $k_n \ll n$, $\Psi'(\beta_n) \to 0$ and $\sigma_{\beta_n} \to \sigma$.
Now, note that $\bfS^{(n-k_n,\beta_n)} \to B$ in law.
This is not a direct consequence of Donsker's theorem because here the law of the random walk changes for each $n$.
However, we can apply a strong invariance principle like Equation (11) of Sakhanenko \cite{sakhanenko2006}: it exists $c_{13} > 0$ such that, for any $n \geq 1$ and $\beta \in I$, it exists a Brownian motion $B^{(n,\beta)}$ such that 
\[
\Pp{\sup_{t \in [0,1]}  \abs{\bfS_t^{(n,\beta)} - B^{(n,\beta)}_t}
	\geq c_{13} \frac{n^{2/5}}{\sigma_\beta \sqrt{n}}}
\leq \frac{n \Ec{\abs{S_{1,\beta}}^3}}{(n^{2/5})^3}.
\]
Since $\sigma_{\beta_n} \to \sigma$ and $\E[\abs{S_{1,\beta_n}}^3] \to \E[\abs{S_1}^3] < \infty$ as $n\to\infty$, this proves that $\bfS^{(n-k_n,\beta_n)} \to B$ in law.
Then, applying Lemma \ref{lemma weak convergence approximation F_n F}, we get that $\Ec{\Upsilon_n} \to \Ec{F(B)} = 0$ as $n \to \infty$.
On the other hand, using that $\P^* (\inf_{x \in \T} V(x) < -L) \leq \varepsilon$ and the Markov inequality, we get
\begin{align*}
\P^*(\abs{\Ecsq{U_n(F_n)}{\sF_{k_n}}} \geq \varepsilon) 
\leq \varepsilon
+ \frac{1}{\varepsilon} 
\alpha_n \Ec{\widetilde{W}_{k_n,\beta_n} \1_{\inf_{x \in \T} V(x) \geq -L}} 
\abs{\Ec{\Upsilon_n}},
\end{align*}
and it proves \eqref{ow} by using Lemma \ref{lem weak disorder regime} and that $\Ec{\Upsilon_n} \to 0$ as $n \to \infty$.
\end{proof}

\section{The near-critical window in the strong disorder regime}
\label{section near-critical window in the strong disorder regime}

In this section, we prove case (i) of Theorems \ref{theorem W} and \ref{theorem trajectory}  and of Corollary \ref{cor typical energy}, where $\beta_n = 1 + 1/\alpha_n$ with $\alpha_n \ll \sqrt{n}$. 
This case constitutes the main part of this paper.

\subsection{Change of probabilities}
\label{subsection change of probabilities}

We introduce a first barrier by setting, for $L >0$ and $F \in \cC_b(\cD([0,1]))$,
\begin{align*}
W_{n,\beta_n} (F)
& \coloneqq
\sum_{\abs{z} = n} 
	\e^{-\beta_n V(z)} 
	F(\bfV (z)), \\
W^{(L)}_{n,\beta_n} (F)
& \coloneqq
\sum_{\abs{z} = n} 
	\e^{-\beta_n V(z)} 
	F(\bfV (z))
	\1_{\underline{V}(z) \geq -L},
\end{align*}
so that for $L$ large this two variables are equal with high probability.
Moreover, recall that, for $u \geq 0$, $R_L(u) = R(L+u)$ and
\begin{align*}
D_n^{(L)} = \sum_{\abs{z}=n} \e^{-V(z)} R_L(V(z)) \1_{\underline{V}(z) \geq -L}
\end{align*}
and that, using the martingale $(D_n^{(L)})_{n \in\N}$, we defined the modified probability measure $\Q^{(L)}$. 
We will work under this measure to study the asymptotic behaviour of $W_{n,\beta_n} (F)$.
\begin{prop} \label{proposition convergence under Q}
For all $L>0$ and $F \in \cC^u_b(\cD([0,1]))$, we have the convergence 
\[
\frac{n^{3\beta_n/2}}{\alpha_n^2} \frac{W^{(L)}_{n,\beta_n} (F)}{D_n^{(L)}}
\xrightarrow[n\to\infty]{}
\frac{\theta}{\sigma^2}
\Ec{F(\mathfrak{e})},
\quad
\text{ in }
\Q^{(L)} \text{-probability}.
\]
\end{prop}
This proposition will be proved in the following subsections, using a second moment argument similar to the one used by A\"idekon and Shi \cite{aidekonshi2014}.
\begin{proof}[Proof of part (i) of Theorems \ref{theorem W} and \ref{theorem trajectory}]
We consider $F \in \cC^u_b(\cD([0,1]))$ and we are going to show here that 
\begin{align} \label{aa}
\frac{n^{3\beta_n/2}}{\alpha_n^2} W_{n,\beta_n} (F)
\xrightarrow[n\to\infty]{}
\frac{c_0 \theta}{\sigma^2} \Ec{F(\mathfrak{e})} D_\infty,
\quad \text{in } \P^*\text{-probability}.
\end{align}
Using \eqref{equation lien theta c_0}, it proves part (i) of Theorem \ref{theorem W} by taking $F \equiv 1$.
Moreover, noting that $\mu_{n,\beta_n} (F) = W_{n,\beta_n} (F) / W_{n,\beta_n} (1)$ and $D_\infty > 0$ $\P^*$-a.\@s.\@, it proves part (i) of Theorem \ref{theorem trajectory} in the case $F \in \cC^u_b(\cD([0,1]))$. 
The general case follows by Lemma \ref{lem weak convergence uc to c}.
Thus, it is now sufficient to prove \eqref{aa} and we can for this purpose assume that $F$ is nonnegative.
Let $\varepsilon, \eta >0$.
We first fix $L > 0$ such that $\P^* (\inf_{x \in \T} V(x) < -L) \leq \eta$.
Combining that $\min_{\abs{x}=n} V(x) \to \infty$ $\P^*$-a.\@s.\@ as $n \to \infty$ and that $R_L(u) \sim c_0 u$ as $u \to \infty$, we get that, on the event $\{ \inf_{x \in \T} V(x) \geq -L \}$, $\lim_{n\to\infty} D_n^{(L)} = c_0 D_\infty > 0$ $\P^*$-a.\@s.
Thus, considering the event
\begin{align*}
\Omega_0 
&\coloneqq 
S \cap \left\{ \inf_{x \in \T} V(x) \geq -L \right\} 
\cap \{\forall n \geq n_0, 0 < c_0(1-\varepsilon)D_\infty \leq D_n^{(L)} \leq c_0(1+\varepsilon)D_\infty\},
\end{align*}
we can fix $n_0 \in \N$ such that $\P^* (\Omega_0) \geq 1 - 2 \eta$.
We now introduce the event
\begin{align*}
E_n
& \coloneqq
\left\{
\frac{n^{3 \beta_n/2}}{\alpha_n^2}
W_{n,\beta_n}(F)
\notin
\left[
\frac{c_0 \theta}{\sigma^2} D_\infty
(\Ec{F(\mathfrak{e})} - \varepsilon)
(1-\varepsilon),
\frac{c_0 \theta}{\sigma^2} D_\infty
(\Ec{F(\mathfrak{e})} +\varepsilon)
(1+\varepsilon)
\right]
\right\}.
\end{align*}
Then, using that on $\Omega_0$ we have $W^{(L)}_{n,\beta_n}(F) = W_{n,\beta_n}(F)$, we get, for $n \geq n_0$,
\begin{align*}
\Ec{D_n^{(L)} \1_{E_n \cap \Omega_0}}
=
\Q^{(L)}
\left(
E_n \cap \Omega_0
\right)
\leq
\Q^{(L)}
\left(
\abs{
	\frac{n^{3 \beta_n/2}}{\alpha_n^2}
	\frac{W^{(L)}_{n,\beta_n}(F)}{D^{(L)}_n}
	- 
	\frac{\theta}{\sigma^2} 
	\Ec{F(\mathfrak{e})}
	}
>
\frac{\theta \varepsilon}{\sigma^2}
\right)
\xrightarrow[n\to\infty]{}
0,
\end{align*}
applying Proposition \ref{proposition convergence under Q}.
Since $D_n^{(L)} \1_{E_n \cap \Omega_0} \geq 0$, it follows that 
$D_n^{(L)} \1_{E_n \cap \Omega_0} \to 0$
in $\P$-probability.
Using again that, on $\Omega_0$, $\lim_{n\to\infty} D_n^{(L)} = c_0 D_\infty > 0$ $\P^*$-a.\@s.\@, we get that $\P^*(E_n \cap \Omega_0) \to 0$.
Recalling that $\P^* (\Omega_0) \geq 1 - 2\eta$, we showed that $\limsup_{n\to\infty} \P^*\left(E_n \right) \leq 2\eta$ and, therefore, it proves \eqref{aa}.
\end{proof}

\subsection{Proof of Proposition \ref{proposition convergence under Q} and of part (i) of Corollary \ref{cor typical energy}}
\label{subsection proof of proposition 4.1 and of part (i) of Theorem 1.3}

The aim of this section is to break down the proof of Proposition \ref{proposition convergence under Q} into the proof of several lemmas.
Our goal is to use a second moment argument, 
but the first moment of $W^{(L)}_{n,\beta_n}(F)/D_n^{(L)}$ under $\Q^{(L)}$ does not have the right order and the second moment is not even necessarily finite.
Thus, we first need to come down to another random variable $Y_n'(F)$ that is close to $W^{(L)}_{n,\beta_n}(F)$ with high probability and that has first and second moments of the right order, by eliminating some rare particles with a too strong weight in the expectations.
The lemmas stated in this subsection will also allow us to prove part (i) of Corollary \ref{cor typical energy}.
Until the end of the paper, $L$ is a fixed positive constant.
We first add a second barrier between times $\lfloor n/2 \rfloor$ and $n$ at position $(3/2) \log n -K$, by setting, for $K >0$,
\begin{align*}
W^{(L,K)}_{n,\beta_n}(F)
& \coloneqq
\sum_{\abs{z} = n} 
	\e^{-\beta_n V(z)} 
	F(\bfV (z))
	\1_{\underline{V}(z) \geq -L, 
		\min_{\lfloor n/2 \rfloor \leq j \leq n} V(z_j) 
			\geq \frac{3}{2} \log n - K}.
\end{align*}
This first lemma shows that $W^{(L)}_{n,\beta_n}(F)$ and $W^{(L,K)}_{n,\beta_n}(F)$ are close with high probability and will be proved in Subsection \ref{subsection addition of the second barrier}.
Note that this step is superfluous when $\log n \ll \alpha_n \ll \sqrt{n}$ : since the particles contributing to $W_{n,\beta_n}$ are of order $\alpha_n$, it is approximately the same difficulty for them to stay above the two barriers than only above the first (see Proposition \ref{prop convergence to the excursion}) and, thus, $W^{(L)}_{n,\beta_n}(F)$ and $W^{(L,K)}_{n,\beta_n}(F)$ have asymptotically the same first moment in that case.
\begin{lem} \label{lemma addition of the barrier}
For all $L>0$ and $\varepsilon, \eta > 0$, it exists $K>0$ such that, for each $F \in \cC_b(\cD([0,1]))$,
\begin{align*}
\limsup_{n\to\infty} 
\Q^{(L)}
\left(
\frac{n^{3\beta_n/2}}{\alpha_n^2} 
\frac{\abs{W^{(L)}_{n,\beta_n}(F) - W^{(L,K)}_{n,\beta_n}(F)}}{D_n^{(L)}}
>
\varepsilon \norme{F}
\right)
\leq 
\eta.
\end{align*}
\end{lem}
We now consider a fixed $K >0$.
We will see in Lemma \ref{lemma first moment} that $W^{(L,K)}_{n,\beta_n} (F)$ has the right first moment, but we still need to remove other particles for the second moment.
Let $(\alpha^+_n)_{n\in\N}$ and $(\alpha^-_n)_{n\in\N}$ be sequences of positive real numbers and $(k_n)_{n\in\N}$ be a sequence of integers such that
\[
1 \ll \alpha^-_n \ll \alpha_n \ll \alpha^+_n \ll \sqrt{n}, 
\]
\[
(\log n)^6 \ll k_n \ll \sqrt{n},
\]
when $n\to \infty$.
We add some controls on the trajectory of the particle's lineage, by considering
\[
Y_n (F)
\coloneqq
\sum_{\abs{z} = n} 
\e^{-\beta_n V(z)} 
F(\bfV (z))
\1_{z \in A_n},
\]
where we set 
$A_n \coloneqq 
\{ \abs{z} = n : 
\Forall j \in \llbracket 0, n \rrbracket, V(z_j) \in I_{n,j} \}$ 
(see Figure \ref{figure A_n}),
with
\begin{align*}
I_{n,j} 
\coloneqq
\left\{
\begin{array}{ll}
\lbrack -L,\infty )
& \text{if } 0 \leq j < \lfloor n/2 \rfloor \text{ and } j \neq k_n, \\
\lbrack k_n^{1/3}, k_n]
& \text{if } j = k_n, \\
\lbrack (3/2) \log n, \infty)
& \text{if } \lfloor n/2 \rfloor \leq j \leq n, \\
\lbrack (3/2) \log n + \alpha_n^-, (3/2) \log n + \alpha^+_n ]
& \text{if } j=n.
\end{array}
\right.
\end{align*}
Note that the second barrier is here simply at $(3/2) \log n$: indeed, Lemma \ref{lemma first moment} shows that it does not change the first moment (and we could even have taken $(3/2) \log n + C$ with any $C>0$).
\begin{figure}[ht]
\centering
\begin{tikzpicture}
\fill[color = gray!20] (0,-0.6) -- (5,-0.6) -- (5,1.6) -- (10,1.6) -- (10,-1.2) -- (0,-1.2) -- cycle;
\draw[->,>=latex] (0,0) -- (10.8,0);
\draw[->,>=latex] (0,-1.2) -- (0,5) node[left]{$V$};
\draw (0,0) node[left]{$0$};
\draw (2.5,0) node[below]{$k_n$};
\draw (5,0) node[below right]{$\lfloor n/2 \rfloor$};
\draw (10,0.07) -- (10,-0.07); \draw (10,0) node[below right]{$n$};
\draw 
(0,0) -- (0.1,0.2) -- (0.2,-0.1) -- (0.3,-0.2) -- (0.4, -0.15) -- (0.5,-0.5) 
-- (0.6,-0.35) -- (0.7,-0.4) -- (0.8,-0.2) -- (0.9,-0.15) -- (1,0.3) 
-- (1.1,0.25) -- (1.2,0.5) -- (1.3,0.95) -- (1.4,1.15) -- (1.5,1.05)
-- (1.6,1.1) -- (1.7,1.4) -- (1.8,1.45) -- (1.9,1.8) -- (2,1.75)
-- (2.1,2.05) -- (2.2,2.15) -- (2.3,2.4) -- (2.4,2.35) 
-- (2.5,2.65) 
-- (2.6,2.65) -- (2.7,2.8) -- (2.8,2.7) -- (2.9,2.65) -- (3,3)
-- (3.1,3.2) -- (3.2,3) -- (3.3,2.95) -- (3.4,3.1) -- (3.5,3)
-- (3.6,3.05) -- (3.7,3.25) -- (3.8,3.18) -- (3.9,3.5) -- (4,3.6)
-- (4.1, 3.58) -- (4.2,3.7) -- (4.3,4) -- (4.4,3.9) -- (4.5,3.85)
-- (4.6,3.7) -- (4.7,4) -- (4.8,4.2) -- (4.9,4.15) -- (5,4.3)
-- (5.1,4.35) -- (5.2,4.6) -- (5.3,4.57) -- (5.4,4.65) -- (5.5,4.5)
-- (5.6,4.4) -- (5.7,4.47) -- (5.8,4.6) -- (5.9,4.9) -- (6,4.85)
-- (6.1,4.55) -- (6.2,4.6) -- (6.3,4.45) -- (6.4,4.3) -- (6.5,4.33)
-- (6.6,4.1) -- (6.7,4.3) -- (6.8,4) -- (6.9,4.05) -- (7,4.12)
-- (7.1,3.9) -- (7.2,3.85) -- (7.3,4) -- (7.4,4.15) -- (7.5,4.05)
-- (7.6,3.8) -- (7.7,3.7) -- (7.8,3.75) -- (7.9,3.45) -- (8,3.5)
-- (8.1,3.1) -- (8.2,3.15) -- (8.3,3.25) -- (8.4,3.55) -- (8.5,3.4)
-- (8.6,2.95) -- (8.7,3.05) -- (8.8,3) -- (8.9,2.75) -- (9,2.7)
-- (9.1,2.5) -- (9.2,2.65) -- (9.3,2.7) -- (9.4,2.6) -- (9.5,2.35)
-- (9.6,2.3) -- (9.7,1.9) -- (9.8,2) -- (9.9,2.25)
-- (10,2.2); 
\draw (0,-0.6) -- (5,-0.6) -- (5,1.6) -- (10,1.6);
\draw[dashed] (5,1.6) -- (0,1.6) node[left]{$\frac{3}{2} \log n$};
\draw (0,-0.6) node[left]{$-L$};
\draw[dashed] (2.5,2.2) -- (0,2.2) node[left]{$k_n^{1/3}$};
\draw[dashed] (2.5,3.2) -- (0,3.2) node[left]{$k_n$};
\draw[dashed] (2.5,2.2) -- (2.5,0);
\draw[very thick] (2.5,3.2) -- (2.5,2.2);
\draw[very thick] (10,2.6) -- (10,2);
\draw (10-0.07,2.6) -- (10+0.07,2.6) node[right]{$\frac{3}{2} \log n + \alpha_n^+$};
\draw (10-0.07,2) -- (10+0.07,2) node[right]{$\frac{3}{2} \log n + \alpha_n^-$};
\end{tikzpicture}
\caption{Representation of the trajectory of a particle in $A_n$. 
It has to stay above the gray area and to pass through both thick segments.}
\label{figure A_n}
\end{figure}
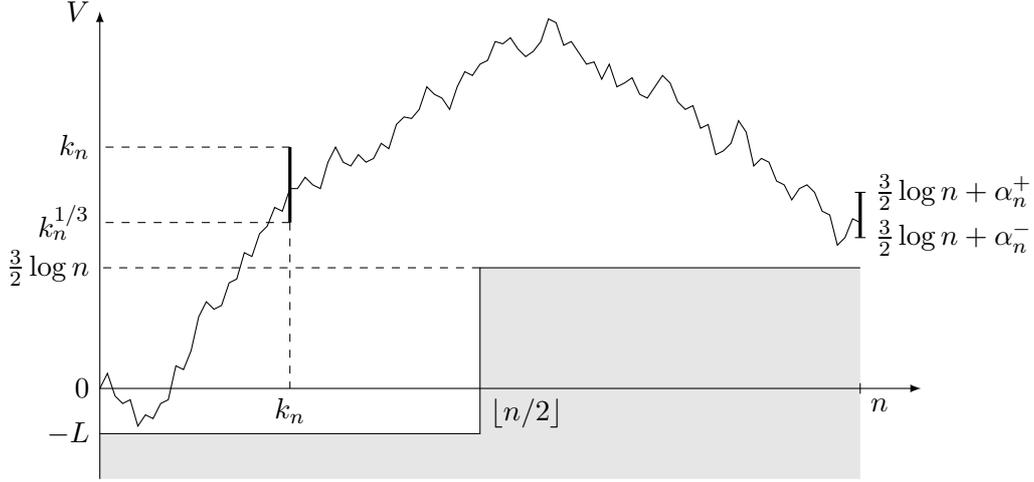

But, in order to compute the first moment of $Y_n(F)$, we will first need to take the conditional expectation given $\sF_{k_n}$ and, thus, we want to show that $F(\bfV (z))$ does not asymptotically depend on what happen before time $k_n$.
For this, we consider for each $n \in \N^*$ a slight modification $F_n$ of function $F$ such that
\begin{align*}
F_n(\bfV (z))
& \coloneqq
F \left(
\frac{V(z_{\lfloor (n- k_n) t \rfloor + k_n}) - V(z_{k_n})}{\sigma \sqrt{n - k_n}},
t \in [0,1]
\right).
\end{align*}
More formally, we set, for each $x \in \cD([0,1])$, $F_n(x) \coloneqq F( [\sqrt{n} / \sqrt{n-k_n}] \cdot [x_{((n-k_n)t + k_n)/n} - x_{k_n / n}]$, $t\in[0,1])$.
The following lemma, proved in Subsection \ref{subsection from F to F_n}, shows that we can replace $F$ with $F_n$. 
It does not play a crucial role, but makes the calculations easier for the next lemmas.
\begin{lem} \label{lemma from F to F_n}
For all $L,K >0$ and $F \in \cC_b^u(\cD([0,1]))$, we have
\[
\Eci{\Q^{(L)}}{ \frac{W^{(L,K)}_{n,\beta_n}(\abs{F - F_n})}{D_n^{(L)}} }
=
\petito{\frac{\alpha_n^2}{n^{3\beta_n/2}}},
\]
as $n \to \infty$.
\end{lem}
Noting that, for $F$ nonnegative, $Y_n(F) \leq W^{(L,K)}_{n,\beta_n}(F)$, Lemma \ref{lemma from F to F_n} combined with the following lemma shows that the first moments under $\Q^{(L)}$ of $W^{(L,K)}_{n,\beta_n}(F)$, $W^{(L,K)}_{n,\beta_n}(F_n)$, $Y_n(F)$ and $Y_n(F_n)$ (divided by $D_n^{(L)}$) have the same equivalent as $n \to \infty$.
It will be proved in Subsection \ref{subsection first moment}.
\begin{lem} \label{lemma first moment}
For all $L,K >0$ and $F \in \cC_b(\cD([0,1]))$ nonnegative, we have
\begin{align*}
\limsup_{n\to\infty} 
\frac{n^{3\beta_n/2}}{\alpha_n^2} \Eci{\Q^{(L)}}{\frac{W^{(L,K)}_{n,\beta_n}(F)}{D_n^{(L)}}}
\leq
\frac{\theta}{\sigma^2}
\Ec{F(\mathfrak{e})}
\leq
\liminf_{n\to\infty} 
\frac{n^{3\beta_n/2}}{\alpha_n^2} \Eci{\Q^{(L)}}{\frac{Y_n(F_n)}{D_n^{(L)}}}.
\end{align*}
\end{lem}
However, this is still not sufficient for controlling the second moment of $Y_n(F)$ and we need to introduce a new random variable $Y_n'(F)$.
We consider the sequences
\begin{align*}
a_j^{(n)}
\coloneqq
\left\{
\begin{array}{ll}
-L 
& \text{if } 0 \leq j < \lfloor n/2 \rfloor, \\
\frac{3}{2} \log n
& \text{if } \lfloor n/2 \rfloor \leq j \leq n,
\end{array}
\right.
\quad \text{ and } \quad
\ell_j^{(n)}
\coloneqq
\left\{
\begin{array}{ll}
j^{1/7} 
& \text{if } 0 \leq j < \lfloor n/2 \rfloor, \\
(n-j)^{1/7}
& \text{if } \lfloor n/2 \rfloor \leq j \leq n.
\end{array}
\right.
\end{align*}
Then, for some fixed sequence $(\rho_n)_{n\in\N}$ that tends to infinity such that $\rho_n \ll \alpha_n^2$, we define the following set, that will allow us to control the offspring of the spine in the second moment calculation (see A\"{\i}dékon \cite{aidekon2013} for the first use of this method),
\[
B_n
\coloneqq
\Biggl\{
	\abs{z} = n :
	\Forall j \in \llbracket 0, n-1 \rrbracket,
	\sum_{y \in \Omega(z_{j+1})} 
	\e^{-[V(y)-a_j^{(n)}]} 
	\leq 
	\rho_n
	\e^{- \ell_j^{(n)}} 
\Biggr\},
\]
where $\Omega(x)$ is the set of brothers of $x$, and we set 
\[
Y_n'(F)
\coloneqq
\sum_{\abs{z} = n} 
\e^{-\beta_n V(z)} 
F_n(\bfV (z))
\1_{z \in A_n \cap B_n}.
\]
The following lemma shows that this new random variable $Y_n'(F)$ is close to $Y_n(F)$.
Its proof relies on the peeling lemma stated in Subsection \ref{subsection peeling lemma} in a more general feature.
\begin{lem} \label{lemma from Y_n to Y_n'}
For all $L,K >0$ and $F \in \cC_b(\cD([0,1]))$, we have
\[
\Eci{\Q^{(L)}}{\frac{\abs{Y_n(F_n) - Y_n'(F_n)}}{D_n^{(L)}}}
=
\petito{\frac{\alpha_n^2}{n^{3\beta_n/2}}},
\]
as $n \to \infty$.
\end{lem}
By considering $Y_n'(F_n)$, 
we can now control the second moment properly, as stated in this last lemma, proved in Subsection \ref{subsection second moment}. 
This second moment is exactly the square of the first moment in Lemma \ref{lemma first moment}.
\begin{lem} \label{lemma second moment}
For all $L,K >0$ and $F \in \cC_b(\cD([0,1]))$ nonnegative, we have
\[
\limsup_{n\to\infty} 
\left( \frac{n^{3\beta_n/2}}{\alpha_n^2} \right)^2 
\Eci{\Q^{(L)}}{\left( \frac{Y_n'(F_n)}{D_n^{(L)}} \right)^2}
\leq
\left( \frac{\theta}{\sigma^2} \Ec{F(\mathfrak{e})} \right)^2.
\]
\end{lem}
Using these lemmas, we can now prove Proposition \ref{proposition convergence under Q}.
\begin{proof}[Proof of Proposition \ref{proposition convergence under Q}]
We consider here the case $F \geq 0$ and the general case follows by taking the positive and negative parts of $F$.
We first fix $K>0$.
Using Lemmas \ref{lemma first moment} et \ref{lemma from Y_n to Y_n'}, we have
\begin{align} \label{ba}
\liminf_{n\to\infty} 
\frac{n^{3\beta_n/2}}{\alpha_n^2} \Eci{\Q^{(L)}}{\frac{Y_n'(F_n)}{D_n^{(L)}}}
\geq
\frac{\theta}{\sigma^2} \Ec{F(\mathfrak{e})},
\end{align}
and so, by applying Bienaymé-Chebyshev inequality and Lemma \ref{lemma second moment}, we get that
\begin{align} \label{bb}
\frac{n^{3 \beta_n/2}}{\alpha_n^2} \frac{Y_n'(F_n)}{D_n^{(L)}}
\xrightarrow[n\to\infty]{}
\frac{\theta}{\sigma^2} \Ec{F(\mathfrak{e})},
\quad \text{in } \Q^{(L)}\text{-probability}.
\end{align}
Combining \eqref{bb} with Lemmas \ref{lemma first moment}, \ref{lemma from F to F_n} and \ref{lemma from Y_n to Y_n'}, we deduce that for all $K>0$,
\begin{align} \label{bd}
\frac{n^{3\beta_n/2}}{\alpha_n^2} \frac{W^{(L,K)}_{n,\beta_n}(F)}{D_n^{(L)}}
\xrightarrow[n\to\infty]{}
\frac{\theta}{\sigma^2} \Ec{F(\mathfrak{e})},
\quad \text{in } \Q^{(L)}\text{-probability}.
\end{align}
Now, for $\varepsilon,\eta >0$, by Lemma \ref{lemma addition of the barrier}, we can choose $K >0$ such that
\begin{align} \label{be}
\limsup_{n\to\infty} 
\Q^{(L)}
\left(
\frac{n^{3\beta_n/2}}{\alpha_n^2} 
\frac{\abs{W^{(L)}_{n,\beta_n}(F) - W^{(L,K)}_{n,\beta_n}(F)}}{D_n^{(L)}}
>
\frac{\varepsilon}{2}
\right)
\leq 
\eta
\end{align}
and, thus, we get, using the triangle inequality, \eqref{bd} and \eqref{be},
\begin{align*}
\limsup_{n\to\infty} 
\Q^{(L)}
\left(
\abs{
\frac{n^{3\beta_n/2}}{\alpha_n^2} 
\frac{W^{(L)}_{n,\beta_n}(F)}{D_n^{(L)}}
-
\frac{\theta}{\sigma^2} \Ec{F(\mathfrak{e})}
}
>
\varepsilon
\right)
\leq 
\eta.
\end{align*}
This concludes the proof of Proposition \ref{proposition convergence under Q}.
\end{proof}
We conclude this subsection by proving part (i) of Corollary \ref{cor typical energy}, which is a consequence of the fact that considering only particles in $A_n$ does not change the first moment asymptotic.
\begin{proof}[Proof of part (i) of Corollary \ref{cor typical energy}]
First note that it is sufficient to prove that, for all sequences $(\alpha^-_n)_{n\in\N}$ and $(\alpha^+_n)_{n\in\N}$ such that $1 \ll \alpha^-_n \ll \alpha_n \ll \alpha^+_n \ll \sqrt{n}$, we have
\begin{align} \label{bf}
\nu_{n,\beta_n}
\left( \left[
(3/2) \log n + \alpha_n^-, (3/2) \log n + \alpha_n^+
\right] \right)
\xrightarrow[n\to\infty]{}
1,
\quad \text{in } \P^*\text{-probability}.
\end{align}
Indeed, if we assume that part (i) of Corollary \ref{cor typical energy} is false, then it exists $\varepsilon>0$ such that for any $k \geq 1$, it exists $n_k > n_{k-1}$ (with $n_0 \coloneqq 0$) such that
\begin{align} \label{bi}
\P^* \left(
	\nu_{n_k,\beta_{n_k}}
	\left( \left[
	(3/2) \log n_k + k^{-1} \alpha_{n_k}, (3/2) \log n_k + k \alpha_{n_k}
	\right] \right)
	\leq 1-\varepsilon
	\right)
\geq
1 - \varepsilon.
\end{align}
Setting $e_n \coloneqq \inf \{k\in\N : n_k \geq n\}$, we have $e_n \to \infty$ and $e_{n_k} = k$.
Thus, with $\alpha_n^+ \coloneqq e_n \alpha_n$ and $\alpha_n^- \coloneqq e_n^{-1} \alpha_n$, \eqref{bi} implies the negation of \eqref{bf}.
Therefore, we now want to prove \eqref{bf}.
The left-hand side of \eqref{bf} is larger than $Y_n(1) / W_{n,\beta_n}$, therefore, it is sufficient to show that $Y_n(1) / W_{n,\beta_n} \to 1$ in $\P^*$-probability. 
Combining \eqref{bb} with Lemma \ref{lemma from Y_n to Y_n'}, we first have
\begin{align} \label{bg}
\frac{n^{3 \beta_n/2}}{\alpha_n^2} \frac{Y_n(1)}{D_n^{(L)}}
\xrightarrow[n\to\infty]{}
\frac{\theta}{\sigma^2},
\quad \text{in } \Q^{(L)}\text{-probability},
\end{align}
and, in the same way as in the proof of part (i) of Theorems \ref{theorem W} and \ref{theorem trajectory}, where we showed \eqref{aa} from Proposition \ref{proposition convergence under Q}, it follows from \eqref{bg} that
\begin{align} \label{bh}
\frac{n^{3\beta_n/2}}{\alpha_n^2} \frac{Y_n(1)}{D_n}
\xrightarrow[n\to\infty]{}
\frac{c_0 \theta}{\sigma^2},
\quad \text{in } \P^*\text{-probability}.
\end{align}
Using \eqref{bh} and \eqref{aa} with $F\equiv 1$, we get that $Y_n(1) / W_{n,\beta_n} \to 1$ in $\P^*$-probability and so \eqref{bf} is proved.
\end{proof}

\subsection{First moments of \texorpdfstring{$W^{(L,K)}_{n,\beta_n}(F)$}{W(L,K)(F)} and \texorpdfstring{$Y_n(F_n)$}{Yn(Fn)}}
\label{subsection first moment}

We start with the proof of Lemma \ref{lemma first moment} in this subsection, because this first moment calculation will be a kind of routine at which we will refer for the proof of other lemmas.
In this calculations, sums of general term $(i+C) \e^{-i / \gamma_n}$ appear regularly, with $C \in \R$ a constant and $\gamma_n \to \infty$ as $n \to \infty$.
Therefore, we state the following result, proved by explicit computation: if $(\gamma^+_n)_{n\in\N}$ and $(\gamma^-_n)_{n\in\N}$ are sequences with values in $\R_+ \cup \{ \infty \}$ such that $\gamma_n^- \ll \gamma_n \ll \gamma^+_n$ as $n\to\infty$, then we have, for any constant $C \in \R$,
\begin{align} \label{eq resultat somme}
\sum_{i \in [\gamma_n^-, \gamma_n^+] \cap \N}
(i+C) \e^{-i / \gamma_n}
\underset{n\to\infty}{\sim}
\gamma_n^2
\quad \text{ and } \quad
\sum_{i \notin [\gamma_n^-, \gamma_n^+] \cap \N}
(i+C) \e^{-i / \gamma_n}
=
\petito{\gamma_n^2}.
\end{align}
\begin{proof}[Proof of Lemma \ref{lemma first moment}]
Recall that we consider $F \in \cC_b(\cD([0,1]))$ nonnegative.
We begin with the control of $\E_{\Q^{(L)}}[W^{(L,K)}_{n,\beta_n}(F)/D_n^{(L)}]$.
Note that $\E_{\Q^{(L)}}[W^{(L,K)}_{n,\beta_n}(F)/D_n^{(L)}] = \E[W^{(L,K)}_{n,\beta_n}(F)] / R(L)$.
Then, using the many-to-one lemma, we get
\begin{align*}
\Ec{W^{(L,K)}_{n,\beta_n}(F)}
& =
\Ec{
\e^{-S_n / \alpha_n}
F(\bfS^{(n)})
\1_{\underline{S}_n \geq -L,
	\min_{\lfloor n/2 \rfloor \leq j \leq n} V(z_j) 
			\geq \frac{3}{2} \log n - K}}.
\end{align*}
We cut this expectation in two pieces depending on whether $S_n \leq \frac{3}{2} \log n + \alpha^+_n$ or $S_n > \frac{3}{2} \log n + \alpha^+_n$.
Let start with the case $S_n > \frac{3}{2} \log n + \alpha^+_n$: we have, by cutting the interval $[\alpha^+_n + K, \infty)$ in pieces of length $1$,
\begin{align}
& \Ec{
\e^{-S_n / \alpha_n}
F(\bfS^{(n)})
\1_{\underline{S}_n \geq -L, 
	\min_{\lfloor n/2 \rfloor \leq j \leq n} S_j \geq \frac{3}{2} \log n - K,
	S_n > \frac{3}{2} \log n + \alpha^+_n}
} \nonumber \\
& \leq
\sum_{i \geq \lfloor \alpha^+_n + K \rfloor}
\Ec{
\e^{-(\frac{3}{2} \log n -K + i) / \alpha_n}
\norme{F}
\1_{\underline{S}_n \geq -L, 
	\min_{\lfloor n/2 \rfloor \leq j \leq n} S_j \geq \frac{3}{2} \log n - K, 
	S_n - \left( \frac{3}{2} \log n - K \right) \in [i, i+1) 
}} \nonumber \\
& \leq
\norme{F}
\frac{\e^{K / \alpha_n}}{n^{3/2\alpha_n}}
\sum_{i \geq \lfloor \alpha^+_n + K \rfloor}
\e^{-i / \alpha_n}
c_5 \frac{(L+1) 2 (i+2)}{n^{3/2}}, \label{cb}
\end{align}
by using \eqref{eq ballot theorem and barrier}.
Since $\lfloor \alpha^+_n + K \rfloor \gg \alpha_n$, it follows from \eqref{eq resultat somme} that the right-hand side of \eqref{cb} is a $o(\alpha_n^2/n^{3\beta_n/2})$.
We now control the term corresponding to the case $S_n \leq \frac{3}{2} \log n + \alpha^+_n$, which is the dominant term in $\E[W^{(L,K)}_{n,\beta_n}(F)]$. 
This time, we cut the segment $[0, \alpha^+_n + K]$ in pieces of length $h>0$, where $h$ is any real number if $S_1$ is nonlattice and is the span of the lattice if $S_1$ is lattice, and thus we get 
\begin{align}
& \Ec{
\e^{-S_n / \alpha_n}
F(\bfS^{(n)})
\1_{\underline{S}_n \geq -L, 
	\min_{\lfloor n/2 \rfloor \leq j \leq n} S_j \geq \frac{3}{2} \log n - K,
	S_n \leq \frac{3}{2} \log n + \alpha^+_n}
} \nonumber \\
& \leq
\sum_{i = 0}^{\lceil (\alpha^+_n + K) /h \rceil - 1}
\frac{\e^{(K-ih) / \alpha_n}}{n^{3/2\alpha_n}}
\Ec{
	F(\bfS^{(n)})
	\1_{
	\underline{S}_n \geq -L, 
	\min_{\lfloor n/2 \rfloor \leq j \leq n} S_j \geq \frac{3}{2} \log n - K,
	S_n - \left( \frac{3}{2} \log n - K \right) \in [ih, (i+1)h)} 
	} \nonumber
\\
& \leq
\sum_{i = 0}^{\lceil (\alpha^+_n + K) /h \rceil - 1}
\frac{\e^{-ih / \alpha_n}}{n^{3/2\alpha_n}}
\sqrt{\frac{\pi}{2}} \frac{\theta \theta^-}{\sigma}
\frac{R(L)}{n^{3/2}} 
(\Ec{F(\mathfrak{e})}+\petito{1})
h R^-((i+1)h), \label{cm}
\end{align}
by using Proposition \ref{prop convergence to the excursion} in both lattice and nonlattice cases, with uniformity in $i$ because $h (\lceil (\alpha^+_n + K) /h \rceil - 1) \ll \sqrt{n}$.
Then, by applying \eqref{equation R 1} to $R^-$, for $\varepsilon >0$, it exists $M >0$ such that, for all $u\geq 0$, $R^-(u) \leq c_0^- (1+\varepsilon) (M+u)$ and thus we get
\begin{align}
\sum_{i = 0}^{\lceil (\alpha^+_n + K) /h \rceil - 1}
R^-((i+1)h) \e^{-ih / \alpha_n}
& \leq 
\sum_{i = 0}^{\infty}
c_0^- (1+\varepsilon)
\left( i+1 + \frac{M}{h} \right) h
\e^{-ih / \alpha_n} \nonumber \\
& =
(1+\petito{1}) c_0^- (1+\varepsilon) h
\left( \frac{\alpha_n}{h} \right)^2, \label{cn}
\end{align}
by applying \eqref{eq resultat somme}.
Coming back to \eqref{cm}, we get 
\begin{align}
& \limsup_{n\to\infty} \frac{n^{3\beta_n/2}}{\alpha_n^2} 
\Ec{\e^{-S_n / \alpha_n}
F(\bfS^{(n)})
\1_{\underline{S}_n \geq -L, 
	\min_{\lfloor n/2 \rfloor \leq j \leq n} S_j \geq \frac{3}{2} \log n - K,
	S_n \leq \frac{3}{2} \log n + \alpha^+_n}
} \nonumber \\
& \leq
\sqrt{\frac{\pi}{2}} 
\frac{\theta \theta^-}{\sigma} 
R(L) \Ec{F(\mathfrak{e})} h
\frac{c_0^- (1+\varepsilon)}{h}
\xrightarrow[\varepsilon\to0]{} 
\frac{\theta}{\sigma^2} 
R(L) \Ec{F(\mathfrak{e})}, \label{cd}
\end{align}
by applying \eqref{equation lien theta c_0} to constants $c_0^-$ and $\theta^-$.
Combining (\ref{cb}) and (\ref{cd}), we conclude that
\begin{align*}
\Eci{\Q^{(L)}}{\frac{W^{(L,K)}_{n,\beta_n}}{D_n^{(L)}}}
= 
\frac{1}{R(L)} \Ec{W^{(L,K)}_{n,\beta_n}}
& \leq
\frac{\theta}{\sigma^2} 
\frac{\alpha_n^2}{n^{3\beta_n/2}}  (\Ec{F(\mathfrak{e})}+\petito{1}),
\end{align*}
and it shows the first part of Lemma \ref{lemma first moment}.

We now want to prove the lower bound for $\E_{\Q^{(L)}}[Y_n (F_n)/D_n^{(L)}] = \E[Y_n(F_n)] / R(L)$.
We use the branching property at time $k_n$ to get
\begin{align} \label{ce}
\Ec{Y_n(F_n)}
& =
\E \Biggl[
\sum_{\abs{x} = k_n} 
\1_{\underline{V}(x) \geq -L, V(x) \in [k_n^{1/3}, k_n]}
\psi(V(x))
\Biggr],
\end{align}
where we set, for all $b \in [k_n^{1/3}, k_n]$, 
\begin{align*}
\psi(b)
& \coloneqq
\E_b \Biggr[
\sum_{\abs{z} = n-k_n}
\e^{-\beta_n V(z)} 
F \left( \bfV(z)- \frac{b}{\sigma \sqrt{n-k_n}} \right) 
\1_{ \Forall 0 \leq i \leq n-k_n, V(z_i) \in I_{n,i+k_n}}
\Biggl].
\end{align*}
We note $m \coloneqq n-k_n$ and fix $\lambda \in (0,1/2)$, then we have $\lfloor \lambda m \rfloor \leq \lfloor n/2 \rfloor - k_n$ for $n$ large enough because $k_n \ll n$.
Using the many-to-one lemma, we get, for $b \in [k_n^{1/3}, k_n]$,
\begin{align*} 
\psi(b)
& \geq
\e^{-b}
\Eci{b}{
\e^{-S_m / \alpha_n} 
F \left( \bfS^{(m)} - \frac{b}{\sigma \sqrt{m}} \right)
\1_{
S_m - \frac{3}{2} \log n \in [\alpha^-_n, \alpha^+_n],
\underline{S}_m \geq -L,
\min_{\lfloor \lambda m \rfloor \leq j \leq m} S_j 
	\geq \frac{3}{2} \log n}
} \\
& =
\e^{-b}
\Ec{
\e^{-(S_m+b) / \alpha_n} 
F(\bfS^{(m)})
\1_{
S_m - ( \frac{3}{2} \log n -b ) \in [\alpha^-_n, \alpha^+_n],
\underline{S}_m \geq -(L+b),
\min_{\lfloor \lambda m \rfloor \leq j \leq m} S_j 
	\geq \frac{3}{2} \log n -b}
}.
\end{align*}
Then, we cut the segment $[\alpha^-_n, \alpha^+_n]$ in pieces of length $h>0$, where $h$ is any real number if $S_1$ is nonlattice and is the span of the lattice if $S_1$ is lattice, and we get that $\psi (b)$ is larger than
\begin{align*}
&
\e^{- b} 
\sum_{i = \lceil \alpha^-_n / h \rceil}
	^{\lfloor \alpha^+_n / h \rfloor -1}
\frac{\e^{- (i+1)h/ \alpha_n}}{n^{3/2\alpha_n}} 
\Ec{
	F(\bfS^{(m)})
	\1_{S_m - ( \frac{3}{2} \log n -b ) \in \left[ ih, (i+1)h \right),
		\underline{S}_m \geq - (L+b),
		\min_{\lfloor \lambda m \rfloor \leq j \leq m} S_j 
			\geq \frac{3}{2} \log n -b}
	} \\
& \geq
\frac{\e^{- b}}{n^{3/2\alpha_n}} 
\sum_{i = \lceil \alpha^-_n / h \rceil}
	^{\lfloor \alpha^+_n / h \rfloor -1}
\e^{-(i+1)h / \alpha_n}
\sqrt{\frac{\pi}{2}} \frac{\theta \theta^-}{\sigma}
\frac{R(L+b)}{m^{3/2}} 
(\Ec{F(\mathfrak{e})}+\petito{1})
h R^-(ih),
\end{align*}
where the $\petito{1}$ is uniform in $b \in [k_n^{1/3}, k_n]$ and $i \in \llbracket \lceil \alpha^-_n / h \rceil, \lfloor \alpha^+_n / h \rfloor -1 \rrbracket$, by using Proposition \ref{prop convergence to the excursion}, because we have $L + k_n \ll \sqrt{n}$ and $h(\lfloor \alpha^+_n / h \rfloor -1) \ll \sqrt{n}$. 
Thus, we get, using that $R(L+b) \sim c_0 (L+b)$ and $R^-(ih) \sim c_0^- ih$ uniformly in $b$ and $i$,
\begin{align*}
\psi(b)
& \geq
\frac{\e^{- b}}{n^{3/2\alpha_n}} 
\sqrt{\frac{\pi}{2}} 
\frac{\theta \theta^-}{\sigma} 
\frac{c_0 (L+b)}{n^{3/2}}
h (\Ec{F(\mathfrak{e})}+\petito{1})
\sum_{i = \lceil \alpha^-_n / h \rceil}
	^{\lfloor \alpha^+_n / h \rfloor -1} 
\e^{-(i+1)h / \alpha_n}
c_0^- ih \\
& =
(L+b) \e^{- b}
\sqrt{\frac{2}{\pi}} 
\frac{1}{\sigma^3}
\frac{\alpha_n^2}{n^{3 \beta_n/2}}
(\Ec{F(\mathfrak{e})}+\petito{1}),
\end{align*}
uniformly in $b \in [k_n^{1/3}, k_n]$, where we used \eqref{equation lien theta c_0} twice and also \eqref{eq resultat somme}.
Coming back to (\ref{ce}), we get that $\Ec{Y_n(F_n)}$ is larger than
\begin{align} \label{cf}
\sqrt{\frac{2}{\pi}} 
\frac{1}{\sigma^3}
\frac{\alpha_n^2}{n^{3 \beta_n/2}}
(\Ec{F(\mathfrak{e})}+\petito{1})
\E \Biggl[
\sum_{\abs{x} = k_n} 
\1_{\underline{V}(x) \geq -L, V(x) \in [k_n^{1/3}, k_n]}
(L+V(x)) \e^{-V(x)}
\Biggr].
\end{align}
Using the many-to-one lemma, the expectation in \eqref{cf} is equal to
\begin{align}
\Ec{(S_{k_n} + L)
\1_{\underline{S}_{k_n} \geq -L, S_{k_n} \in [k_n^{1/3}, k_n]}} 
& \geq
\Ec{(S_{k_n} + L)
\1_{\underline{S}_{k_n} \geq -L, 
	(S_{k_n} + L)/ \sigma k_n^{1/2} \in [C^{-1},C]}}, \label{cg}
\end{align}
for all $C >0$.
We then choose a function $\chi \colon \R_+ \to \R$ continuous and bounded such that, for all $t \in \R_+$,
\[
t \1_{t  \in [2 C^{-1} , C/2 ]}
\leq
\chi(t) 
\leq 
t \1_{t  \in [C^{-1} , C]},
\]
and \eqref{cg} is larger than
\begin{align*}
\sigma k_n^{1/2}
\Ec{\chi \left( \frac{S_{k_n} + L}{\sigma k_n^{1/2}} \right) 
\1_{\underline{S}_{k_n} \geq -L}}
& =
\sigma k_n^{1/2} (1+\petito{1})
\frac{\theta R(L)}{k_n^{1/2}}
\int_0^\infty \chi(t) t \e^{-t^2/2} \diff t \\
& \geq
\sigma \theta R(L) (1+\petito{1})
\int_{2C^{-1}}^{C / 2} t^2 \e^{-t^2/2} \diff t,
\end{align*}
by applying \eqref{equation convergence vers la loi Rayleigh}.
Coming back to (\ref{cf}), we get
\begin{align*}
\liminf_{n\to\infty} \frac{n^{3 \beta_n/2}}{\alpha_n^2}
\Ec{Y_n(F_n)}
& \geq
\sqrt{\frac{2}{\pi}} 
\frac{1}{\sigma^3}
\Ec{F(\mathfrak{e})}
\sigma \theta R(L)
\int_{2C^{-1}}^{C / 2} t^2 \e^{-t^2/2} \diff t 
\xrightarrow[C \to \infty]{}
\frac{\theta}{\sigma^2}
R(L)
\Ec{F(\mathfrak{e})},
\end{align*}
using that $\int_0^\infty t^2 \e^{-t^2/2} \diff t = \sqrt{\frac{\pi}{2}}$.
Since $\E_{\Q^{(L)}}[Y_n(F_n)/D_n^{(L)}] = \E[Y_n(F_n)] / R(L)$, it concludes the proof of Lemma \ref{lemma first moment}.
\end{proof}

\subsection{Addition of the second barrier}
\label{subsection addition of the second barrier}

In this section, we prove Lemma \ref{lemma addition of the barrier} with a method similar to the one used by Madaule \cite{madaule2015} for his Lemma 4.9 (or Lemma 3.3 of A\"{\i}dékon \cite{aidekon2013}).
The main difference is that we do not only consider particles that are at a distance of order $1$ from $(3/2) \log n$, but we can nevertheless apply some of Madaule's results.
\begin{proof}[Proof of Lemma \ref{lemma addition of the barrier}]
We fix $L>0$ and $\varepsilon, \eta > 0$. 
For all $F \in \cC_b(\cD([0,1]))$ and $K > 0$, we have $\lvert W^{(L)}_{n,\beta_n}(F) - W^{(L,K)}_{n,\beta_n}(F) \rvert \leq 2 \norme{F} \lvert W^{(L)}_{n,\beta_n} - W^{(L,K)}_{n,\beta_n} \rvert$, where $W^{(L)}_{n,\beta_n} \coloneqq W^{(L)}_{n,\beta_n}(1)$ and $W^{(L,K)}_{n,\beta_n} \coloneqq W^{(L,K)}_{n,\beta_n}(1)$.
Therefore, it is sufficient to show that we have 
\begin{align} \label{fi}
\limsup_{n\to\infty}
\Q^{(L)}
\left(
\frac{n^{3\beta_n/2}}{\alpha_n^2} 
\frac{\abs{W^{(L)}_{n,\beta_n} - W^{(L,K)}_{n,\beta_n}}}{D_n^{(L)}}
>
\varepsilon
\right)
\leq \eta,
\end{align}
for $K > 0$ large enough.
Using Proposition A.3 of A\"{\i}dékon \cite{aidekon2013}, we know that $D_n^{(L)}$ converges in $L^1$ to $D_\infty^{(L)}$ under $\P$, so we can choose $M > 0$ large enough such that, for all $n \in \N$, $\Q^{(L)} (D_n^{(L)} > M) = \E [D_n^{(L)} \1_{D_n^{(L)} > M}] / R(L) \leq \eta/4$.
Therefore, the probability in \eqref{fi} is smaller than
\begin{align*} 
& 
\Q^{(L)}
\left(
n^{3\beta_n/2}
\abs{ W^{(L)}_{n,\beta_n} - W^{(L,K)}_{n,\beta_n}}
>
\varepsilon \alpha_n^2 D_n^{(L)},
\frac{\eta}{4} \leq D_n^{(L)} \leq M
\right)
+
\frac{\eta}{4} 
+
\Q^{(L)}
\left(
D_n^{(L)} < \frac{\eta}{4}
\right) \\
& =
\Ec{D_n^{(L)}
	\1_{n^{3\beta_n/2}
		\lvert W^{(L)}_{n,\beta_n} - W^{(L,K)}_{n,\beta_n} \rvert
		>
		\varepsilon \alpha_n^2 D_n^{(L)},
		\frac{\eta}{4} \leq D_n^{(L)} \leq M}}
+
\frac{\eta}{4}
+
\frac{1}{R(L)}
\Ec{D_n^{(L)} \1_{D_n^{(L)} < \frac{\eta}{4}}} \\
& \leq 
M
\Pp{n^{3\beta_n/2}
	\abs{ W^{(L)}_{n,\beta_n} - W^{(L,K)}_{n,\beta_n}}
	>
	\varepsilon \alpha_n^2 \frac{\eta}{4}}
+
\frac{\eta}{4}
+
\frac{\eta}{4},
\end{align*}
because $R(L) \geq 1$.
Thus, we now want to prove that, for some $K >0$ large enough, we have
\begin{align*}
\limsup_{n\to\infty}
\Pp{n^{3\beta_n/2}
	\abs{ W^{(L)}_{n,\beta_n} - W^{(L,K)}_{n,\beta_n}}
	>
	\varepsilon' \alpha_n^2}
\leq
2 \eta',
\end{align*}
with $\eta' \coloneqq \eta/4M$ and $\varepsilon' \coloneqq \varepsilon \eta /4$.
Moreover, using \eqref{eq tension position la plus basse}, we can fix $K' \geq 0$ such that $\P ( \min_{\abs{z} = n} V(z) < \frac{3}{2} \log n - K') \leq \eta'$.
Thus, our aim is now to show that we have
\begin{align} \label{fa}
\limsup_{n\to\infty}
\Pp{
	\abs{ W^{(L)}_{n,\beta_n} - W^{(L,K)}_{n,\beta_n}}
	\1_{\min_{\abs{z} = n} V(z) \geq \frac{3}{2} \log n - K'}
	>
	\varepsilon' \frac{\alpha_n^2}{n^{3\beta_n/2}}}
\leq
\eta',
\end{align}
for some $K \geq K'$ large enough.

Now, following Madaule's \cite{madaule2015} proof of his Lemma 4.9, we introduce the intervals $J_n(x) \coloneqq [ \frac{3}{2} \log n - x-1, \frac{3}{2} \log n - x )$, for $x\in \R$, and the events, for $i,\ell \in \N$ and $\lfloor n/2 \rfloor \leq k \leq n$,
\begin{align*}
E_{i,k,\ell}(z)
\coloneqq
\left\{
\underline{V}(z) \geq -L,
V(z_k) = \min_{\lfloor n/2 \rfloor \leq j \leq n} V(z_j)
	   \in J_n(K) - \ell,
V(z) \in J_n(K') + i
\right\},
\end{align*}
and, denoting $a_\ell \coloneqq \lfloor \e^{\nu (\ell + K)} \rfloor$ for some fixed $\nu \in (0,1)$,
\begin{align*}
F_\ell^1 (z)
\coloneqq
\bigcup_{i \geq 1, \lfloor n/2 \rfloor \leq k < n-a_\ell} E_{i,k,\ell}(z)
\quad \text{ and } \quad
F_\ell^2 (z)
\coloneqq
\bigcup_{i \geq 1, n- a_\ell \leq k \leq n} E_{i,k,\ell}(z).
\end{align*}
Then, we have
\begin{align} \label{fb}
\abs{W^{(L)}_{n,\beta_n} - W^{(L,K)}_{n,\beta_n}}
\1_{\min_{\abs{z} = n} V(z) \geq \frac{3}{2} \log n - K'}
& \leq
\sum_{\ell \geq 0}
\sum_{\abs{z} = n} 
	\e^{-\beta_n V(z)} 
	\left(
	\1_{F_\ell^1 (z)} + \1_{F_\ell^2 (z)}
	\right).
\end{align}
On the one hand, by Madaule's \cite{madaule2015} proof of his Lemma 4.9, we have the inequality
\begin{align*}
\P \Biggl( \sum_{\abs{z} = n} \1_{F_\ell^2 (z)} \geq 1 \Biggr)
\leq
c_{14} (1+a_\ell) (1+L) \e^{-K-\ell}
\end{align*}
and, therefore,
\begin{align}
\P \Biggl( 
\sum_{\ell \geq 0}
\sum_{\abs{z} = n} 
\e^{-\beta_n V(z)} 
\1_{F_\ell^2 (z)}
> 0 \Biggr)
& \leq
\sum_{\ell \geq 0}
\P \Biggl( \sum_{\abs{z} = n} \1_{F_\ell^2 (z)} \geq 1 \Biggr)
\leq
c_{15} (1+L) \e^{-(1-\nu)K}. \label{fc}
\end{align}
On the other hand, by using the many-to-one lemma, we get, for $i \geq 1$, $\lfloor n/2 \rfloor \leq k < n-a_\ell$ and $\ell \geq 0$,
\begin{align} \label{fd}
\E \Biggl[
\sum_{\abs{z} = n} 
	\e^{-\beta_n V(z)} 
	\1_{E_{i,k,\ell}(z)}
	\Biggr]
& =
\Ec{
\e^{-S_n/\alpha_n }
\1_{E_{i,k,\ell}} 
}
\leq
\frac{\e^{-(i-1 -K')/\alpha_n }}{n^{3/2\alpha_n}}
\Pp{E_{i,k,\ell}}, 
\end{align}
where we set
\begin{align*}
E_{i,k,\ell}
\coloneqq
\left\{
\underline{S}_n \geq -L,
S_k = \min_{\lfloor n/2 \rfloor \leq j \leq n} S_j
	\in J_n(K) - \ell,
S_n \in J_n(K') + i
\right\}.
\end{align*}
We recall Equation (4.27) of Madaule \cite{madaule2015}: 
\begin{align} \label{fe}
\Pp{E_{i,k,\ell}}
\leq
\left\{
\begin{array}{ll}
c_{16} \frac{(1+L) \log n}{n^{3/2} (n-k+1)^{3/2}} (1 + \ell + i)
& \text{if } \lfloor n/2 \rfloor \leq k < \lfloor 3n/4 \rfloor, \\
c_{16} \frac{(1+L)}{n^{3/2} (n-k+1)^{3/2}} (1 + \ell + i)
& \text{if } \lfloor 3n/4 \rfloor \leq k \leq n.
\end{array}
\right.
\end{align}
Applying \eqref{fe}, we get
\begin{align} \label{ff}
\sum_{k= \lfloor n/2 \rfloor}^{n-a_\ell-1}
\Pp{E_{i,k,l}}
\leq
c_{17} \frac{(1+L)}{n^{3/2}} (1 + \ell + i)
\left(
\frac{\log n}{\sqrt{n}}
+
a_\ell^{-1/2}
\right).
\end{align}
Using \eqref{fd}, \eqref{ff} and that $\Pp{E_{i,k,\ell}} = 0$ for $\ell \geq \frac{3}{2} \log n - K + L$, we have
\begin{align}
& \E \Biggl[
\sum_{\ell \geq 0}
\sum_{\abs{z} = n} 
\e^{-\beta_n V(z)} 
\1_{F_\ell^1 (z)}
\Biggr] \nonumber \\
& \leq
\sum_{\ell \geq 0}
\sum_{i\geq 1}
\sum_{k= \lfloor n/2 \rfloor}^{n-a_\ell-1}
\E \Biggl[
\sum_{\abs{z} = n} 
	\e^{-\beta_n V(z)} 
	\1_{E_{i,k,\ell}(z)} \Biggr] \nonumber \\
& \leq
c_{17} \e^{(K'+1)/\alpha_n }
\frac{(1+L)}{n^{3\beta_n /2}}
\sum_{\ell = 0}^{\lceil \frac{3}{2} \log n - K + L \rceil}
\left(
\frac{\log n}{\sqrt{n}}
+
a_\ell^{-1/2}
\right)
\sum_{i\geq 1}
\e^{-i/\alpha_n }
(1 + \ell + i).  \label{fg}
\end{align}
We can bound the sum on $i$ in \eqref{fg} by $(1+\ell) \alpha_n^2(1+\petito{1})$ uniformly in $\ell$.
Moreover, taking $K$ large enough such that $\e^{- \nu K} < 1/2$, we have $a_\ell \geq \e^{\nu (\ell + K)} / 2$ for all $\ell \geq 0$.
Thus, we get that \eqref{fg} is smaller than
\begin{align*}
& c_{18} (1+ \petito{1})
\frac{(1+L)}{n^{3\beta_n /2}}
\alpha_n^2
\left(
\frac{(\log n)^3}{n^{1/2}}
+
\e^{- \nu K / 2}
\right)
=
c_{18} \e^{- \nu K / 2} (1+L)
(1+ \petito{1})
\frac{\alpha_n^2}{n^{3\beta_n /2}},
\end{align*}
and, thus, with the Markov inequality, we have
\begin{align} \label{fh}
\P \Biggl(
\sum_{\ell \geq 0}
\sum_{\abs{z} = n} 
\e^{-\beta_n V(z)} 
\1_{F_\ell^1 (z)} 
>
\varepsilon' \frac{\alpha_n^2}{n^{3\beta_n/2}} \Biggr)
& \leq
\frac{c_{18}}{\varepsilon'} \e^{- \nu K / 2}
(1+L) (1+ \petito{1}). 
\end{align}
Finally, \eqref{fa} follows from \eqref{fb}, \eqref{fc} and \eqref{fh} by taking $K$ large enough and it concludes the proof of Lemma \ref{lemma addition of the barrier}.
\end{proof}

\subsection{From \texorpdfstring{$F$}{F} to \texorpdfstring{$F_n$}{Fn}}
\label{subsection from F to F_n}

We prove here that considering $F_n$ instead of $F$ does not change significantly the first moment.
\begin{proof}[Proof of Lemma \ref{lemma from F to F_n}]
To control the first moment of $W^{(L,K)}_{n,\beta_n}(\abs{F - F_n}) / D_n^{(L)}$ under $\E_{\Q^{(L)}}$, we follow the proof of Lemma \ref{lemma first moment} for the upper bound of the first moment of $W^{(L,K)}_{n,\beta_n}(F) / D_n^{(L)}$, but, instead of applying directly Proposition \ref{prop convergence to the excursion} as in \eqref{cm}, we use that
\begin{align*}
& \Ec{\abs{F - F_n} (\bfS^{(n)})
	\1_{\underline{S}_n \geq -L, 
	\min_{\lfloor n/2 \rfloor \leq j \leq n} S_j \geq \frac{3}{2} \log n - K,
	S_n - ( \frac{3}{2} \log n - K ) \in [ih, (i+1)h)}} \\
& =
\petito{
\sqrt{\frac{\pi}{2}} \frac{\theta \theta^-}{\sigma}
\frac{R(L)}{n^{3/2}} 
h R^-((i+1)h) },
\end{align*}
uniformly in $i \in [0, h (\lceil (\alpha^+_n + K) /h \rceil - 1)]$,
by using Proposition \ref{prop convergence to the excursion} combined with Lemmas \ref{lemma weak convergence approximation F_n F} and \ref{lemma weak convergence approximation F_n F 2}, because $F \in \cC_b^u(\cD([0,1]))$.
The result follows with the same calculations as in the proof of Lemma \ref{lemma first moment}.
\end{proof}

\subsection{The peeling lemma}
\label{subsection peeling lemma}

The aim of this subsection is to prove Lemma \ref{lemma from Y_n to Y_n'}, which shows that introducing the event $\{z \in B_n \}$ does not change the first moment.
This proof is based on the so-called peeling lemma (see Shi \cite{shi2015}), which controls that the spine, conditioned to have a specific trajectory, does not have too many and too low children. 
Such lemmas have been proved in the case where the spine ends up at a distance of constant order from $\frac{3}{2} \log n$ (see \cite[Lemma C.1]{aidekon2013}, \cite[Lemma 7.1]{madaule2015} and \cite[Theorem 5.14]{shi2015})
and also when it ends up at a position of order $\sqrt{n}$ (see \cite[Lemma 4.7]{aidekonshi2014}).
Here we have to deal with the intermediate case where the spine end up far above $\frac{3}{2} \log n$ and far below $\sqrt{n}$.

In order to state the peeling lemma in a general setting, we introduce some notation.
For $b,u,v \in \R$ and $n \in \N$, we set
\begin{align*}
A_n^{b,u,v}
\coloneqq
\left\{
	\abs{z} = n :
	V(z) \in [v+b, v+b+1), 
	\underline{V}(z) \geq -u, 
	\min_{\lfloor n/2 \rfloor \leq j \leq n} V(z_j) \geq v
\right\}.
\end{align*}
We consider
\begin{align*}
a_i^{(n)}
\coloneqq
\left\{
\begin{array}{ll}
-u 
& \text{if } 0 \leq i < \lfloor n/2 \rfloor, \\
v
& \text{if } \lfloor n/2 \rfloor \leq i \leq n,
\end{array}
\right.
\quad \text{ and } \quad
\ell_i^{(n)}
\coloneqq
\left\{
\begin{array}{ll}
i^{1/7} 
& \text{if } 0 \leq i < \lfloor n/2 \rfloor, \\
(n-i)^{1/7}
& \text{if } \lfloor n/2 \rfloor \leq i \leq n,
\end{array}
\right.
\end{align*}
and the following set
\[
B_n^\rho
\coloneqq
\Biggl\{
	\abs{z} = n :
	\Forall j \in \llbracket 0, n-1 \rrbracket,
	\sum_{y \in \Omega(z_{j+1})} 
	\left(1 + [V(y)-a_j^{(n)}]_+ \right)
	\e^{-[V(y)-a_j^{(n)}]} 
	\leq 
	\rho
	\e^{- \ell_j^{(n)}} 
\Biggr\}.
\]
We can now state our version of the peeling lemma, which covers the case where the spine ends up far below $\sqrt{n}$ and is therefore more general than the peeling lemmas in \cite{aidekon2013,madaule2015,shi2015}.
\begin{lem}[Peeling lemma] \label{lemma peeling}
For all $\varepsilon >0$, there exist $\rho > 0$ and $n_0 \in \N$ such that, for all $n \geq n_0$, $b \in \R_+$ and $u,v \in [0, n^{1/8}]$,
\begin{align*}
\Q \left( w_n \in A_n^{b,u,v} \cap (B_n^\rho)^c \right) 
\leq
\varepsilon \frac{R(u) R^-(b)}{n^{3/2}}.
\end{align*}
\end{lem}
\begin{rem}
We present here the peeling lemma in terms of probability measure $\Q$, because it simplifies somehow the proof (for example, under $\Q$, the reproduction law along the spine does not depend on the position), but it is a direct consequence that,
for all $\varepsilon >0$, there exist $\rho > 0$ and $n_0 \in \N$ such that, for all $n \geq n_0$, $b \in \R_+$ and $u,v \in [0, n^{1/8}]$,
\begin{align*}
\Q^{(u)} \left( w_n^{(u)} \in A_n^{b,u,v} \cap (B_n^\rho)^c \right) 
\leq
\varepsilon \frac{R(u+v+b) R^-(b)}{n^{3/2}}.
\end{align*}
\end{rem}
Before proving the peeling lemma, we first use it to show Lemma \ref{lemma from Y_n to Y_n'}.
\begin{proof}[Proof of Lemma \ref{lemma from Y_n to Y_n'}]
We set $Y_n \coloneqq Y_n(1)$ and $Y_n' \coloneqq Y_n'(1)$ and, since $F$ is bounded, it is sufficient to show that
\begin{align*}
\Eci{\Q^{(L)}}{\frac{Y_n - Y_n'}{D_n^{(L)}}}
& =
\petito{\frac{\alpha_n^2}{n^{3 \beta_n /2}}}
\end{align*}
We first change probabilities from $\Q^{(L)}$ to $\Q$: we have
\begin{align}
\Eci{\Q^{(L)}}{\frac{Y_n - Y_n'}{D_n^{(L)}}}
& =
\frac{1}{R(L)}
\Eci{\Q}{\frac{Y_n - Y_n'}{W_n}}
=
\frac{1}{R(L)}
\Eci{\Q}{
\e^{-V(w_n) / \alpha_n} 
\1_{w_n \in A_n \cap B_n^c}}. \label{ds}
\end{align}
Then, setting $u = L$ and $v = \frac{3}{2} \log n$ and cutting the segment $[\alpha^-_n, \alpha^+_n]$ in pieces of length 1, we get that \eqref{ds} is smaller than
\begin{align}
\frac{1}{R(L)}
\sum_{i = \lfloor \alpha_n^- \rfloor}
	^{\lceil \alpha_n^+ \rceil - 1}
\e^{-i/ \alpha_n} 
\Q \left( w_n \in A_n^{i,u,v} \cap B_n^c \right)
& \leq
\petito{\frac{1}{n^{3/2}}}
\sum_{i = \lfloor \alpha_n^- \rfloor}
	^{\lceil \alpha_n^+ \rceil - 1}
\e^{-i/ \alpha_n} 
R^-(i), \label{dt}
\end{align}
using Lemma \ref{lemma peeling} uniformly in $i$, noting that $B_n \subset B_n^{\rho_n}$ and $\rho_n \to \infty$.
Using then \eqref{equation R 2} and \eqref{eq resultat somme}, we get that \eqref{dt} is a $o(\alpha_n^2 / n^{3/2})$ and it concludes the proof of Lemma \ref{lemma from Y_n to Y_n'}.
\end{proof}

\begin{proof}[Proof of Lemma \ref{lemma peeling}]
By Lemma \ref{lemme trajectoire supérieure à i^alpha} and \eqref{equation R 2}, it exists $\mu >0$ such that, for all $b,u \in \R_+$, $v \in \R$ and $n\in\N^*$, we have
\begin{align*}
\Q \left( w_n \in A_n^{b,u,v},
	\Exists i \in \llbracket 0,n \rrbracket:
	V(w_i) < a_i^{(n)} + 2\ell_i^{(n)} - \mu
	\right) 
& \leq
\frac{\varepsilon}{2}
\frac{R(u) R^-(b)}{n^{3/2}}.
\end{align*}
Thus, it is sufficient to show that
\begin{align} \label{ga}
\Q \left( w_n \in A_n^{b,u,v} \cap B_n^c,
	\Forall j \in \llbracket 0,n \rrbracket,
	V(w_j) \geq a_j^{(n)} + 2\ell_j^{(n)} - \mu
	\right) 
& \leq
\frac{\varepsilon}{2}
\frac{R(u) R^-(b)}{n^{3/2}},
\end{align}
for $n$ large enough, $b \in \R_+$ and $u,v \in [0, n^{1/8}]$.
Therefore, we now prove \eqref{ga}.
We first set, for $0 \leq i \leq n-1$,
\[
B_{n,i}^\rho
\coloneqq
\Biggl\{
	\abs{z} = n :
	\sum_{y \in \Omega(z_{i+1})} 
	\left(1 + [V(y)-a_i^{(n)}]_+ \right)
	\e^{-[V(y)-a_i^{(n)}]} 
	\leq 
	\rho
	\e^{- \ell_i^{(n)}} 
\Biggr\}.
\]
Since for all $u,v \in \R$, $1 + (u+v)_+ \leq (1+ u_+) (1 + v_+)$, we have
\begin{align*}
\sum_{y \in \Omega(z_{i+1})} 
\left(1 + [V(y)-a_i^{(n)}]_+ \right)
\e^{-[V(y)-a_i^{(n)}]} 
\leq
\left(1 + [V(w_i)-a_i^{(n)}]_+ \right)
\e^{-[V(w_i)-a_i^{(n)}]}
\Theta (w_{i+1}),
\end{align*}
where we set, for $x \in \T$, by noting $\overleftarrow{x}$ the parent of $x$ and  $\Delta(x) \coloneqq V(x) - V(\overleftarrow{x})$,
\[
\Theta (x) 
\coloneqq
\sum_{y \in \Omega(x)} 
(1 + \Delta(y)_+)
\e^{-\Delta(y)}.
\]
Thus, we have, on event $\{ w_n \notin B_{n,i}^\rho \} \cap \{ V(w_i) \geq a_i^{(n)} + 2\ell_i^{(n)} - \mu \}$,
\begin{align*}
\left(1 + [V(w_i)-a_i^{(n)}]_+ \right)
\e^{-[V(w_i)-a_i^{(n)}]}
\Theta (w_{i+1})
\geq
\rho \e^{- \ell_i^{(n)}}
\geq
\rho \e^{- [V(w_i) - a_i^{(n)} + \mu]/2}
\end{align*}
and, if we are moreover on event $\{ w_n \in A_n^{b,u,v} \}$ so that $V(w_i) \geq a_i^{(n)}$, it implies that
\begin{align*}
\Theta (w_{i+1})
\geq
\rho \e^{- \mu/2} 
\frac{\e^{[V(w_i)-a_i^{(n)}]/2}}{1 + [V(w_i)-a_i^{(n)}]_+}
\geq
c_{19} \rho \e^{- \mu/2} 
\e^{[V(w_i)-a_i^{(n)}]/3},
\end{align*}
where $c_{19} \coloneqq \inf_{u\geq0} \e^{u/6}/(1+u) > 0$.
Therefore, we get
\begin{align}
\Q \left( w_n \in A_n^{b,u,v} \cap (B_n^\rho)^c,
	\Forall j \in \llbracket 0,n \rrbracket,
	V(w_j) \geq a_j^{(n)} + 2 \ell_j^{(n)} - \mu
	\right)
& \leq
\sum_{i = 0}^{n-1} q_i, \label{gb}
\end{align}
where we set, for $0 \leq i \leq n-1$,
\begin{align*}
q_i
\coloneqq
\Q \left( 
	w_n \in A_n^{b,u,v}, 
	\Forall j \in \llbracket 0,n \rrbracket, 
		V(w_j) \geq a_j^{(n)} + 2 \ell_j^{(n)} - \mu,
	\Theta (w_{i+1}) > \tilde{\rho} \e^{[V(w_i)-a_i^{(n)}]/3}
	\right)
\end{align*}
and $\tilde{\rho} \coloneqq c_{19} \rho \e^{- \mu/2}$.
From now, we choose $\rho$ such that $\tilde{\rho} \geq \e$.
We first consider the case $0 \leq i < \lfloor n/2 \rfloor$.
Since $a_i^{(n)} = -u$, we have
\begin{align*}
q_i
& \leq
\Eci{\Q}{
1_{\underline{V}(w_i) \geq -u,
	\Theta (w_{i+1}) > \tilde{\rho} \e^{[V(w_i) + u]/3} }
G_i( V(w_{i+1}))
},
\end{align*}
where we set, for $x \in \R$,
\begin{align*}
G_i(x)
& \coloneqq
\Q_x \left( 
	\underline{V}(w_{n-i-1}) \geq -u,
	V(w_{n-i-1}) -v \in [b, b+1),
	\min_{\lfloor n/2 \rfloor-i-1 \leq j \leq n-i-1} V(w_j) \geq v
	\right) \\
& \leq
\Pp{
	\underline{S}_{n-i-1} \geq -u-x,
	S_{n-i-1} - (v-x) \in [b, b+1),
	\min_{\lfloor (n-i-1)/2 \rfloor \leq j \leq n-i-1} S_j \geq v-x
	} \\
& \leq
c_5 
\frac{(1+ x+u)2(b+2)}{(n-i)^{3/2}} \\
& \leq
c_{20} 
\frac{(1+ x+u)R^-(b)}{n^{3/2}},
\end{align*}
using successively Proposition \ref{prop lyon's change of measure} (ii), \eqref{eq ballot theorem and barrier} and \eqref{equation R 2}.
We thus get
\begin{align*}
q_i
& \leq
c_{20}
\frac{R^-(b)}{n^{3/2}} 
\Eci{\Q}{
1_{\underline{V}(w_i) \geq -u,
	\Theta (w_{i+1}) > \tilde{\rho} \e^{[V(w_i) +u]/3} }
\left(
1 + V(w_{i+1}) + u
\right)
} \\
& \leq
c_{20}
\frac{R^-(b)}{n^{3/2}} 
\Eci{\Q}{
1_{\underline{V}(w_i) \geq -u,
	V(w_i) +u
		< 3 \log \frac{\Theta (w_{i+1})}{\tilde{\rho}} }
\left(
1 + V(w_i) +u + \Delta (w_{i+1})
\right)
} \\
& \leq
c_{20}
\frac{R^-(b)}{n^{3/2}} 
\Eci{\Q}{
1_{\underline{V}(w_i) \geq -u,
	V(w_i) +u
		< 3 \log \frac{\Theta (w_{i+1})}{\tilde{\rho}} }
\left[
1 + 3 \log_+ \frac{\Theta (w_{i+1})}{\tilde{\rho}} + (\Delta (w_{i+1}))_+
\right]
}.
\end{align*}
But, under $\Q$, $(\Theta (w_{i+1}),\Delta (w_{i+1}))$ is independent of $(V(w_j))_{0\leq j \leq i}$ and has moreover the same law as $(X+ \widetilde{X},V(w_1))$, where $X$ and $\widetilde{X}$ are defined in \eqref{equation def X et tildeX}.
Therefore, we get, by integrating first on $(V(w_j))_{0\leq j \leq i}$,
\begin{align}
q_i
& \leq
c_{20}
\frac{R^-(b)}{n^{3/2}} 
\Eci{\Q}{
	F_i(X + \widetilde{X}) 
	\left(
	1 + 3 \log_+ \frac{X+ \widetilde{X}}{\tilde{\rho}} + V(w_1)_+]
	\right)}, \label{gc}
\end{align}
where we set, for $x>0$,
\begin{align*}
F_i(x) 
& \coloneqq
\Q
\left(
\underline{V}(w_i) \geq -u,
V(w_i) +u 
	< 3 \log \frac{x}{\tilde{\rho}} 
\right)  =
\1_{0 < \log \frac{x}{\tilde{\rho}}} 
\Pp{\underline{S}_i \geq -u, 
	S_i < - u + 3 \log_+ x},
\end{align*}
by using Proposition \ref{prop lyon's change of measure} (ii) and that $\tilde{\rho} \geq 1$.
Then, applying \eqref{equation MA somme de probas}, we have
\begin{align*}
\sum_{i = k_n}^{\lfloor n/2 \rfloor - 1}
F_i(x) 
& \leq
\1_{x > \tilde{\rho}} 
c_6 
\left( 1 + 3\log_+ x \right) 
(1+ u)
\leq
3 \frac{c_6}{c_1}
R(u)
(1 + \log_+ x)
\1_{x > \tilde{\rho}}
\end{align*}
using also \eqref{equation R 2}.
Coming back to \eqref{gc} and noting that, since $X+ \widetilde{X} > \tilde{\rho} \geq \e$, we have $\log_+ (X+ \widetilde{X}) \geq 1$, this gives 
\begin{align}
\sum_{i = k_n}^{\lfloor n/2 \rfloor - 1} q_i
& \leq
c_{21} 
\frac{R(u) R^-(b)}{n^{3/2}} 
\Eci{\Q}{
	\1_{X + \widetilde{X} > \tilde{\rho}}
	\log_+ (X+ \widetilde{X})
	\left(
	\log_+ (X+ \widetilde{X}) + V(w_1)_+]
	\right)} \label{gd}
\end{align}
and concludes the case $0 \leq i < \lfloor n/2 \rfloor$.

We now consider the case $\lfloor n/2 \rfloor \leq i < n$, so $a_i^{(n)} = v$.
On event $\{ w_n \in A_n^{b,u,v}, \Theta (w_{i+1}) > \tilde{\rho} \e^{[V(w_i)-v]/3} \}$, 
we have $V(w_i) = V(w_n) - (V(w_n) - V(w_i)) \in [v + b - (V(w_n) - V(w_i)), v + b+1 - (V(w_n) - V(w_i))]$,
$\Theta(w_{i+1}) > \tilde{\rho} \vee \tilde{\rho} \e^{[V(w_i)-V(w_n) + b] / 3}$ and, for all $i+1 \leq j \leq n$, $V(w_n) - V(w_j) \leq b+1$.
Moreover, we set
\begin{align*}
\tilde{a}_j^{(n)}
\coloneqq
\left\{
\begin{array}{ll}
-u 
& \text{if } 0 \leq j < \lfloor n/4 \rfloor, \\
v
& \text{if } \lfloor n/4 \rfloor \leq j \leq n,
\end{array}
\right.
\end{align*}
and note that, on event $\{ w_n \in A_n^{b,u,v}, \Forall j \in \llbracket 0,n \rrbracket, V(w_j) \geq a_j^{(n)} + 2 \ell_j^{(n)} - \mu \}$, for $n$ large enough, we have 
$\Forall j \in \llbracket \lfloor n/4 \rfloor,\lfloor n/2 \rfloor - 1 \rrbracket, V(w_j) \geq v = \tilde{a}_j^{(n)}$, 
because $\ell_j^{(n)} \geq \lfloor n/4 \rfloor^{1/7}$ and $u+v \leq 2 n^{1/8}$.
Thus, we get
\begin{align}
\begin{split}
q_i
& \leq
\Q
\Bigl( 
\Forall j \in \llbracket 1,i \rrbracket, V(w_j) \geq \tilde{a}^{(n)}_j,
	V(w_i) - \left( v - (V(w_n) - V(w_i)) \right) \in [b, b+1]
\\
& \relphantom{\leq} \hphantom{\Q \Bigl( {}}
\Theta (w_{i+1}) > \tilde{\rho} \vee \tilde{\rho} \e^{[V(w_i)-V(w_n) + b] / 3},
	\Forall j \in \llbracket i+1, n \rrbracket, V(w_n) - V(w_j) \leq b+1
\Bigr). \label{ge}
\end{split}
\end{align}
Under $\Q$, $(\Theta (w_{i+1}), (V(w_n) - V(w_j))_{i \leq j \leq n})$ is independent of $(V(w_j))_{0 \leq j \leq i}$, so the right-hand side of \eqref{ge} is equal to
\begin{align}
\Eci{\Q}
{
H_i(V(w_n) - V(w_i))
\1_{\Theta (w_{i+1}) > \tilde{\rho} \vee \tilde{\rho} \e^{[V(w_i)-V(w_n) + b] / 3},
	\Forall j \in \llbracket i+1, n \rrbracket, V(w_n) - V(w_j) \leq b+1}
}, \label{gf}
\end{align}
where we set, for $x \in \R$,
\begin{align*}
H_i(x)
& \coloneqq
\Q
\left( 
\underline{V}(w_i) \geq -u,
\min_{\lfloor n/4 \rfloor \leq j \leq i} V(w_j) \geq v,
V(w_i) - (v - x) \in [b, b+1] 
\right) \\
& \leq
\Pp{ 
\underline{S}_i \geq -u,
\min_{\lfloor i/2 \rfloor \leq j \leq i} S_j \geq v,
S_i - v \in [b-x, b-x+1]
} \\
& \leq
c_{22} \frac{R(u)}{n^{3/2}} (1 + b - x),
\end{align*}
using successively Proposition \ref{prop lyon's change of measure} (ii), \eqref{eq ballot theorem and barrier} and \eqref{equation R 2} as before.
On the event $\{ \Theta (w_{i+1}) > \tilde{\rho} \e^{[V(w_i)-V(w_n) + b] / 3}\}$, using that $\tilde{\rho} \geq 1$, we have $b - (V(w_n) - V(w_i)) \leq 3 \log_+ \Theta (w_{i+1})$ and also $V(w_n) -V(w_{i+1}) \geq b - 3 \log_+ \Theta (w_{i+1}) - \Delta(w_{i+1})_+$,
therefore, \eqref{gf} is smaller than
\begin{align}
\begin{split}
& c_{22} \frac{R(u)}{n^{3/2}} 
\E_{\Q}
\Bigl[
\left( 1 + 3 \log_+ \Theta (w_{i+1}) \right) 
\1_{\Theta (w_{i+1}) > \tilde{\rho}} \\
& \hphantom{c_{22} \frac{R(u)}{n^{3/2}} \E_{\Q} \Bigl[}
{} \times
\1_{V(w_n) -V(w_{i+1}) \geq b - 3 \log_+ \Theta (w_{i+1}) 
								- \Delta(w_{i+1})_+,
	\Forall j \in \llbracket i+1, n \rrbracket, V(w_n) - V(w_j) \leq b+1}
\Bigr].
\end{split} \label{gg}
\end{align}
Note then that, under $\Q$, $(\Theta (w_{i+1}),\Delta (w_{i+1}))$ is independent of $(V(w_n) - V(w_j))_{i+1 \leq j \leq n}$ and has the same law as $(X + \widetilde{X},V(w_1))$, so \eqref{gg} is equal to
\begin{align}
& c_{22} \frac{R(u)}{n^{3/2}} 
\Eci{\Q}
{
\left( 1 + 3 \log_+ (X + \widetilde{X}) \right)
\1_{X + \widetilde{X} > \tilde{\rho}}
\Gamma_i \left( 3 \log_+ (X + \widetilde{X}) + V(w_1)_+ \right)
}, \label{gh}
\end{align}
where we set, for $x \geq 0$,
\begin{align*}
\Gamma_i (x)
& \coloneqq
\Q
\left(	
V(w_n) -V(w_{i+1}) \geq b - x,
\Forall j \in \llbracket i+1, n \rrbracket, V(w_n) - V(w_j) \leq b + 1
\right) \\
& =
\Pp{
S^-_{n-i-1} \leq (x +1) -(b+1),
\underline{S}^-_{n-i-1} \geq -(b+1)
}, 
\end{align*}
by applying Proposition \ref{prop lyon's change of measure} (ii) and then reversing time.
Thus, using \eqref{equation MA somme de probas} and \eqref{equation R 2}, we get
\begin{align*}
\sum_{i = \lfloor n/2 \rfloor}^{n-1}
\Gamma_i (x)
& \leq 
c_6^- (1+x + 1) (1+ b + 1)
\leq
4\frac{c_6^-}{c_1^-} (1+x) R^-(b).
\end{align*}
On event $\{X+ \widetilde{X} > \tilde{\rho}\}$, we have $\log_+ (X+ \widetilde{X}) \geq 1$ so, coming back to (\ref{gh}), we get
\begin{align}
\sum_{i = \lfloor n/2 \rfloor}^{n-1} q_i
& \leq
c_{23}
\frac{R(u) R^-(b)}{n^{3/2}}
\Eci{\Q}
{
\log_+ (X + \widetilde{X})
\1_{X + \widetilde{X} > \tilde{\rho}}
\left(\log_+ (X + \widetilde{X}) + V(w_1)_+ \right)
}. \label{gi}
\end{align}
Finally, using \eqref{gd}, \eqref{gi} and that $\Ecsqi{\Q}{V(w_1)_+}{\sF_1} = \widetilde{X} / X$ by Proposition \ref{prop lyon's change of measure} (i), we get
\begin{align}
\sum_{i = 0}^{n-1} q_i
& \leq
(c_{21} + c_{23})
\frac{R(u) R^-(b)}{n^{3/2}}
\Ec{ \left(X \log_+^2 (X+ \widetilde{X})
	+ \widetilde{X} \log_+ (X+ \widetilde{X}) \right) 
	\1_{X + \widetilde{X} > \tilde{\rho}}}. \label{gj}
\end{align}
Using \eqref{hypothese 3}, we can choose $\rho$ large enough such that the expectation in \eqref{gj} is smaller than $\varepsilon / 2 (c_{21} + c_{23})$ and, recalling \eqref{gb}, it proves \eqref{ga} and concludes the proof of Lemma \ref{lemma peeling}.
\end{proof}

\subsection{Second moment of \texorpdfstring{$Y_n'(F_n)$}{Yn'(Fn)}}
\label{subsection second moment}

For $F \in \cC(\cD([0,1]))$, by decomposing along the spine, $Y_n'(F)$ is equal to
\begin{align*}
\sum_{i=0}^{n-1}
\sum_{y \in \Omega(w_{i+1}^{(L)})}
\sum_{\abs{z}=n, z \geq y}
\e^{- \beta_n V(z)}
F(\bfV (z))
\1_{z \in A_n \cap B_n} 
+
\e^{- \beta_n V(w_n^{(L)})}
F(\bfV (w_n^{(L)}))
\1_{w_n^{(L)} \in A_n\cap B_n}.
\end{align*}
We cut this sum in two pieces, depending on whether the lineage of the considered particle $z$ splits off from the spine's lineage before or after time $k_n$:
\begin{align*}
Y_n'^{[0,k_n)} (F)
& \coloneqq
\sum_{i=0}^{k_n-1}
\sum_{y \in \Omega(w_{i+1}^{(L)})}
\sum_{\abs{z}=n, z \geq y}
\e^{- \beta_n V(z)}
F(\bfV (z))
\1_{z \in A_n \cap B_n} \\
Y_n'^{[k_n,n]} (F)
& \coloneqq
Y_n'(F) - Y_n'^{[0,k_n)} (F).
\end{align*}
We define in the same way $D_n^{(L),[0,k_n)}$ and $D_n^{(L),[k_n,n]}$. 
Then, by A\"idékon and Shi \cite[Lemma 4.7]{aidekonshi2014}, since $(\log n)^6 \ll k_n \ll \sqrt{n}$, we have
\begin{align} \label{eq lemma 4.7 aidekon and shi}
\inf_{u \in [k_n^{1/3}, k_n]} 
\Q^{(L)} 
\left( 
D_n^{(L),[k_n,n]} \leq n^{-2}
\mathrel{}\middle|\mathrel{}
V(w^{(L)}_{k_n}) = u
\right)
\xrightarrow[n \to \infty]{}
1.
\end{align}

\begin{proof}[Proof of Lemma \ref{lemma second moment}]
First note that, using Proposition \ref{prop biggins and kyprianou's change of measure} (i),
\begin{align*}
\frac{Y_n'(F_n)}{D_n^{(L)}}
& =
\sum_{\abs{z} = n} 
\Q^{(L)} \left( w_n^{(L)} = z  \mathrel{}\middle|\mathrel{} \sF_n \right)
\frac{\e^{-V(z) / \alpha_n} }{R_L(V(z))}
F_n(\bfV (z))
\1_{z \in A_n \cap B_n} \\
& =
\Ecsqi{\Q^{(L)}}{
\frac{\e^{-V(w_n^{(L)}) / \alpha_n} }{R_L(V(w_n^{(L)}))}
F_n(\bfV (w_n^{(L)}))
\1_{w_n^{(L)} \in A_n \cap B_n} 
}{\sF_n}
\end{align*}
and, thus, we have
\begin{align*}
\Eci{\Q^{(L)}}{\left( \frac{Y_n'(F_n)}{D_n^{(L)}} \right)^2}
& =
\Eci{\Q^{(L)}}
{\frac{Y_n'(F_n)}{D_n^{(L)}} 
\frac{\e^{-V(w_n^{(L)}) / \alpha_n} }{R_L(V(w_n^{(L)}))}
F_n(\bfV (w_n^{(L)}))
\1_{w_n^{(L)} \in A_n \cap B_n} } \\
& \eqqcolon
\E_{\Q^{(L)}}^{[0,k_n)}
+
\E_{\Q^{(L)}}^{[k_n,n]},
\end{align*}
by splitting $Y_n'(F_n) = Y_n'^{[0,k_n)}(F_n) + Y_n'^{[k_n,n]}(F_n)$. 
The first part will give the right order and constant and the second part will be negligible.
Recall that $F$ is assumed to be nonnegative.

We begin by bounding $\E_{\Q^{(L)}}^{[0,k_n)}$.
Using $D_n^{(L)} \geq D_n^{(L),[0,k_n)}$ and $\1_{w_n^{(L)} \in B_n} \leq 1$,
we get
\begin{align}
\E_{\Q^{(L)}}^{[0,k_n)}
& \leq
\Eci{\Q^{(L)}}
{\frac{Y_n'^{[0,k_n)}(F_n)}{D_n^{(L),[0,k_n)}} 
\frac{\e^{-V(w_n^{(L)}) / \alpha_n} }{R_L(V(w_n^{(L)}))}
F_n (\bfV (w_n^{(L)}))
\1_{w_n^{(L)} \in A_n} } \nonumber \\
& =
\Eci{\Q^{(L)}}
{\frac{Y_n'^{[0,k_n)}(F_n)}{D_n^{(L),[0,k_n)}} 
\varphi \left( V(w_{k_n}^{(L)}) \right) 
\1_{V(w_{k_n}^{(L)}) \in [k_n^{1/3}, k_n]} }, \label{ed}
\end{align}
where we set, for $b \in [k_n^{1/3}, k_n]$ and with $m\coloneqq n-k_n$,
\begin{align*}
\varphi (b)
& \coloneqq
\Eci{\Q^{(L)}_b}{
\frac{\e^{-V(w_m^{(L)}) / \alpha_n} }{R_L(V(w_m^{(L)}))}
F \left( 
\bfV (w_m^{(L)}) - \frac{b}{\sigma \sqrt{m}}
\right) 
\1_{\Forall 0 \leq i \leq n-k_n, V(w_i^{(L)}) \in I_{n,i+k_n}}
} \\
& \leq
\frac{1}{R_L(b)} 
\Ec{
\e^{-(S_m + b) / \alpha_n}
F(\bfS^{(m)})
\1_{\underline{S}_m \geq -(L+b),
\min_{\lfloor \lambda m \rfloor \leq j \leq m} S_j \geq \frac{3}{2} \log n - b,
S_m - (\frac{3}{2} \log n -b) \in [\alpha^-_n, \alpha^+_n]
}},
\end{align*}
for some fixed $\lambda \in (1/2, 1)$ and $n$ large enough, by applying Proposition \ref{prop biggins and kyprianou's change of measure} (ii). 
Then, proceeding in the same way as for the lower bound of $\psi (b)$ in the proof of Lemma \ref{lemma first moment} (but with a sum on $i$ from $\lfloor \alpha^-_n / h \rfloor$ to $\lceil \alpha^+_n / h \rceil -1$), we get the upper bound
\begin{align} \label{el}
\varphi (b)
& \leq
\frac{1}{R_L(b)} 
(L+b)
\sqrt{\frac{2}{\pi}} 
\frac{1}{\sigma^3}
\frac{\alpha_n^2}{n^{3 \beta_n/2}}
(\Ec{F(\mathfrak{e})} +\petito{1})
=
\frac{\theta}{\sigma^2}
\frac{\alpha_n^2}{n^{3 \beta_n/2}}
(\Ec{F(\mathfrak{e})} +\petito{1}),
\end{align}
uniformly in $b \in [k_n^{1/3}, k_n]$, using \eqref{equation R 1} and \eqref{equation lien theta c_0} for the last equality.
Coming back to \eqref{ed}, we showed that
\begin{align}
\E_{\Q^{(L)}}^{[0,k_n)}
& \leq
\frac{\theta}{\sigma^2}
\frac{\alpha_n^2}{n^{3 \beta_n/2}}
(\Ec{F(\mathfrak{e})} +\petito{1})
\Eci{\Q^{(L)}}
{\frac{Y_n'^{[0,k_n)}(F_n)}{D_n^{(L),[0,k_n)}} 
\1_{V(w_{k_n}^{(L)}) \in [k_n^{1/3}, k_n]} }, \label{ee}
\end{align}
and thus we now want to bound the expectation in \eqref{ee}.
We proceed in a way similar to the proof of Lemma 4.5 of A\"idekon et Shi \cite{aidekonshi2014}, by introducing the event $\{ D_n^{(L),[k_n,n]} \leq n^{-2} \}$: 
\begin{align*}
& \Eci{\Q^{(L)}}
{\frac{Y_n'^{[0,k_n)}(F_n)}{D_n^{(L),[0,k_n)}} 
	\1_{D_n^{(L),[k_n,n]} \leq n^{-2}} } \\
& =
\Eci{\Q^{(L)}}
{\frac{Y_n'^{[0,k_n)}(F_n)}{D_n^{(L),[0,k_n)}} 
\Q^{(L)} 
\left( 
D_n^{(L),[k_n,n]} \leq n^{-2}
\mathrel{}\middle|\mathrel{}
V(w^{(L)}_{k_n}) = u
\right)
} \\
& \geq
\Eci{\Q^{(L)}}
{\frac{Y_n'^{[0,k_n)}(F_n)}{D_n^{(L),[0,k_n)}} 
\1_{V(w_{k_n}^{(L)}) \in [k_n^{1/3}, k_n]} 
} 
\inf_{u \in [k_n^{1/3}, k_n]} 
\Q^{(L)} 
\left( 
D_n^{(L),[k_n,n]} \leq n^{-2}
\mathrel{}\middle|\mathrel{}
V(w^{(L)}_{k_n}) = u
\right).
\end{align*}
Thus, applying \eqref{eq lemma 4.7 aidekon and shi}, we get that the expectation in \eqref{ee} is smaller than
\begin{align}
& (1 + \petito{1}) 
\Eci{\Q^{(L)}}
{\frac{Y_n'^{[0,k_n)}(F_n)}{D_n^{(L),[0,k_n)}} 
\1_{D_n^{(L),[k_n,n]} \leq n^{-2} }} 
\nonumber \\
& \leq
(1 + \petito{1})
\Biggl(
\Eci{\Q^{(L)}}
{\frac{Y_n'^{[0,k_n)}(F_n)}{D_n^{(L),[0,k_n)}} 
\1_{D_n^{(L)} \geq n^{-7/4},
	D_n^{(L),[k_n,n]} \leq n^{-2}} }
+
\norme{F}
\Q^{(L)} \left( D_n^{(L)} < n^{-7/4} \right) 
\Biggr) \nonumber \\
& \leq
(1 + \petito{1})
\Biggl(
\Eci{\Q^{(L)}}
{\frac{Y_n'^{[0,k_n)}(F_n)}{(1 - n^{-1/4}) D_n^{(L)}} } 
+
\norme{F}
n^{-7/4} \Eci{\Q^{(L)}}{\frac{1}{D_n^{(L)}}}
\Biggr), \label{ef}
\end{align}
using that $Y_n'^{[0,k_n)}(F_n) / D_n^{(L),[0,k_n)} \leq \norme{F}$ and
applying the Markov inequality for the second term.
Using that $Y_n'^{[0,k_n)} \leq Y_n' \leq W^{(L,K)}_{n,\beta_n}$, and then Lemmas \ref{lemma first moment} and \ref{lemma from F to F_n}, we get that \eqref{ef} is smaller than 
\begin{align*}
(1 + \petito{1})
\Biggl(
(\Ec{F(\mathfrak{e})} + \petito{1}) 
\frac{\theta}{\sigma^2}
\frac{\alpha_n^2}{n^{3\beta_n /2}}
+
\norme{F}
\frac{n^{-7/4}}{R(L)}
\Biggr)
=
(\Ec{F(\mathfrak{e})} + \petito{1}) 
\frac{\theta}{\sigma^2}
\frac{\alpha_n^2}{n^{3\beta_n /2}},
\end{align*}
because $3 \beta_n/2 \leq 7/4$ for $n$ large enough.
Thus, coming back to \eqref{ee}, we proved that
\begin{align}
\E_{\Q^{(L)}}^{[0,k_n)}
& \leq
\left(
(\Ec{F(\mathfrak{e})}+\petito{1})
\frac{\theta}{\sigma^2}
\frac{\alpha_n^2}{n^{3 \beta_n/2}}
\right)^2. \label{eo}
\end{align}

We now want to show that $\E_{\Q^{(L)}}^{[k_n,n]} = o((\alpha_n^2/n^{3 \beta_n/2})^2)$ and, by bounding $F \leq \norme{F}$, it is sufficient to deal with the case $F \equiv 1$.
Using that $D_n^{(L),[0,k_n)} \leq D_n^{(L)}$, $\1_{z \in B_n} \leq 1$ and breaking down $Y_n'^{[k_n,n]}(1)$ along the spine, we first have that $\E_{\Q^{(L)}}^{[k_n,n]}$ is smaller than
\begin{align}
\begin{split}
& 
\sum_{i=k_n}^{n-1}
\E_{\Q^{(L)}}
\Biggl[
\frac{1}{D_n^{(L),[0,k_n)}} 
\frac{\e^{-V(w_n^{(L)}) / \alpha_n} }{R_L(V(w_n^{(L)}))}
\1_{w_n^{(L)} \in A_n \cap B_n}
\sum_{y \in \Omega(w_{i+1}^{(L)})}
\sum_{\abs{z}=n, z \geq y}
\e^{-\beta_n V(z)}
\1_{z \in A_n}
\Biggr] \\
& {} +
\E_{\Q^{(L)}}
\Biggl[
\frac{1}{D_n^{(L),[0,k_n)}} 
\frac{\e^{-V(w_n^{(L)}) / \alpha_n} }{R_L(V(w_n^{(L)}))}
\1_{w_n^{(L)} \in A_n \cap B_n}
\e^{-\beta_n V(w_n^{(L)})}
\Biggr]. \label{ea}
\end{split}
\end{align}
Noting that, on the event $\{ z \in A_n \}$, we have $\e^{-\beta_n V(z)} \leq \e^{-V(z)} n^{-3/2\alpha_n}$, the first term in \eqref{ea} is smaller than
\begin{align}
& \sum_{i=k_n}^{n-1}
\E_{\Q^{(L)}}
\Biggl[
\frac{1}{D_n^{(L),[0,k_n)}} 
\frac{\e^{-V(w_n^{(L)}) / \alpha_n} }{R_L(V(w_n^{(L)}))}
\1_{w_n^{(L)} \in A_n \cap B_n}
\sum_{y \in \Omega(w_{i+1}^{(L)})}
\sum_{\abs{z}=n, z \geq y}
\e^{-V(z)} n^{-3/2\alpha_n}
\Biggr] \nonumber \\
& =
n^{-3/2\alpha_n}
\sum_{i=k_n}^{n-1}
\E_{\Q^{(L)}}
\Biggl[
\frac{1}{D_n^{(L),[0,k_n)}} 
\frac{\e^{-V(w_n^{(L)}) / \alpha_n} }{R_L(V(w_n^{(L)}))}
\1_{w_n^{(L)} \in A_n \cap B_n}
\sum_{y \in \Omega(w_{i+1}^{(L)})}
\e^{-V(y)}
\Biggr], \label{ep}
\end{align}
by conditioning with respect to $\sG_\infty \coloneqq \sigma ( V(w_i^{(L)}), V(y), y \in \Omega(w_{i+1}^{(L)}), i \in \N)$ and noting that, given $\sG_\infty$, $D_n^{(L),[0,k_n)}$ is independent of $(V(z), \abs{z}=n, z \geq y, y \in \Omega(w_{i+1}^{(L)}))$.
Noting that we are on the event $\{ w_n^{(L)} \in B_n \}$, \eqref{ep} is smaller than
\begin{align}
& n^{-3/2\alpha_n}
\sum_{i=k_n}^{n-1}
\Eci{\Q^{(L)}}
{
\frac{1}{D_n^{(L),[0,k_n)}} 
\frac{\e^{-V(w_n^{(L)}) / \alpha_n} }{R_L(V(w_n^{(L)}))}
\1_{w_n^{(L)} \in A_n \cap B_n}
\rho_n
\e^{- a_i^{(n)} - \ell_i^{(n)}} 
} \nonumber \\
& \leq
\frac{\rho_n}{n^{3/2\alpha_n}}
\left(
\sum_{i=k_n}^{\lfloor n/2 \rfloor - 1}
\e^{L - i^{1/7}} 
+
\sum_{i=\lfloor n/2 \rfloor}^{n-1}
\frac{\e^{- (n-i)^{1/7}} }{n^{3/2}}
\right)
\Eci{\Q^{(L)}}
{
\frac{1}{D_n^{(L),[0,k_n)}} 
\frac{\e^{-V(w_n^{(L)}) / \alpha_n} }{R_L(V(w_n^{(L)}))}
\1_{w_n^{(L)} \in A_n}
} \nonumber \\
& \leq
c_{24} \e^L
\frac{\rho_n}{n^{3 \beta_n/2}}
(1+\petito{1})
\Eci{\Q^{(L)}}
{
\frac{1}{D_n^{(L),[0,k_n)}} 
\frac{\e^{-V(w_n^{(L)}) / \alpha_n} }{R_L(V(w_n^{(L)}))}
\1_{w_n^{(L)} \in A_n}
}, \label{eq}
\end{align}
because $k_n^{1/7} \gg \log n$ and therefore the sum for $i \in \llbracket k_n, \lfloor n/2 \rfloor - 1 \rrbracket$ is a $o(n^{-3/2})$.
For the second term in \eqref{ea},
we use that, on the event $\{ w_n^{(L)} \in A_n \}$, $\e^{-\beta_n V(w_n^{(L)})} \leq n^{-3\beta_n / 2}$, and, thus, combining with \eqref{eq}, we get
\begin{align*}
\E_{\Q^{(L)}}^{[k_n,n]}
& \leq
c_{24} \e^L
\frac{\rho_n}{n^{3 \beta_n/2}}
(1+\petito{1})
\Eci{\Q^{(L)}}
{
\frac{1}{D_n^{(L),[0,k_n)}} 
\frac{\e^{-V(w_n^{(L)}) / \alpha_n} }{R_L(V(w_n^{(L)}))}
\1_{w_n^{(L)} \in A_n}
} \\
& =
c_{24} \e^L
\frac{\rho_n}{n^{3 \beta_n/2}}
(1+\petito{1})
\Eci{\Q^{(L)}}
{
\frac{1}{D_n^{(L),[0,k_n)}} 
\varphi \left( V(w_{k_n}^{(L)}) \right) 
\1_{V(w_{k_n}^{(L)}) \in [k_n^{1/3}, k_n]}
},
\end{align*}
where the function $\varphi$ has been defined previously in the proof.
Using \eqref{el} again, we get
\begin{align} \label{er}
\E_{\Q^{(L)}}^{[k_n,n]}
& \leq
c_{25} \e^L
\frac{\rho_n}{n^{3 \beta_n/2}}
\frac{\alpha_n^2}{n^{3 \beta_n/2}}
(1+\petito{1})
\Eci{\Q^{(L)}}
{
\frac{1}{D_n^{(L),[0,k_n)}} 
\1_{V(w_{k_n}^{(L)}) \in [k_n^{1/3}, k_n]}
}.
\end{align}
Proceeding in the same way as before by using \eqref{eq lemma 4.7 aidekon and shi} to introduce the event $\{ D_n^{(L),[k_n,n]} \leq n^{-2} \}$, the expectation in \eqref{er} is smaller than
\begin{align}
& (1 + \petito{1})
\Eci{\Q^{(L)}}
{\frac{1}{D_n^{(L),[0,k_n)}} 
\1_{D_n^{(L),[k_n,n]} \leq n^{-2}} }
\nonumber\\
& \leq
(1 + \petito{1})
\left(
\Eci{\Q^{(L)}}
{\frac{1}{D_n^{(L),[0,k_n)}} 
\1_{D_n^{(L)} \geq n^{-7/4},
	D_n^{(L),[k_n,n]} \leq n^{-2}} }
+
\Q^{(L)} \left( D_n^{(L)} < n^{-7/4} \right) 
\right).  \label{es}
\end{align}
As for \eqref{ef}, \eqref{es} is equal to $(1 + \petito{1}) / R(L)$ and, coming back to \eqref{er}, we get 
\begin{align*}
\E_{\Q^{(L)}}^{[k_n,n]} 
\leq
c_{26} \frac{\e^L}{R(L)}
\frac{\rho_n}{n^{3 \beta_n/2}}
\frac{\alpha_n^2}{n^{3 \beta_n/2}}
(1+\petito{1})
= 
\petito{\left( \frac{\alpha_n^2}{n^{3\beta_n/2}} \right)^2},
\end{align*}
because $\rho_n \ll \alpha_n^2$.
This concludes the proof of Lemma \ref{lemma second moment}.
\end{proof}

\begin{appendix}

\section{Convergence of random measures}

In this section, we present some results concerning convergence of random or deterministic probability measures on a polish space $S$ or more specifically on $\cD([0,1])$.
Some of these results are classical, but we state them here with uniformity in some parameter $\theta \in \Theta$.

\subsection{General space}

Let $(S,d)$, $(S_1,d_1)$ and $(S_2,d_2)$ be Polish spaces.
We consider some set $\Theta$.
In the sequel, for each $\theta \in \Theta$, $(\mu^\theta_n)_{n\in\N}$ will denote a sequence of random probability measures on $S$ and $(\xi^\theta_n)_{n\in\N}$ a sequence of deterministic probability measures on $S$.
Moreover, $\mu$ and $\xi$ will be probability measures on $S$, that are deterministic and do not depend on $\theta$.
\begin{lem} \label{lem weak convergence uc to c}
Assume that for all Lipschitz $F \in \cC_b(S)$ and $\varepsilon>0$,
$\P(\lvert \mu^\theta_n(F) - \mu (F) \rvert \geq \varepsilon) \to 0$ 
as $n \to \infty$ uniformly in $\theta \in \Theta$.
Then, the same convergence holds for all $F \in \cC_b(S)$.
\end{lem}
Note that in the case where we consider a deterministic sequence $(\xi^\theta_n)_{n\in\N}$, it simply means that $\xi^\theta_n (F) \to \xi (F)$ uniformly in $\theta \in \Theta$.
It is necessary that the limit does not depend on $\theta$.
\begin{proof}
We follow the proof of Portmanteau Theorem in Billingsley \cite[Theorem 2.1]{billingsley99}.
Thus, we first consider a closed set $A$ and $\varepsilon>0$ and we want to show that
\begin{align} \label{va}
\lim_{n\to\infty} 
\sup_{\theta \in \Theta}
\Pp{\mu^\theta_n(A) \geq \mu (A) + \varepsilon}
=
0.
\end{align}
We consider, for each $\eta > 0$, the function $F^\eta \colon x \mapsto 1 - (1 \wedge \eta^{-1} d(x,A)) \in [0,1]$ that is Lipschitz and such that $F^\eta \downarrow \1_A$.
Thus, by dominated convergence, we have $\mu(F^\eta) \to \mu(A)$ as $\eta \to 0$.
We fix $\eta$ such that $\mu(F^\eta) \leq \mu(A) + (\varepsilon /2)$.
Since $\mu^\theta_n(F^\eta) \geq \mu^\theta_n(A)$, we get
\begin{align*}
\sup_{\theta \in \Theta}
\Pp{\mu^\theta_n(A) \geq \mu (A) + \varepsilon}
\leq
\sup_{\theta \in \Theta}
\Pp{\mu^\theta_n(F^\eta) \geq \mu (F^\eta) + \frac{\varepsilon}{2}}
\xrightarrow[n\to\infty]{}
0,
\end{align*}
by using the assumption of the lemma.
From \eqref{va}, we get that, for all set $A$ such that $\mu(\partial A) = 0$ $\P$-a.\@s.\@ and all $\varepsilon>0$,
\begin{align} \label{vc}
\lim_{n\to\infty} 
\sup_{\theta \in \Theta}
\Pp{\abs{\mu^\theta_n(A) - \mu (A)} \geq \varepsilon}
=
0.
\end{align}
We now consider $F \in \cC_b(S)$ and we can assume that $F$ is nonnegative.
We fix $\varepsilon>0$ and set $M \coloneqq \norme{F}$.
Firstly, we have
\begin{align*}
\int_0^M \Ec{\mu( \partial \{ F > t \} )} \diff t
\leq \int_0^M \Ec{\mu( \{ F = t \} )} \diff t
= \Ec{\int_S \int_0^M \1_{F(x) = t} \diff t \diff \mu (x)}
= 0,
\end{align*}
and, therefore, for almost every $t \in [0,M]$ (in the sense of the Lebesgue measure), $\P$-a.\@s.\@, $\mu( \partial \{ F > t \} ) = 0$.
Thus, for all $N \in \N$, we can fix a subdivision $0 = t_0 < t_1 < \dots < t_N = M$ such that for all $0\leq k \leq N-1$, $t_{k+1} - t_k \leq 2M/N$ and for all $1\leq k \leq N-1$, $\mu( \partial \{ F > t_k \} ) = 0$ $\P$-a.\@s.
Since $\mu (F) = \int_0^M \mu (F > t) \diff t$ and $t \mapsto \mu (F > t)$ is nonincreasing,
we have
\begin{align} \label{vd}
\sum_{k=1}^N (t_k - t_{k-1}) \mu (F > t_k)
\leq
\mu (F)
\leq
\sum_{k=0}^{N-1} (t_{k+1} - t_k) \mu (F > t_k)
\end{align}
and the same holds for $\mu_n^\theta$ instead of $\mu$, for all $n\in\N$ and $\theta \in\Theta$.
Since, in \eqref{vd}, the left-hand side and right-hand side of \eqref{vd} tend to $\mu(F)$ as $N\to\infty$, we can choose $N$ large enough such that they are at most at distance $\varepsilon /2$ from $\mu (F)$.
Then, using \eqref{vd} for $\mu_n^\theta$, it follows that
\begin{align*}
\sup_{\theta \in \Theta}
\Pp{\abs{\mu^\theta_n(F) - \mu (F)} \geq \varepsilon}
& \leq
\sup_{\theta \in \Theta} 
\Biggl(
\Pp{
\sum_{k=1}^{N-1} 
(t_{k+1} - t_k) 
\left( \mu^\theta_n (F > t_k) - \mu (F > t_k) \right)
\geq 
\frac{\varepsilon}{4} } \\
& \relphantom{\leq} \phantom{\sup_{\theta \in \Theta} \Biggl(} {}
+
\Pp{
\sum_{k=1}^{N-1}
(t_{k+1} - t_k) 
\left( \mu^\theta_n (F > t_k) - \mu (F > t_k) \right)
\leq 
- \frac{\varepsilon}{4} }
\Biggr),
\end{align*}
by noting that, in the sums, the term for $k=N$ is zero, since $t_N = M = \norme{F}$, and the term for $k=0$ is smaller than $2 t_1 \leq 4M/N \leq \varepsilon/ 4$, if we choose $N$ large enough.
Then, we have
\begin{align*}
\sup_{\theta \in \Theta}
\Pp{\abs{\mu^\theta_n(F) - \mu (F)} \geq \varepsilon}
& \leq
\sup_{\theta \in \Theta}
\sum_{k=1}^{N-1} 
\Pp{
\abs{ \mu^\theta_n (F > t_k) - \mu (F > t_k) }
\geq 
\frac{\varepsilon N}{8M(N-1)} }
\xrightarrow[n\to\infty]{}
0,
\end{align*}
by using \eqref{vc}, because for all $1\leq k \leq N-1$, $\mu( \partial \{ F > t_k \} ) = 0$ $\P$-a.\@s.
\end{proof}

\begin{lem} \label{lemma weak convergence G_1 G_2}
We consider the product space $S \coloneqq S_1 \times S_2$.
Assume that, for all $G_1 \in \cC_b^u(S_1)$ and $G_2 \in \cC_b^u(S_2)$,
$\xi^\theta_n (G_1 \star G_2) \to \xi (G_1 \star G_2)$ as $n \to \infty$ uniformly in $\theta \in \Theta$, 
where $G_1 \star G_2 \colon (x,y) \in S \mapsto G_1(x) G_2(y)$.
Then, for all $F \in \cC_b(S)$, $\xi^\theta_n(F) \to \xi (F)$ uniformly in $\theta \in \Theta$.
\end{lem}
\begin{proof}
Using Lemma \ref{lem weak convergence uc to c}, it is sufficient to consider $F \in \cC_b^u(S)$.
Let $\varepsilon >0$.
Since $F$ is uniformly continuous, it exists $\eta > 0$ such that for any $x,y \in S$ that verify $d(x,y) \leq \eta$, we have $\abs{F(x) - F(y)} \leq \varepsilon$.
Since $S_2$ is separable for the metric $d$, it exists $(y_i)_{i\in\N} \in S_2^\N$ such that
$\bigcup_{i\in\N} B_2(y_i,\eta /2) =  S_2$,
where $B_2(y, r)$ denotes the open ball of radius $r$ centred at $y$ in $S_2$.
Now we consider a compact set $K \subset S$ such that $\xi (K) \geq 1 - \varepsilon$ and $K'$ the image of $K$ under the canonical projection $S \to S_2$.
Since $K'$ is a compact set of $S_2$, we can extract a finite cover $K' \subset \bigcup_{i=0}^N B(y_i,\eta /2)$.
Using again the compacity of $K'$, there exist nonnegative Lipschitz functions $\chi_0, \dots, \chi_N \in \cC_b(S_2)$ such that, for all $0 \leq i \leq N$, $\supp \chi_i \subset B(y_i,\eta /2)$, $\chi_0+ \dots+ \chi_N \leq 1$ and $\chi_0+ \dots+ \chi_N = 1$ on $K'$.
Finally, we set $\chi \coloneqq 1 \star (\chi_0+ \dots+ \chi_N)$, so that $\chi$ is a Lipschitz continuous function from $S \to [0,1]$ and $\chi = 1$ on $K$.
We can now construct some functions of the form $G_1 \star G_2$ to approach $F$.
For $0 \leq i \leq N$, we set
$G^i_1 \colon x \in S_1 \mapsto F( x, y_i)$
and
$G^i_2 \colon y \in S_2 \mapsto \chi_i(y)$.
By the triangle inequality, we have that $\lvert \xi^\theta_n(F) - \xi (F) \rvert$ is smaller than
\begin{align} \label{vg}
\abs{\xi^\theta_n \left( F - \sum_{i=0}^N G^i_1 \star G^i_2 \right)}
+
\sum_{i=0}^N \abs{\xi^\theta_n(G^i_1 \star G^i_2) - \xi (G^i_1 \star G^i_2)}
+
\abs{\xi \left( F - \sum_{i=0}^N G^i_1 \star G^i_2 \right)}.
\end{align}
The second term in \eqref{vg} tends to 0 as $n \to \infty$ uniformly in $\theta \in \Theta$ by the assumption of the lemma.
On the other hand, we have
\begin{align*}
\abs{F - \sum_{i=0}^N G^i_1 \star G^i_2}
\leq
\norme{F} \left( 1 - \chi \right)
+
\abs{\chi F  - \sum_{i=0}^N G^i_1 \star G^i_2}
\leq 
\norme{F} (1 - \chi)
+
\varepsilon \chi,
\end{align*}
because of the choice of $\eta$.
Thus, we get that the first and third terms of \eqref{vg} are smaller than
$2 \varepsilon
+
\norme{F} ( \xi^\theta_n (1 - \chi) +  \xi(1 - \chi))$.
Since $1 -\chi = 1 \star (1 - \chi_0 - \dots - \chi_N)$, we can use again the assumption of the lemma to get that $\xi^\theta_n (1 - \chi) \to \xi(1 - \chi)$ as $n\to\infty$ uniformly in $\theta \in \Theta$.
Noting that $1 - \chi \leq \1_{K^c}$ and $\xi(K^c) \leq \varepsilon$, it proves Lemma \ref{lemma weak convergence G_1 G_2}.
\end{proof}

\subsection{Weak convergence in \texorpdfstring{$\cD([0,1])$}{D([0,1])}}
\label{section appendix D([0,1])}

We keep here the notation of the previous subsection, but we take $S = \cD([0,1])$.
Recall the definition of the Skorokhod distance $d$ on $\cD ([0,1])$: for $x,y \in \cD([0,1])$,
\begin{align*}
d(x,y)
\coloneqq
\inf_{\lambda \in \Lambda}
\left(
\norme{\lambda - \mathrm{id}}_\infty
\vee
\norme{x - y \circ \lambda}_\infty \right),
\end{align*}
where we set $\Lambda \coloneqq \enstq{ \lambda \colon [0,1] \to [0,1]}{\lambda(0) = 0, \lambda(1) = 1, \lambda \text{ continuous and increasing}}$, and that, equipped with this distance, $\cD ([0,1])$ is a polish space (see Billingsley \cite{billingsley99}).
\begin{lem} \label{lemma weak convergence approximation F_n F}
Assume that, for all $F \in \cC_b^u(\cD ([0,1]))$, we have
$\xi^\theta_n(F) \to \xi (F)$ as $n \to \infty$ uniformly in $\theta \in \Theta$.
We consider $F \in \cC_b^u(\cD ([0,1]))$, $(u_n)_{n \in \N} \in (\R_+^*)^\N$, $(v_n)_{n \in \N} \in \cD ([0,1])^\N$ and $(\lambda_n)_{n \in \N} \in \Lambda^\N$ such that $u_n \to 1$, $\lVert v_n \rVert_\infty \to 0$ and $\lVert \lambda_n - \mathrm{id} \rVert_\infty \to 0$. 
We set $F_n \colon x \in \cD([0,1]) \mapsto F(v_n + u_n (x \circ \lambda_n))$.
Then, we have $\xi^\theta_n(\lvert F_n - F \rvert ) \to 0$ uniformly in $\theta \in \Theta$.
\end{lem}
\begin{proof}
For $x \in \cD([0,1])$, we first have
\begin{align}
d(x, v_n + u_n (x \circ \lambda_n))
& \leq
\lVert v_n \rVert_\infty
+
d(x, x \circ \lambda_n)
+
d(x \circ \lambda_n, u_n (x \circ \lambda_n)) \nonumber \\
& \leq
\lVert v_n \rVert_\infty
+
\lVert \lambda_n - \mathrm{id} \rVert_\infty
+
\abs{u_n -1}
\lVert x \rVert_\infty. \label{vi}
\end{align}
Now, we consider $\varepsilon >0$ and 
we fix $K > 0$ such that $\xi(\{ \lVert x \rVert_\infty \geq K \}) \leq \varepsilon$ and 
some Lipschitz function $\chi \colon \cD([0,1]) \to [0,1]$ such that $\1_{\lVert x \rVert_\infty < K} \leq \chi(x) \leq \1_{\lVert x \rVert_\infty < K+1}$ for all $x \in \cD([0,1])$ 
(this is possible since $\lVert - \rVert_\infty$ is Lipschitz on $\cD([0,1])$).
Thus, we have
\begin{align} \label{wc}
\xi^\theta_n(\lvert F_n - F \rvert ) 
\leq 
\norme{F_n}
\xi^\theta_n(1 - \chi)
+
\xi^\theta_n(\chi \lvert F_n - F \rvert)
+
\norme{F}
\xi^\theta_n(1 - \chi).
\end{align}
On the one hand, for $x \in \cD([0,1])$, we have, using $\chi(x) \leq \1_{\lVert x \rVert_\infty < K+1}$ and \eqref{vi},
\begin{align*}
\chi(x) \lvert F_n(x) - F(x) \rvert
& \leq
\omega_F
( \lVert v_n \rVert_\infty
+
\lVert \lambda_n - \mathrm{id} \rVert_\infty
+
\abs{u_n -1} (K+1)),
\end{align*}
so we get $\lvert \xi^\theta_n(\chi (F_n - F)) \rvert \leq \norme{\chi (F_n - F)} \to 0$ as $n \to \infty$ uniformly in $\theta \in \Theta$.
On the other hand, by using the assumption of the lemma with the function $1-\chi$, 
the first and third terms in the right-hand side of \eqref{wc} tend towards 
$2 \norme{F} \xi(1 - \chi)$ uniformly in $\theta \in \Theta$.
Since $\xi(1 - \chi) \leq \varepsilon$, it concludes the proof of Lemma \ref{lemma weak convergence approximation F_n F}.
\end{proof}
\begin{rem} \label{rem weak convergence approximation F_n F}
In Lemma \ref{lemma weak convergence approximation F_n F}, if the limit measure $\xi^\theta$ depends on $\theta \in \Theta$, then the result is still true under the additional assumption that $\sup_{\theta \in \Theta} \xi(\{ \lVert x \rVert_\infty \geq K \}) \to 0$ as $K \to \infty$, so that in the proof $K$ and $\chi$ could be chosen independently of $\theta$.
\end{rem}

\begin{lem} \label{lemma weak convergence approximation F_n F 2}
Assume that $\xi(\{ x(0) = 0 \}) = 1$ and that, for all $F \in \cC_b^u(\cD ([0,1]))$, we have
$\xi^\theta_n(F) \to \xi (F)$ as $n\to\infty$ uniformly in $\theta \in \Theta$.
We consider $F \in \cC_b^u(\cD ([0,1]))$ and $(\kappa_n)_{n \in \N}$ such that $\kappa_n \to 0$. 
We set $F_n \colon x \in \cD([0,1]) \mapsto F(x_{\kappa_n + (1-\kappa_n)t} - x_{\kappa_n},t \in [0,1])$.
Then, we have $\xi^\theta_n(\abs{F_n - F}) \to 0$ uniformly in $\theta \in \Theta$.
\end{lem}
\begin{proof}
The function $\varphi_n \colon t \in [0,1] \mapsto \kappa_n + (1-\kappa_n)t$ is not bijective from $[0,1]$ to $[0,1]$, so we consider a function $\lambda_n \in \Lambda$ such that $\lambda_n (t) = \kappa_n + (1-\kappa_n)t$ for $t \geq \kappa_n$ and that is linear on $[0,\kappa_n]$.
Then, for $x \in \cD([0,1])$, we have
\begin{align}
d(x,x \circ \varphi_n - x_{\kappa_n})
& \leq
\norme{\lambda_n - \mathrm{id}}_\infty
\vee
\norme{x \circ \lambda_n - x \circ \varphi_n - x_{\kappa_n}}_\infty 
\leq
\kappa_n
\vee
3 \max_{[0,\kappa_n]} \abs{x}. \label{vl}
\end{align}
Let $\varepsilon > 0$.
Since $\xi(\{ x(0) = 0 \}) = 1$, it exists $\delta > 0$ such that $\xi (\{\max_{[0,\delta]} \abs{x} > \varepsilon \}) \leq \varepsilon$.
Let $\chi \colon \cD([0,1]) \to [0,1]$ be a Lipschitz function such that
$
\1_{\max_{[0,\delta]} \abs{x} < \varepsilon} \leq \chi(x) \leq \1_{\max_{[0,\delta]} \abs{x} < 2 \varepsilon}
$
for all $x \in \cD([0,1])$.
Thus, we have, using the triangle inequality and then \eqref{vl},
\begin{align*}
\xi^\theta_n(\abs{F_n - F}) 
& \leq 
\norme{F_n}
\xi^\theta_n(1 - \chi)
+
\xi^\theta_n(\chi \abs{F_n - F})
+
\norme{F}
\xi^\theta_n(1 - \chi) \\
& \leq 
2 \norme{F}
\xi^\theta_n(1 - \chi)
+
\omega_F( \kappa_n
\vee
6 \varepsilon),
\end{align*}
for $n$ large enough such that $\kappa_n \leq \delta$ (so independent of $\theta$). 
Finally, using the assumption of the lemma with the function $1 - \chi$ and recalling that $\xi(1 - \chi) \leq \varepsilon$, it concludes the proof. 
\end{proof}

\section{Proofs of the preliminary results concerning random walk}

In this section, we prove the results stated in Section \ref{section preliminary results}. 
Subsection \ref{subsection appendix convergence towards the bessel and the meander} is devoted to the proof of Proposition \ref{prop CV vers le meandre} and Corollary \ref{cor convergence to the meander with extension}, and Subsection \ref{subsection lower envelope} to the proof of Lemma \ref{lemme trajectoire supérieure à i^alpha}.
Subsections \ref{subsection appendix local limit theorems} and \ref{subsection appendix convergence towards the Bessel bridge} contain preliminary results for the proof of Proposition \ref{prop convergence to the excursion} in Subsection \ref{subsection appendix convergence towards the Brownian excursion}.

\subsection{Convergence towards the 3-dimensional Bessel process and the Brownian meander}
\label{subsection appendix convergence towards the bessel and the meander}

We first recall a known invariance principle for the random walk conditioned to stay nonnegative for all time.
For all $n\in \N$, $u \in \R_+$ and $B \in \sF_n$, we set
\begin{align} \label{eq definition P^+}
\P^+_u (B)
\coloneqq
\frac{1}{R(u)}
\Eci{u}{R(S_n) \1_B \1_{\underline{S}_n \geq 0}}.
\end{align}
It defines a probability measure $\P^+_u$, that is called the law of the random walk started at $u \in \R_+$ and conditioned to stay nonnegative for all time.
Then we have the following invariance principle, 
by Theorem 1.1 of Caravenna and Chaumont \cite{caravennachaumont2008}: 
for any $b \in \R_+$ and $(b_n)_{n\in\N}$ such that $b_n / \sigma \sqrt{n} \to b$ as $n\to\infty$ and for any $F \in \cC_b (\cD([0,1]))$,
\begin{align} \label{eq CV vers le Bessel}
\E^+_{b_n} \left[
	F ( \bfS^{(n)} )
	\right]
\xrightarrow[n\to\infty]{}
\Eci{b}
	{F( \cR )},
\end{align}
where $\cR$ denotes the 3-dimensional Bessel process on $[0,1]$.

Proposition \ref{prop CV vers le meandre} follows from \eqref{eq CV vers le Bessel} and from the following link between the 3-dimensional Bessel process and the Brownian meander (see Imhof \cite{imhof84}): for all $F \in \cC_b(\cD([0,1]))$, we have
\begin{align} \label{eq lien bessel meandre}
\Ec{F(\cM)}
=
\sqrt{\frac{\pi}{2}}
\Ec{\frac{1}{\cR(1)} F(\cR)}.
\end{align}
\begin{proof}[Proof of Proposition \ref{prop CV vers le meandre}]
We can assume that $F$ is nonnegative.
For $K>0$, we consider $\chi \colon \R_+ \to [0,1]$ continuous such that $\1_{[K^{-1}, K]} \leq \chi \leq \1_{[(2K)^{-1}, 2K]}$.
On the one hand, we have
\begin{align}
\Eci{u}{F(\bfS^{(n)}) 
	\left( 1 - \chi \left( \frac{S_n}{\sigma \sqrt{n}} \right) \right)
	\1_{\underline{S}_n \geq 0}}
& \leq 
\norme{F} 
\frac{\theta R(u)}{\sqrt{n}}
\left(
\int_0^\infty (1 - \chi(t)) t \e^{-t^2/2} \diff t
+
\petito{1}
\right), \label{bx}
\end{align}
uniformly in $u\in [0, \gamma_n]$, using \eqref{equation convergence vers la loi Rayleigh}.
On the other hand, we have
\begin{align} \label{by}
\Eci{u}{F(\bfS^{(n)})
	\chi \left( \frac{S_n}{\sigma \sqrt{n}} \right)
	\1_{\underline{S}_n \geq 0}}
= 
\frac{R(u)}{\sqrt{n}}
\E^+_u \left[ 
	h_n \left( \frac{S_n}{\sigma \sqrt{n}} \right)
	F(\bfS^{(n)})
	\right],
\end{align}
where we set, for $z \in \R_+$, $h_n(z) \coloneqq \sqrt{n} \chi(z) / R(z \sigma \sqrt{n})$.
Using \eqref{equation R 1} and $\chi \leq \1_{[(2K)^{-1}, 2K]}$, we have
$\limsup_{n\to\infty} 
\sup_{z \in \R_+}
\abs{h_n(z) - h(z)}
=
0$,
with $h(z) \coloneqq \chi(z) / c_0 \sigma z$.
Thus, we get, uniformly in $u$ because $h_n$ and $h$ do not depend on $u$,
\begin{align*}
\limsup_{n\to\infty} 
\abs{
\E^+_u \left[ 
	h_n \left( \frac{S_n}{\sigma \sqrt{n}} \right)
	F(\bfS^{(n)})
	\right]
-
\E^+_u \left[ 
	h \left( \frac{S_n}{\sigma \sqrt{n}} \right)
	F(\bfS^{(n)})
	\right]
}
= 0.
\end{align*}
Moreover, using \eqref{eq CV vers le Bessel}, \eqref{eq lien bessel meandre} and \eqref{equation lien theta c_0}, we have
\begin{align*}
\E^+_u \left[ 
	h \left( \frac{S_n}{\sigma \sqrt{n}} \right)
	F(\bfS^{(n)})
	\right]
& \xrightarrow[n\to\infty]{}
\Ec{\frac{\chi(\cR(1))}{c_0 \sigma \cR(1)} 
	F(\cR(t), t \in [0,1])}
=
\theta
\Ec{\chi(\cM(1)) F(\cM)},
\end{align*}
uniformly in $u \in [0,\gamma_n]$.
Coming back to \eqref{bx} and \eqref{by} and using that the density of $\cM(1)$ is $t \mapsto t \e^{-t^2/2} \1_{t > 0}$, we showed that
\begin{align*}
\limsup_{n\to\infty} 
\sup_{u \in [0,\gamma_n]}
\abs{
\frac{\sqrt{n}}{\theta R(u)}
\Eci{u}{F(\bfS^{(n)}) 
	\1_{\underline{S}_n \geq 0}}
-
\Ec{F(\cM)}
}
& \leq 
2 \norme{F} 
\int_0^\infty (1 - \chi(t)) t \e^{-t^2/2} \diff t,
\end{align*}
which tends to 0 as $K\to\infty$, so it concludes the proof.
\end{proof}

\begin{proof}[Proof of Corollary \ref{cor convergence to the meander with extension}]
By Lemma A.2 of Madaule \cite{madaule2016} (that holds under the assumption of this corollary, see Remark \ref{rem assumption madaule}), it exists $c_{27}(L,C) > 0$ such that for all $n$ large enough and $K \geq 0$,
\begin{align} \label{wz}
\Ec{\e^{C S_n / \sqrt{n}} \1_{\underline{S}_n \geq -L, S_n \geq K \sqrt{n}}}
\leq
\frac{c_{27}(L,C)}{\sqrt{n}} \e^{-CK/2}.
\end{align}
We consider some $\varepsilon > 0$ and fix $K$ large enough such that $c_{27}(L,C) \e^{-CK/2} \leq \varepsilon \theta R(L) / \norme{F}$ and also $\E[\e^{C \sigma \cM(1)} \1_{\cM(1) \geq K}] \leq \varepsilon / \norme{F}$.
Considering a continuous function $\chi \colon \R \to [0,1]$ such that $\1_{x \leq K} \leq \chi(x) \leq \1_{x \leq K+1}$ and using the triangle inequality, we get
\begin{align}
& \abs{
\frac{\sqrt{n}}{\theta R(L)} 
\Ec{\e^{C S_n / \sqrt{n}} 
	F(\bfS^{(n)}) 
	\1_{\underline{S}_n \geq -L}}
-
\Ec{\e^{C \sigma \cM(1)} F(\cM)}} \nonumber \\
& \leq
2 \varepsilon
+
\abs{
\frac{\sqrt{n}}{\theta R(L)} 
\Ec{\e^{C S_n / \sqrt{n}} 
	F(\bfS^{(n)}) 
	\chi \left( \frac{S_n}{\sqrt{n}} \right)
	\1_{\underline{S}_n \geq -L}
	}
-
\Ec{\e^{C \sigma \cM(1)} 
	F(\cM) 
	\chi (\sigma \cM(1))
	}}. \label{wy}
\end{align}
Then, note that in Proposition \ref{prop CV vers le meandre}, we can replace $\E_{u} [F(\bfS^{(n)}) \1_{\underline{S}_n \geq 0}]$ by $\E [F(\bfS^{(n)}) \1_{\underline{S}_n \geq -u}]$: indeed, it works when $F \in \cC_b^u(\cD([0,1]))$ and we extend to $F \in \cC_b(\cD([0,1]))$ by Lemma \ref{lem weak convergence uc to c}.
Thus, applying Proposition \ref{prop CV vers le meandre} with the function $x \in \cD([0,1]) \mapsto \e^{C \sigma x_1} F(x) \chi (\sigma x_1)$ which belongs to $\cC_b(\cD([0,1]))$
we get that the right-hand of \eqref{wy} tends to $2\varepsilon$ as $n \to \infty$ and it concludes the proof of Corollary \ref{cor convergence to the meander with extension}.
\end{proof}

\subsection{Lower envelope for the random walk above two barriers}
\label{subsection lower envelope}

We prove here Lemma \ref{lemme trajectoire supérieure à i^alpha}.
Firstly, by Equation (A.9) of Shi \cite{shi2015}, we have: 
it exists $c_{28} >0$ such that, for all $b > a \geq 0$, $u \geq 0$ and $n \geq 1$,
\begin{align} \label{eq ballot theorem 2}
\Pp{S_n \in [b-u, b-u+1], \underline{S}_n \geq -u} 
\leq
c_{28} \frac{b+1}{n}.
\end{align}
Secondly, it exists $c_{29} >0$ such that for all $b > a \geq 0$, $u \geq 0$, $v \in\R$ and $n \geq k \geq 0$, 
\begin{align} \label{eq ballot theorem and barrier 2} 
\Pp{S_n  \in [b+v,b+v+1], \underline{S}_k \geq -u, 
	\min_{k \leq j \leq n} S_j \geq v} 
\leq
c_{29} \frac{(u+1)(b+1)}{(n-m)\sqrt{k+1}}. 
\end{align}
Indeed, the left-hand side of \eqref{eq ballot theorem and barrier 2} is equal to
\begin{align*}
&\Ec{\1_{\underline{S}_k \geq -u} 
\Ppi{S_k}{S_{n-k} \in [b+v,b+v+1], \underline{S}_{n-k} \geq v}} 
\leq 
\Ec{\1_{\underline{S}_k \geq -u} c_{28} \frac{b+1}{n-k}},
\end{align*}
by using \eqref{eq ballot theorem 2}, and then we get \eqref{eq ballot theorem and barrier 2} by applying \eqref{equation majoration proba min S geq -a}. 

\begin{proof}[Proof of Lemma \ref{lemme trajectoire supérieure à i^alpha}]
Recall that, for $\ell, i \in \llbracket 0, n \rrbracket$, $u, \mu \geq 0$ and $v \in \R$, we have
\begin{align*}
m_i^{(n,\ell)}
\coloneqq
\left\{
\begin{array}{ll}
- u + r_i - \mu & \text{if } 0 \leq i < \ell, \\
v + r_{n-i} - \mu & \text{if } \ell \leq i \leq n.
\end{array}
\right.
\end{align*}
Moreover, for $J \subset \llbracket 0,n \rrbracket$, we set
\begin{align*}
\P_J (n, \ell, u,v,b,\mu)
&\coloneqq \Pp{
	\underline{S}_\ell \geq -u,
	\min_{\ell \leq j \leq n} S_j \geq v,
	S_n \in [b+v, b+v+1],
	\Exists i \in J:
	S_i \leq m_i^{(n,\ell)}}
\end{align*}
and write simply $\P_J= \P_J (n, \ell, u,v,b,\mu)$ when the parameters are obvious.
For any $\varepsilon > 0$ and $\lambda \in (0,1/2)$, we want to prove that for $\mu$ large enough, for any $b,u \geq 0$, $v \in \R$, $n\in \N$ and $\ell \in [\lambda n, (1-\lambda)n]$, $\P_{\llbracket 0,n \rrbracket} \leq \varepsilon (1+u) (1+b) n^{-3/2}$.
For this, it is sufficient to prove that 
\begin{align} 
\begin{split} \label{za}
& \Forall \varepsilon >0, \Forall \lambda \in (0,1/2), 
\Exists \mu > 0 : 
\Forall b,u \geq 0, \Forall v \in \R, \Forall n\in \N, 
\Forall \ell \in [\lambda n, (1-\lambda)n], \\
& \P_{\llbracket 0,\ell \rrbracket} \leq \varepsilon (1+u) (1+b) n^{-3/2}.
\end{split}
\end{align}
Indeed, we have $\P_{\llbracket 0,n \rrbracket} \leq \P_{\llbracket 0,\ell \rrbracket} + \P_{\llbracket \ell +1,n \rrbracket}$ and, setting $\tilde{S}_k \coloneqq S_{n-k} - S_n$,
\begin{align*}
\P_{\llbracket \ell +1,n \rrbracket}
& \leq \P \Bigl(
	\min_{n-\ell \leq j \leq n} \tilde{S}_j \geq -u-(b+v+1),
	\underline{\tilde{S}}_{n-\ell} \geq -b-1, \\
&	\relphantom{\leq} \hphantom{\P \Bigl(}
	\tilde{S}_n \in [-(b+v+1), -(b+v)],
	\Exists i \in \llbracket 0,n-\ell-1 \rrbracket:
	\tilde{S}_i \leq m_{n-i}^{(n,\ell)} - (b+v) \Bigr) \\
& = \tilde{\P}_{\llbracket 0,n-\ell-1 \rrbracket}(n,n-\ell,b+1,-u-b-v-1,u,\mu -1),
\end{align*}
where $\tilde{\P}_J$ is the analogue of $\P_J$ for the random walk $\tilde{S}$.
Then, using \eqref{za} for $\P_{\llbracket 0,\ell \rrbracket}$ and for $\tilde{\P}_{\llbracket 0,n-\ell-1 \rrbracket}$, we get the wanted bound for $\P_{\llbracket 0,n \rrbracket}$.

We now prove \eqref{za}.
We first have $\P_{\llbracket 0,\ell \rrbracket} \leq \P_{\llbracket 0,i_0 \rrbracket} + \P_{\llbracket i_0,\ell \rrbracket}$, 
for some $i_0 \geq 1$ that will be chosen afterwards.
We take $\mu = r_{i_0} + 1$, so that $m_i^{(n,\ell)} < -u$ for $0 \leq i \leq i_0$, and thus $\P_{\llbracket 0,i_0 \rrbracket} = 0$.
We set $\tau \coloneqq \max \{ i \in \llbracket 0,\ell \rrbracket: S_i \leq m_i^{(n,\ell)} \}$ and get
\begin{align*}
\P_{\llbracket i_0,\ell \rrbracket}
& =
\sum_{i=i_0}^{\ell}
\Pp{	\underline{S}_n \geq -u,
	\min_{\ell \leq j \leq n} S_j \geq v,
	S_n \in [b+v, b+v+1],
	\tau = i} 
= \sum_{i=i_0}^{\ell} p_i,
\end{align*}
where, for $i \in \llbracket i_0, \ell \rrbracket$, we set
\begin{align*}
p_i &\coloneqq
\Pp{\underline{S}_n \geq -u,
	\min_{\ell \leq j \leq n} S_j \geq v,
	S_n \in [b+v, b+v+1],
	S_i \leq m_i^{(n,\ell)}, 
	\forall k \in \llbracket i+1,\ell \rrbracket, S_k > m_k^{(n,\ell)} }.
\end{align*}
We first consider the case $i \in \llbracket i_0, \ell -1 \rrbracket$ and the case $i=\ell$ will be treated after.
Applying the Markov property at time $i+1$ and noting that for $k \in \llbracket i+1,\ell \rrbracket$, $m_k^{(n,\ell)} \geq m_i^{(n,\ell)}$, we have
\begin{align}
\begin{split} \label{ia}
p_i 
& \leq 
\E \Bigl[
\1_{\underline{S}_i \geq -u, S_i \leq m_i^{(n,\ell)}}
\1_{S_{i+1} > m_i^{(n,\ell)}}
\P_{S_{i+1}} \Bigl(
	\underline{S}_{\ell-i-1} > m_i^{(n,\ell)}, \\
& \relphantom{\leq} \hphantom{\E \Bigl[}
\min_{\ell-i-1 \leq j \leq n-i-1} S_j \geq v,
	S_{n-i-1} \in [b+v, b+v+1] \Bigr)
\Bigr].
\end{split}
\end{align}
If $i \leq \lambda n/ 2$, we can use \eqref{eq ballot theorem and barrier} to bound the probability in \eqref{ia} (because $\ell-i-1 \geq \lambda n/ 2$), but, if $i > \lambda n/ 2$, we use instead \eqref{eq ballot theorem and barrier 2}.
Thus, the probability in \eqref{ia} is smaller than
\begin{align*} 
(1+S_{i+1}-m_i^{(n,\ell)})(1+b)
\left( 
\frac{2 c_5}{(n-i-1)^{3/2}} \1_{i \leq \lambda n/ 2}
+\frac{c_{29}}{(n-\ell)(\ell -i)^{1/2}} \1_{i > \lambda n/ 2}
\right).
\end{align*}
Setting $\xi \coloneqq S_{i+1} - S_i$, we get
\begin{align} \label{zb}
p_i
& \leq 
c_{30} \frac{(1+b)}{n^{3/2}} 
\left( 
1+\frac{\sqrt{n} \1_{i > \lambda n/ 2}}{(\ell -i)^{1/2}}
\right)
\Ec{
\1_{\underline{S}_i \geq -u, S_i \leq m_i^{(n,\ell)}}
\1_{\xi > m_i^{(n,\ell)}-S_i}
(1+S_i+\xi-m_i^{(n,\ell)})}.
\end{align}
Cutting the interval $[-u,m_i^{(n,\ell)}]$ in pieces of length 1, the expectation in \eqref{zb} is smaller than
\begin{align*}
& \sum_{k=1}^{\lceil m_i^{(n,\ell)} +u \rceil} 
\Ec{
\1_{\underline{S}_i \geq -u, S_i \in [m_i^{(n,\ell)} - k, m_i^{(n,\ell)} - k +1]}
\1_{\xi > k-1}
(2-k+\xi)} \\
& \leq 
\sum_{k=1}^{\lceil m_i^{(n,\ell)} +u \rceil} 
c_4\frac{(1+u)(1+u + m_i^{(n,\ell)} - k +1) 2}{i^{3/2}} 
\Ec{\1_{\xi > k-1} (2-k+\xi)} ,
\end{align*}
by noting that $\xi$ is independent of $(S_j)_{0\leq j\leq i}$ and applying \eqref{eq ballot theorem}.
Recalling that $m_i^{(n,\ell)} = r_i - u - \mu \leq r_i - u$ and coming back to \eqref{zb}, we get
\begin{align}
p_i
& \leq c_{31}
\frac{(1+b)(1+u)}{n^{3/2}}
\frac{(1+r_i)}{i^{3/2}} 
\left( 1+\frac{\sqrt{n} \1_{i > \lambda n/ 2}}{(\ell -i)^{1/2}} \right)
\sum_{k \geq 1} \Ec{\1_{\xi > k-1} (2+\xi)}. \label{ib}
\end{align}
Moreover, we have
\begin{align} \label{zc}
\sum_{k \geq 1} \Ec{\1_{\xi > k-1} (2+\xi)}
= \sum_{i \geq 0} (i+1) \Ec{\1_{\xi \in (i,i+1]} (2+\xi)} 
\leq \Ec{(2+\xi)^2} 
= 4+\sigma^2, 
\end{align}
using that $\xi$ has the same law than $S_1$.
We now deal with the case $i = \ell$: we have
\begin{align*} 
p_{\ell} 
& \leq 
\Ec{\1_{\underline{S}_{\ell} \geq -u, 
	S_{\ell} \leq m_{\ell}^{(n,\ell)}}
	\Ppi{S_i}{\underline{S}_{n-\ell} \geq v,
		S_{n-\ell} \in [b+v, b+v+1] \Bigr)} } \\
& \leq 
\Ec{\1_{\underline{S}_{\ell} \geq -u, 
	S_{\ell} \leq v + r_{n-\ell}}
	c_5 \frac{(1+S_{\ell}-v) (1+b) 2}{(n-\ell)^{3/2}} } \\
& \leq c_{32} \frac{(1+r_{n-\ell}) (1+b)}{n^{3/2}} 
	\frac{(1+u)}{\sqrt{n}},
\end{align*}
by using first \eqref{eq ballot theorem and barrier} and then \eqref{equation majoration proba min S geq -a}.
Combining this with \eqref{ib} and \eqref{zc}, we get
\begin{align*}
\sum_{i=i_0}^{\ell} p_i
& \leq c_{33}
\frac{(1+b)(1+u)}{n^{3/2}}
\left( 
\sum_{i=i_0}^{\ell-1} \frac{(1+r_i)}{i^{3/2}} 
+ \sum_{i= \lfloor\lambda n/2 \rfloor +1}^{\ell-1}
\frac{(1+r_i)}{i^{3/2}} 
\frac{\sqrt{n}}{(\ell -i)^{1/2}} 
+ \frac{(1+r_n)}{\sqrt{n}}  \right) \\
& \leq c_{33}
\frac{(1+b)(1+u)}{n^{3/2}}
\left( 
\sum_{i=i_0}^\infty \frac{(1+r_i)}{i^{3/2}} 
+ \frac{(1+r_n)\sqrt{n}}{(\lfloor\lambda n/2 \rfloor + 1)^{3/2}} 
\sum_{j=1}^{\ell - \lfloor\lambda n/2 \rfloor -1}
\frac{1}{j^{1/2}} 
+ \frac{(1+r_n)}{\sqrt{n}} \right) \\
& \leq c_{34}
\frac{(1+b)(1+u)}{n^{3/2}}
\left( 
\sum_{i=i_0}^\infty \frac{(1+r_i)}{i^{3/2}} 
+ \frac{(1+r_n)}{n} 
\sum_{j=1}^n
\frac{1}{j^{1/2}} 
+ \frac{(1+r_n)}{\sqrt{n}} \right) \\
& \leq c_{35}
\frac{(1+b)(1+u)}{n^{3/2}}
\left( 
\sum_{i=i_0}^\infty \frac{(1+r_i)}{i^{3/2}} 
+ \frac{(1+r_n)}{\sqrt{n}} \right),
\end{align*}
where we used that $(r_n)_{n\in\N}$ is an increasing sequence and the constants depend only on $\lambda$.
Since furthermore $\sum_{n\geq1} r_n n^{-3/2} < \infty$, we have $r_n / \sqrt{n} \to 0$. 
Thus, we can choose $i_0$ such that for all $n \geq i_0$, $c_{35} (1+r_n)/\sqrt{n} \leq \varepsilon/4$.
Moreover, we can choose $i_0$ such that $c_{35} \sum_{i=i_0}^\infty \frac{(1+r_i)}{i^{3/2}} \leq \varepsilon/4$. 
This concludes the proof of \eqref{za} and, therefore, of the lemma.
\end{proof}

\subsection{Local limit theorems}
\label{subsection appendix local limit theorems}

We first recall the classical Stone's \cite{stone65} local limit theorem: letting $h>0$ be any real number if $S_1$ is nonlattice and be the span of the lattice if $S_1$ is lattice, we have
\begin{align} \label{eq TLL Stones}
\Pp{S_n \in [b,b+h)}
=
\frac{h}{\sigma \sqrt{2 \pi n}}
\e^{- b^2 / 2 \sigma^2 n}
+
\petito{\frac{1}{\sqrt{n}}},
\end{align}
as $n\to\infty$, uniformly in $b \in \R$.
Thus we have the following uniform bound: it exists $c_{36} > 0$ such that, for all $n \geq 1$ and $b \in \R$,
\begin{align} \label{eq TLL Stones uniform}
\Pp{S_n \in [b,b+1)}
\leq
\frac{c_{36}}{\sqrt{n}}.
\end{align}
Now, we state a local limit theorem for the random walk staying above a barrier, in the case where the starting point is at distance of order $\sqrt{n}$ from the barrier and the endpoint at distance $o(\sqrt{n})$.
\begin{lem} \label{lemme deviations normales inverse}
Let $(\gamma_n)_{n\in \N}$ be a sequence of positive numbers such that $\gamma_n \ll \sqrt{n}$ as $n\to \infty$.
We set $f \colon t \in \R \mapsto t \e^{-t^2/2} \1_{t \geq 0}$.
\begin{enumerate}
\item If the law of $S_1$ is nonlattice, then, for all $h>0$, 
\begin{align*}
\Pp{S_n \in [u-b, u-b+h), \underline{S}_n \geq -b} 
=
\frac{\theta^-}{\sigma n} 
f \left( \frac{b}{\sigma \sqrt{n}} \right)
\int_u^{u+h} R^-(t) \diff t
+ \petito{\frac{R^-(u)}{n}},
\end{align*}
as $n \to \infty$, uniformly in $b \in \R$ and in $u \in [0, \gamma_n]$.
\item If the law of $S_1$ is $(h,a)$-lattice, then, 
\[
\Pp{S_n = u-b, \underline{S}_n \geq -b} 
=
\frac{\theta^-}{\sigma n} 
f \left( \frac{b}{\sigma \sqrt{n}} \right)
h R^-(u)
+ \petito{\frac{R^-(u)}{n}},
\]
as $n \to \infty$, uniformly in $b \in \R$ and $u \in [0, \gamma_n] \cap (b + an + h\Z)$.
\end{enumerate}
\end{lem}
\begin{proof}
First note that, for each $D >0$, by Propositions 11, 18 and 24 of Doney \cite{doney2012}, both estimates of Lemma \ref{lemme deviations normales inverse} holds uniformly in $b \in [D^{-1} \sqrt{n}, D \sqrt{n}]$ and $u \in [0, \gamma_n)$%
\footnote{Doney states his results in terms of the renewal function for the first \textit{weak} increasing ladder height process of $S$ and of $\P ( \min_{1 \leq k \leq n} S_n >0)$. 
Our formulation follows from Remark 4.6 of Caravenna and Chaumont \cite{caravennachaumont2013} (although they work only in the lattice and absolutely continuous cases, this remark does not rely on these assumptions).}.
Noting also that $f$ tends to $0$ at $0$ and at infinity, it is sufficient to prove that, for each $h>0$ and $\varepsilon >0$, there exist $D$ and $n$ large enough such that $\Pp{S_n \in [u-b, u-b+h), \underline{S}_n \geq -b} \leq \varepsilon R^-(u) / n$ for all $b \notin [D^{-1} \sqrt{n}, D \sqrt{n}]$ and $u \in \R_+$.
Reversing time, we can equivalently prove that $\Pp{S_n \in [b-u, b-u+h), \underline{S}_n \geq -u} \leq \varepsilon R(u) / n$.
For $b \in [0, D^{-1} \sqrt{n}]$, using \eqref{eq ballot theorem} and \eqref{equation R 2}, we get 
$\Pp{S_n \in [b-u, b-u+h), \underline{S}_n \geq -u} 
\leq
c_{37} R(u) (1+D^{-1} \sqrt{n}) / n^{3/2}$.
For $D$ and $n$ large enough and independent of $u$, this is smaller than $\varepsilon R(u)/n$.
For $b > D \sqrt{n}$, we cancel the lower barrier between times $\lfloor n/2 \rfloor + 1$ and $n$ so that we get
\begin{align} \label{se}
\Pp{S_n \in [b-u, b-u+h), \underline{S}_n \geq -u} 
& \leq
\Ec{\1_{\underline{S}_{\lfloor n/2 \rfloor} \geq -u} 
	h\left( 
		\frac{S_{\lfloor n/2 \rfloor} + u}{\sigma \lfloor n/2 \rfloor^{1/2}}
	\right)},
\end{align}
where we set for $x \geq 0$, $h(x) \coloneqq
\P ( S_{n-\lfloor n/2 \rfloor} + x\sigma \lfloor n/2 \rfloor^{1/2} 
	\in	[ b, b+h ) )$.
Using \eqref{eq TLL Stones}, we have, uniformly in $b > D \sqrt{n}$ and $x \in \R_+$,
\begin{align*}
h(x)
& \leq 
\frac{c_{38}}{\sqrt{n}}
\exp \left( 
- \frac{(x \sigma \lfloor n/2 \rfloor^{1/2} - b)^2}{2 \sigma^2 (n-\lfloor n/2 \rfloor)}  \right)
+ \petito{\frac{1}{\sqrt{n}}} 
\leq 
\frac{c_{38}}{\sqrt{n}} 
g_D(x)
+ \petito{\frac{1}{\sqrt{n}}},
\end{align*}
where $g_D \colon \R_+ \to [0,1]$, obtained by taking the supremum on $b > D \sqrt{n}$, is continuous and converges simply to 0 as $D \to \infty$.
Thus, using \eqref{equation convergence vers la loi Rayleigh} and \eqref{equation majoration proba min S geq -a}, we get that \eqref{se} is smaller than
\begin{align}
\frac{c_{38}}{\sqrt{n}}
\frac{\theta R(u)}{\sqrt{n}}
\left(
\int_0^\infty g_D(t) t \e^{-t^2/2} \diff t
+
\petito{1}
\right)
+
\frac{c_3 R(u)}{\sqrt{n}}
\petito{\frac{1}{\sqrt{n}}}, \label{bz}
\end{align}
uniformly in $b \geq D \sqrt{n}$ and $u \in[0,\gamma_n]$.
By the dominated convergence theorem, the integral in \eqref{bz} tends to 0 as $D\to \infty$ and, thus, \eqref{se} is smaller than $\varepsilon R(u) / n$ for $D$ and $n$ large enough and independent of $u$.
\end{proof}

\subsection{Convergence towards the Bessel bridge}
\label{subsection appendix convergence towards the Bessel bridge}

In this subsection, we are going to prove the following generalization of Lemma 2.4 of Chen, Madaule and Mallein \cite{cmm2015arxiv}.
It proves that conditioned by the event of Lemma \ref{lemme deviations normales inverse} the trajectory $\bfS^{(n)}$ converges to the Bessel bridge.
This is a first step in the proof of Proposition \ref{prop convergence to the excursion}.
\begin{lem} \label{lem convergence pont de bessel}
Let $(\gamma_n)_{n\in \N}$ be a sequence of positive numbers such that $\gamma_n \ll \sqrt{n}$ as $n\to \infty$.
We set $f \colon t \in \R \mapsto t \e^{-t^2/2} \1_{t \geq 0}$ and we denote by $\rho^1_{b,0}$ the 3-dimensional Bessel bridge of length $1$ from $b \in \R_+$ to $0$.
\begin{enumerate}
\item If the law of $S_1$ is nonlattice, then, for all $h>0$ and $F \in \cC_b^u(\cD([0,1]))$,
\begin{align*}
\Eci{b\sigma \sqrt{n}}
	{F(\bfS^{(n)})
	\1_{S_n \in [u,u+h), 
		\underline{S}_n \geq 0}}
=
\frac{\theta^-}{\sigma n} 
\int_u^{u+h} R^-(t) \diff t
f(b)
\Ec{F(\rho^1_{b,0})}
+ \petito{\frac{R^-(u)}{n}},
\end{align*}
as $n \to \infty$, uniformly in $b \in \R$ and in $u \in [0, \gamma_n]$.
\item If the law of $S_1$ is $(h,a)$-lattice, then, for all $F \in \cC_b^u(\cD([0,1]))$,
\begin{align*}
\Eci{b\sigma \sqrt{n}}
	{F(\bfS^{(n)})
	\1_{S_n = u, 
		\underline{S}_n \geq 0}}
=
\frac{\theta^-}{\sigma n} 
h R^-(u)
f(b)
\Ec{F(\rho^1_{b,0})}
+ \petito{\frac{R^-(u)}{n}},
\end{align*}
as $n \to \infty$, uniformly in $b \in \R$ and $u \in [0, \gamma_n] \cap (b \sigma \sqrt{n} + an + h\Z)$.
\end{enumerate}
\end{lem}
\begin{proof}
We will treat only the nonlattice case, because the proof in the lattice case is exactly the same, with $h R^-(u)$ instead of $\int_u^{u+h} R^-(t) \diff t$.
Moreover, since $F$ is bounded, $f(b) \to 0$ as $b \to \infty$ and Lemma \ref{lemme deviations normales inverse} deals with the case $F \equiv 1$,
it is sufficient to show that, for each $K>0$, the estimate holds uniformly in $b\in[0,K]$ instead of $b \in \R_+$.
We first assume that $\Forall x \in \cD ([0,1])$, $F(x) = F'(x_t,t\in[0,1-\varepsilon])$ for some $F' \in \cC^u_b (\cD ([0,1-\varepsilon]))$.
Thus, $F(\bfS^{(n)})$ is $\sF_m$-measurable with $m \coloneqq \lceil (1-\varepsilon)n \rceil$ and we have
\begin{align} \label{bt}
\Eci{b\sigma \sqrt{n}}
	{F(\bfS^{(n)})
	\1_{S_n \in [u,u+h), 
		\underline{S}_n \geq 0}}
& =
\Eci{b\sigma \sqrt{n}}
	{F(\bfS^{(n)})
	g \left( \frac{S_m}{\sigma \sqrt{n}} \right)
	\1_{\underline{S}_m \geq 0}},
\end{align}
where we set, for $z \in \R_+$,
\begin{align*}
g(z)
& \coloneqq
\Ppi{z\sigma \sqrt{n}}
	{\underline{S}_{n-m} \geq 0, 
	S_{n-m} \in [u,u+h)} 
=
\frac{\theta^-}{\sigma \varepsilon n} 
f \left( \frac{z}{\sqrt{\varepsilon}} \right)
\int_u^{u+h} R^-(t) \diff t
+ \petito{\frac{R^-(u)}{n}},
\end{align*}
uniformly in $u \in [0,\gamma_n]$ and $z \in \R_+$, using Lemma \ref{lemme deviations normales inverse}.
Therefore, \eqref{bt} is equal to
\begin{align*}
\frac{\theta^-}{\sigma \varepsilon n}
\int_u^{u+h} R^-(t) \diff t
\Eci{b\sigma \sqrt{n}}
	{F(\bfS^{(n)})
	 f \left( \frac{S_m}{\sigma \sqrt{\varepsilon n}} \right)
	\1_{\underline{S}_m \geq 0}}
+
\Eci{b\sigma \sqrt{n}}
	{F(\bfS^{(n)})
	\1_{\underline{S}_m \geq 0}}
	\petito{\frac{R^-(u)}{n}},
\end{align*}
uniformly in $u \in [0,\gamma_n]$ and $b \in [0,K]$.
Thus, it is now sufficient to prove that
\begin{align} \label{bs}
\Eci{b\sigma \sqrt{n}}
	{F(\bfS^{(n)})
	 f \left( \frac{S_m}{\sigma \sqrt{\varepsilon n}} \right)
	\1_{\underline{S}_m \geq 0}}
\xrightarrow[n\to \infty]{}
\varepsilon f(b)
\Ec{F(\rho^1_{b,0})},
\end{align}
uniformly in $b \in [0,K]$.
On the one hand, by Equation (2.30) of Chen, Madaule and Mallein \cite{cmm2015arxiv}, we have
\begin{align} \label{br}
\varepsilon f(b) \Ec{F(\rho^1_{b,0})}
= 
\frac{b}{\sqrt{\varepsilon}} 
\Eci{b}
	{F'( \cR(s), s\in [0,1-\varepsilon]) 
	\e^{-\cR(1-\varepsilon)^2/2\varepsilon}}.
\end{align}
On the other hand, recalling \eqref{eq definition P^+}, we get
\begin{align} \label{bq}
\Eci{b\sigma \sqrt{n}}
	{F(\bfS^{(n)})
	 f \left( \frac{S_m}{\sigma \sqrt{\varepsilon n}} \right)
	\1_{\underline{S}_m \geq 0}} 
& = 
\E^+_{b\sigma \sqrt{n}}
\left[
F(\bfS^{(n)})
h_{\varepsilon,b}^n \left( \frac{S_m}{\sigma \sqrt{n}} \right)
\right],
\end{align}
where we set, for $z \in \R_+$,
\begin{align*}
h_{\varepsilon,b}^n (z)
\coloneqq
f \left( \frac{z}{\sqrt{\varepsilon}} \right)
\frac{R(b \sigma \sqrt{n})}{R(z \sigma \sqrt{n})}
=
\frac{z}{\sqrt{\varepsilon}} \e^{-z^2/2 \varepsilon}
\frac{R(b \sigma \sqrt{n})}{R(z \sigma \sqrt{n})}.
\end{align*}
Since $\e^{-z^2/2 \varepsilon} \to 0$ as $z \to \infty$ and using \eqref{equation R 1}, it is clear that
\begin{align} \label{bp}
\sup_{b\in [0,K],z \in \R_+}
\abs{
h_{\varepsilon,b}^n (z)
-
\frac{b}{\sqrt{\varepsilon}}
\e^{-z^2/2 \varepsilon}} 
\xrightarrow[n \to \infty]{} 0,
\end{align}
so we get, combining \eqref{br}, \eqref{bq} and \eqref{bp},
\begin{align*}
& \limsup_{n\to \infty}
\sup_{b\in [0,K]}
\abs{
\Eci{b\sigma \sqrt{n}}
	{F(\bfS^{(n)})
	 f \left( \frac{S_m}{\sigma \sqrt{\varepsilon n}} \right)
	\1_{\underline{S}_m \geq 0}}
-
\varepsilon f(b)
\Ec{F(\rho^1_{b,0})}
} \\
& \leq
\limsup_{n\to \infty}
\sup_{b\in [0,K]}
\frac{b}{\sqrt{\varepsilon}}
\abs{
\E^+_{b\sigma \sqrt{n}}
	\left[
	F(\bfS^{(n)})
	\e^{-S_m^2/2 \sigma^2 \varepsilon n}
	\right]
-
\Eci{b}
	{F'( \cR(s), s\in [0,1-\varepsilon]) 
	\e^{-\cR(1-\varepsilon)^2/2\varepsilon}}
},
\end{align*}
which is equal to 0 by applying \eqref{eq CV vers le Bessel}.
It proves \eqref{bs} and so it concludes the case where $F(x) = F'(x_t,t\in[0,1-\varepsilon])$ for some $F' \in \cC^u_b (\cD ([0,1-\varepsilon]))$.
We now want to extend the result to the case $F \in \cC^u_b (\cD ([0,1]))$.
For $\varepsilon >0$ and $x \in \cD ([0,1])$, we define $\varphi_\varepsilon (x) \in \cD ([0,1])$ by $\varphi_\varepsilon (x)\restreinta_{[0,1-\varepsilon)} = x\restreinta_{[0,1-\varepsilon)}$ and $\varphi_\varepsilon (x)\restreinta_{[1-\varepsilon,1]} \equiv 0$,
so that $F \circ \varphi_\varepsilon$ satisfies the assumption of the particular case that is already proved.
Thus, it is now sufficient to show that, for each $\eta > 0$, it exists $\varepsilon > 0$ such that
\begin{align} \label{bo}
\limsup_{n \to \infty}
\sup_{u \in [0,\gamma_n], b \in [0,K]}
\Eci{b\sigma \sqrt{n}}
	{\abs{F(\bfS^{(n)}) - F \circ \varphi_\varepsilon(\bfS^{(n)})}
	\1_{S_n \in [u,u+h), 
		\underline{S}_n \geq 0}}
& \leq
\eta \frac{R^-(u)}{n}, \\
\sup_{b \in [0,K]}
f(b) \Ec{\abs{F(\rho^1_{b,0}) - F \circ \varphi_\varepsilon(\rho^1_{b,0})}} 
& \leq
\eta. \label{bn}
\end{align}
We first prove \eqref{bn}: we have
\begin{align*}
\Ec{\abs{F(\rho^1_{b,0}) - F \circ \varphi_\varepsilon(\rho^1_{b,0})}} 
& \leq
\Ec{\omega_F(\lVert \rho^1_{b,0} - \varphi_\varepsilon(\rho^1_{b,0})  \rVert_\infty)}
\leq
\Ec{\omega_F \left( \sup_{t\in[1-\varepsilon,1]} \rho^1_{b,0}(t) \right)}
\eqqcolon
\E_{b,\varepsilon}.
\end{align*}
Since the function $b \mapsto \E_{b,\varepsilon}$ is nondecreasing and $f \leq 1$, the left-hand side of \eqref{bn} is smaller than $\E_{K,\varepsilon}$ and, thus, it tends to 0 as $\varepsilon \to 0$ by dominated convergence, because $\omega_F$ is bounded.
Now, we prove \eqref{bo}: 
using that 
$d(\bfS^{(n)}, \varphi_\varepsilon(\bfS^{(n)})) 
\leq \max_{0 \leq k \leq \varepsilon n} S_{n-k}/\sigma \sqrt{n}$ and reversing time, we get that the expectation in \eqref{bo} is smaller than
\begin{align}
\Ec{\omega_F \left( 
		\sup_{0 \leq k \leq \varepsilon n} 
		\frac{S^-_k + u +h}{\sigma \sqrt{n}} 
		\right)
	\1_{S^-_n - b\sigma \sqrt{n} \in (-u-h,-u], 
		\underline{S}^-_n \geq -u-h}}. \label{bm}
\end{align}
Conditioning with respect to $\sF_{\lfloor n/2 \rfloor}$ and applying \eqref{eq TLL Stones uniform}, we get that \eqref{bm} is smaller than
\begin{align}
& \frac{c_{36} (1+h)}{(n-\lfloor n/2 \rfloor)^{1/2}}
\Ec{\omega_F \left( 
		\sup_{0 \leq k \leq \varepsilon n} 
		\frac{S^-_k + u +h}{\sigma \sqrt{n}} 
		\right)
	\1_{\underline{S}^-_{\lfloor n/2 \rfloor} \geq -u-h}} \nonumber \\
& =
\frac{c_{36} (1+h)}{\sqrt{n/2}}
\frac{\theta^- R^-(u+h)}{\sqrt{n/2}}
\left(
\Ec{\omega_F \left( \frac{1}{\sqrt{2}}
		\sup_{t \in [0, 2 \varepsilon]} 
		\cM_t
		\right)}
+
\petito{1}
\right), \label{bl}
\end{align}
uniformly in $u \in [0,\gamma_n]$ et $b \in [0,K]$, by using Lemma \ref{prop CV vers le meandre} to get the last equality.
The expectation in the right-hand side of \eqref{bl} does not depend on $n$, $b$ and $u$ and tends to 0 as $\varepsilon \to 0$ by dominated convergence, so it shows \eqref{bo} and concludes the proof. 
\end{proof}

\subsection{Convergence towards the Brownian excursion}
\label{subsection appendix convergence towards the Brownian excursion}

We prove here Proposition \ref{prop convergence to the excursion}, in a similar way as Lemma 2.5 of Chen, Madaule and Mallein \cite{cmm2015arxiv}, but directly with a first barrier that can be different of 0.
Following \cite{cmm2015arxiv}, we fix some $\lambda \in (0,1)$ and, for $G_1 \colon \cD([0,\lambda]) \to \R$, $G_2 \colon \cD([0,1-\lambda]) \to \R$ and $x \in \cD([0,1])$, we set
\begin{align*}
G_1 \star G_2 (x) 
\coloneqq
G_1 (x_s, s \in [0,\lambda])
G_2 (x_{\lambda+s}, s \in [0,1-\lambda]).
\end{align*}
Then, by \cite[Lemma 2.3]{cmm2015arxiv}, we have, for each $G_1 \in \cC_b(\cD([0,\lambda]))$ and $G_2 \in \cC_b(\cD([0,1-\lambda]))$,
\begin{align} \label{eq lien excursion meandre pont de bessel}
\Ec{G_1 \star G_2 (\mathfrak{e})}
=
\sqrt{\frac{2}{\pi}}
\frac{1}{\lambda^{1/2} (1- \lambda)^{3/2}}
\Ec{\cM(\lambda) 
	\e^{-\cM(\lambda)^2/2(1-\lambda)}
	G_1(\cM)
	G_2(\rho^{1-\lambda}_{\cM(\lambda),0})
},
\end{align}
where $\mathfrak{e}$ is the normalized Brownian excursion, $\cM$ is the Brownian meander of length $\lambda$ and $(\rho^{1-\lambda}_{z,0})_{z\in\R_+}$ is a family of Bessel bridges of length $1-\lambda$ from $z$ to $0$, independent of $\cM$.
\begin{proof}[Proof of Proposition \ref{prop convergence to the excursion}]
Using Lemma \ref{lemma weak convergence G_1 G_2}, it is sufficient to consider the case where $F = G_1 \star G_2$ for some $G_1 \in \cC_b^u(\cD([0,\lambda]))$ and $G_2 \in \cC_b^u(\cD([0,1-\lambda]))$.
The proof will be only treated in the lattice case: in the nonlattice case, the reasoning is exactly the same with $\int_b^{b+h} R^-(t) \diff t$ instead of $h R^-(b)$.
We set $m \coloneqq \lfloor \lambda n \rfloor$ and take the conditional expectation according to $\sF_m$ to get that
\begin{align} \label{bk}
& \Ec{F(\bfS^{(n)})
	\1_{\underline{S}_{\lfloor \lambda n \rfloor} \geq -u, 
	\min_{\lfloor \lambda n \rfloor \leq i \leq n} S_i \geq v,
	S_n = v + b}} 
=
\Ec{G_1 \left( \frac{S_{\lfloor sn \rfloor}}{\sigma \sqrt{n}}, s\leq \lambda \right)
	\varphi \left( \frac{S_m}{\sigma \sqrt{n}} \right)
	\1_{\underline{S}_m \geq -u}},
\end{align}
where we set
$\varphi (z) \coloneqq
\E_{z \sigma \sqrt{n}}
	[G_2(S_{\lfloor (s + \lambda) n \rfloor -m}/\sigma \sqrt{n}, s\leq 1-\lambda)
	\1_{\underline{S}_{n -m}  \geq v, S_n = v + b}]$, for each $z \in \R$ such that $v + b + z \sigma \sqrt{n} \in a(n-m) + h \Z$.
Then, since $G_2$ is uniformly continuous, using Lemma \ref{lem convergence pont de bessel} combined with Lemma \ref{lemma weak convergence approximation F_n F}%
\footnote{Here the limit measure is the law of $\rho^{1-\lambda}_{z,0}$ that depends on the parameter $(b,v,z)$. Thus, using Remark \ref{rem weak convergence approximation F_n F}, we should prove that $\P(\lVert \rho^{1-\lambda}_{z,0} \rVert_\infty > K) \to 0$ as $K \to \infty$ uniformly in $z \in \R_+$, but it is obviously false.
However, since $f ( z / \sqrt{1-\lambda}) \to 0$ as $z \to \infty$, it is sufficient to have, for each $M >0$, $\P(\lVert \rho^{1-\lambda}_{z,0} \rVert_\infty > K) \to 0$ as $K \to \infty$ uniformly in $z \in [0,M]$ and this is clearly true.}, 
we get that $\varphi(z)$ is equal to
\begin{align*}
& \frac{\theta^-}{\sigma (1-\lambda) n} 
h R^-(b) 
f \left( \frac{z}{\sqrt{1-\lambda}} \right) 
\Ec{G_2 \left( \sqrt{1-\lambda} \rho^1_{z/\sqrt{1-\lambda},0} (s/(1-\lambda)), 
				s\leq 1-\lambda \right)}
+ \petito{\frac{R^-(b)}{n}},
\end{align*}
uniformly in $v \in [-\gamma_n, \gamma_n]$, $b \in [0, \gamma_n] \cap (-v + an + h\Z)$ and $z \in \R$ such that $v + b + z \sigma \sqrt{n} \in a(n-m) + h \Z$ (note that if $z<0$, the fact that the Bessel bridge is not well-defined is not a problem, because $f(z/\sqrt{1-\lambda}) = 0$ so the first term is zero).
Since the last expectation is equal to $\E[G_2 ( \rho^{1-\lambda}_{z,0} )]$ by scaling properties of the Bessel bridge, \eqref{bk} is equal to
\begin{align}
\begin{split} \label{bj}
& \frac{\theta^-}{\sigma (1-\lambda) n} 
h R^-(b)
\Ec{G_1 \left( \frac{S_{\lfloor sn \rfloor}}{\sigma \sqrt{n}}, s\leq \lambda \right)
	f \left( \frac{S_m}{\sqrt{1-\lambda} \sigma \sqrt{n}} \right)
	G_2 \left( \rho^{1-\lambda}_{S_m / \sigma \sqrt{n},0} \right)
	\1_{\underline{S}_m \geq -u}} \\
& {} +
\Ec{G_1 \left( \frac{S_{\lfloor sn \rfloor}}{\sigma \sqrt{n}}, s\leq \lambda \right)
	\1_{\underline{S}_m \geq -u}}
\petito{\frac{R^-(b)}{n}},
\end{split}
\end{align}
where $(\rho^{1-\lambda}_{z,0})_{z\in\R_+}$ is independent of $(S_n)_{n\in\N}$.
Then, since the function $z \in \R \mapsto \E[G_2 ( \rho^{1-\lambda}_{z,0} )]$ is continuous, using Lemma \ref{prop CV vers le meandre} combined with Lemma \ref{lemma weak convergence approximation F_n F}, we get that \eqref{bj} is equal to
\begin{align*}
& \frac{\theta^-}{\sigma (1-\lambda) n} 
h R^-(b)
\frac{\theta R(u)}{\sqrt{m}} 
\left( 
\Ec{G_1(\cM) 
	f \left( \frac{\cM(\lambda)}{\sqrt{1-\lambda}} \right)
	G_2 \left( \rho^{1-\lambda}_{\cM(\lambda),0} \right)}
+ \petito{1} 
\right) 
+
\petito{\frac{R^-(b) R(u)}{n^{3/2}}},
\end{align*}
where $\cM$ is the Brownian meander of length $\lambda$, independent of $(\rho^{1-\lambda}_{z,0})_{z\in\R_+}$.
Finally, recalling the definition of $f$ and using \eqref{eq lien excursion meandre pont de bessel}, it concludes the proof of Proposition \ref{prop convergence to the excursion}.
\end{proof}

\section{Trajectories in the weak disorder regime}
\label{section appendix weak disorder}

We prove here \eqref{eq cv trajectories weak disorder regime}. We consider some $\beta <1$ and $p >1$ such that $\E[W_{1,\beta}^p] < \infty$.
Using also \eqref{hypothese 2}, \eqref{hypothese 3} and the convexity of $\Psi$, we get that $\Psi$ is finite and nonincreasing on $[\beta,1]$ and positive on $[\beta,1)$.
Moreover, we can assume that $p \leq 2$ and that $p \beta < 1$.
Then, we have $\Psi(p\beta) \leq \Psi(\beta) < p \Psi(\beta)$ and, by Theorem 1 of Biggins \cite{biggins92}, we get that $\widetilde{W}_{n,\beta} \to \widetilde{W}_{\infty,\beta}$ in $L^p$.
The proof follows the lines of the proof of part (iv) of Theorem \ref{theorem trajectory}, without having to introduce the barrier at $-L$.
By Lemma \ref{lem weak convergence uc to c}, we can reduce the proof to the case $F \in \cC_b^u(\cD([0,1]))$.
Since $\widetilde{W}_{\infty,\beta}$ is positive $\P^*$-a.\@s., it is sufficient to prove that
\begin{align} \label{pz}
U_n(F)
\coloneqq
\e^{-n \Psi(\beta)}
\sum_{\abs{z} = n} 
\e^{-\beta V(z)} 
	F \left( \widetilde{\bfV}^{(n)}(z) \right)
\xrightarrow[n\to\infty]{} \widetilde{W}_{\infty,\beta} \Ec{F(B)},
\end{align}
in $\P^*$-probability, where $\widetilde{\bfV}^{(n)}_t(z) \coloneqq [V(z_{\lfloor tn \rfloor}) + t n\Psi'(\beta)]/\sigma_\beta \sqrt{n}$ for $t \in [0,1]$.
We consider a sequence of integers $(k_n)_{n\in\N}$ such that $1 \ll k_n \ll n$ 
and, for each $x \in \cD([0,1])$, $F_n(x) \coloneqq F( x_{((n-k_n)t + k_n)/n} - x_{k_n / n}, t\in [0,1])$.
Then, using \eqref{eq many-to-one beta < 1} and setting $\bfS^{(n,\beta)} \coloneqq (S_{\lfloor tn \rfloor,\beta} / \sigma_{\beta} \sqrt{n})_{t \in [0,1]}$, we get
\begin{align*}
\E^*[\abs{U_n(F)-U_n(F_n)}]
& \leq \E^*[U_n(\abs{F-F_n})]
= \Ec{\abs{F-F_n} (\bfS^{(n,\beta)})}
\xrightarrow[n\to\infty]{} 0,
\end{align*}
by Lemma \ref{lemma weak convergence approximation F_n F 2}.
Thus, in order to prove \eqref{pz}, it is now sufficient to prove that $U_n(F_n) \to \widetilde{W}_{\infty,\beta} \Ec{F(B)}$ in $\P^*$-probability.

First, we prove that $\zeta_n \coloneqq U_n(F_n)- \Ecsq{U_n(F_n)}{\sF_{k_n}} \to 0$ in $\P^*$-probability. 
We set $\zeta'_n \coloneqq \E [ \abs{\zeta_n}^p | \sF_n]$ and, by \eqref{or}, we have 
$\P^*(\abs{\zeta_n} \geq \varepsilon)
\leq \varepsilon \P(S)^{-1} + \P^*(\zeta'_n \geq \varepsilon^{1+p_n})$.
By the branching property at time $k_n$, we have
\begin{align} \label{ps}
\zeta_n
= \sum_{\abs{x} = k_n} \e^{-\beta V(x)- k_n \Psi(\beta)} 
\left( \Upsilon_n^{(x)} - \Ec{\Upsilon_n} \right),
\end{align}
where, conditionally on $\sF_{k_n}$, the $\Upsilon_n^{(x)}$ for $\abs{x} = k_n$ are independent variables with the same law as $\Upsilon_n$ defined by
\begin{align} \label{pt}
\Upsilon_n
& \coloneqq 
\sum_{\abs{z} = n-k_n} \e^{-\beta V(z) - (n-k_n) \Psi(\beta)} 
F \left( \frac{V(z_{\lfloor t(n-k_n) \rfloor})+ t (n-k_n) \Psi'(\beta)}{\sigma_\beta \sqrt{n}}, 
	t\in [0,1] \right).
\end{align}
Since $p \in [1,2]$ and $c_{39} \coloneqq \sup_{n\in\N} \E [\widetilde{W}_{n,\beta}^p] < \infty$, we get, in the same way as for \eqref{oq},
\begin{align*}
\zeta'_n 
& \leq 2 c_{39} (4\norme{F})^p
\widetilde{W}_{k_n,p \beta} \e^{k_n (\Psi(p \beta) - p \Psi(\beta))}.
\end{align*}
Using that $\Psi(p \beta) - p \Psi(\beta) < 0$ and $\widetilde{W}_{k_n,p \beta} \to \widetilde{W}_{\infty,p \beta} < \infty$ $\P^*$-a.s.\@, we get $\zeta'_n \to 0$ $\P^*$-a.s.\@ and, therefore, $\zeta_n \to 0$ in $\P^*$-probability.

Finally, we prove that $\Ecsq{U_n(F_n)}{\sF_{k_n}} \to \widetilde{W}_{\infty,\beta} \Ec{F(B)}$ $\P^*$-a.s.
Using the branching property in the same way as for \eqref{ps}, we have $\Ecsq{U_n(F_n)}{\sF_{k_n}} = \widetilde{W}_{k_n,\beta} \Ec{\Upsilon_n}$,  where $\Upsilon_n$ is defined in \eqref{pt}.
By \eqref{eq many-to-one beta < 1} and recalling that $\bfS^{(n,\beta)} \coloneqq (S_{\lfloor tn \rfloor,\beta} / \sigma_{\beta} \sqrt{n})_{t \in [0,1]}$, we have 
\begin{align*}
\Ec{\Upsilon_n}
& = \Ec{F \left( \frac{S_{\lfloor t(n-k_n) \rfloor,\beta}
	+ (t (n-k_n)-\lfloor t(n-k_n) \rfloor) \Psi'(\beta)}{\sigma_\beta \sqrt{n}}, 
	t\in [0,1] \right)} \\
& = \Ec{F \left( u_n \bfS^{(n-k_n,\beta)} + v_n \right)},
\end{align*}
where 
$(u_n)_{n \in \N} \in (\R_+^*)^\N$ and $(v_n)_{n \in \N} \in \cD ([0,1])^\N$ satisfy $u_n \to 1$ and $\lVert v_n \rVert_\infty \to 0$.
Using that $\bfS^{(n-k_n,\beta)}$ converges in law towards the Brownian motion and applying Lemma \ref{lemma weak convergence approximation F_n F}, we get that $\Ec{\Upsilon_n} \to \Ec{F(B)}$ as $n \to \infty$.
On the other hand, we have $\widetilde{W}_{k_n,\beta} \to \widetilde{W}_{\infty,\beta}$ $\P^*$-a.s. 
Therefore, $\Ecsq{U_n(F_n)}{\sF_{k_n}} \to \widetilde{W}_{\infty,\beta} \Ec{F(B)}$ $\P^*$-a.s.\@ and it concludes the proof.

\end{appendix}

\label{FIN}
\addcontentsline{toc}{section}{References}
\bibliographystyle{abbrv}

\begin{thebibliography}{10}

\bibitem{aidekon2013}
E.~A{\"{\i}}d{\'e}kon.
\newblock Convergence in law of the minimum of a branching random walk.
\newblock {\em Ann. Probab.}, 41(3A):1362--1426, 2013.

\bibitem{aidekonjaffuel2011}
E.~A{\"{\i}}d{\'e}kon and B.~Jaffuel.
\newblock Survival of branching random walks with absorption.
\newblock {\em Stochastic Process. Appl.}, 121(9):1901--1937, 2011.

\bibitem{aidekonshi2014}
E.~A{\"{\i}}d{\'e}kon and Z.~Shi.
\newblock The {S}eneta-{H}eyde scaling for the branching random walk.
\newblock {\em Ann. Probab.}, 42(3):959--993, 2014.

\bibitem{akq2014-2}
T.~Alberts, K.~Khanin, and J.~Quastel.
\newblock The intermediate disorder regime for directed polymers in dimension
  {$1+1$}.
\newblock {\em Ann. Probab.}, 42(3):1212--1256, 2014.

\bibitem{albertsortgiese2013}
T.~Alberts and M.~Ortgiese.
\newblock The near-critical scaling window for directed polymers on disordered
  trees.
\newblock {\em Electron. J. Probab.}, 18:no. 19, 24, 2013.

\bibitem{brv2012}
J.~Barral, R.~Rhodes, and V.~Vargas.
\newblock Limiting laws of supercritical branching random walks.
\newblock {\em C. R. Math. Acad. Sci. Paris}, 350(9-10):535--538, 2012.

\bibitem{biggins77-1}
J.~D. Biggins.
\newblock Martingale convergence in the branching random walk.
\newblock {\em J. Appl. Probability}, 14(1):25--37, 1977.

\bibitem{biggins79}
J.~D. Biggins.
\newblock Growth rates in the branching random walk.
\newblock {\em Z. Wahrsch. Verw. Gebiete}, 48(1):17--34, 1979.

\bibitem{biggins92}
J.~D. Biggins.
\newblock Uniform convergence of martingales in the branching random walk.
\newblock {\em Ann. Probab.}, 20(1):137--151, 1992.

\bibitem{bigginskyprianou97}
J.~D. Biggins and A.~E. Kyprianou.
\newblock Seneta-{H}eyde norming in the branching random walk.
\newblock {\em Ann. Probab.}, 25(1):337--360, 1997.

\bibitem{bigginskyprianou2004}
J.~D. Biggins and A.~E. Kyprianou.
\newblock Measure change in multitype branching.
\newblock {\em Adv. in Appl. Probab.}, 36(2):544--581, 2004.

\bibitem{bigginskyprianou2005}
J.~D. Biggins and A.~E. Kyprianou.
\newblock Fixed points of the smoothing transform: the boundary case.
\newblock {\em Electron. J. Probab.}, 10:no. 17, 609--631, 2005.

\bibitem{billingsley99}
P.~Billingsley.
\newblock {\em Convergence of probability measures}.
\newblock Wiley Series in Probability and Statistics: Probability and
  Statistics. John Wiley \& Sons, Inc., New York, second edition, 1999.
\newblock A Wiley-Interscience Publication.

\bibitem{bolthausen76}
E.~Bolthausen.
\newblock On a functional central limit theorem for random walks conditioned to
  stay positive.
\newblock {\em Ann. Probability}, 4(3):480--485, 1976.

\bibitem{caravennachaumont2008}
F.~Caravenna and L.~Chaumont.
\newblock Invariance principles for random walks conditioned to stay positive.
\newblock {\em Ann. Inst. Henri Poincar\'e Probab. Stat.}, 44(1):170--190,
  2008.

\bibitem{caravennachaumont2013}
F.~Caravenna and L.~Chaumont.
\newblock An invariance principle for random walk bridges conditioned to stay
  positive.
\newblock {\em Electron. J. Probab.}, 18:no. 60, 32, 2013.

\bibitem{csz2014arxiv}
F.~Caravenna, R.~Sun, and N.~Zygouras.
\newblock Polynomial chaos and scaling limits of disordered systems.
\newblock 2014.
\newblock {\tt arXiv:1312.3357v2 [math.PR]}.

\bibitem{csz2016arxiv}
F.~Caravenna, R.~Sun, and N.~Zygouras.
\newblock Universality in marginally relevant disordered systems.
\newblock 2016.
\newblock {\tt arXiv:1510.06287v2 [math.PR]}.

\bibitem{chauvinrouault88}
B.~Chauvin and A.~Rouault.
\newblock K{PP} equation and supercritical branching {B}rownian motion in the
  subcritical speed area. {A}pplication to spatial trees.
\newblock {\em Probab. Theory Related Fields}, 80(2):299--314, 1988.

\bibitem{chen2015-1}
X.~Chen.
\newblock A necessary and sufficient condition for the nontrivial limit of the
  derivative martingale in a branching random walk.
\newblock {\em Adv. in Appl. Probab.}, 47(3):741--760, 2015.

\bibitem{chen2015-2}
X.~Chen.
\newblock Scaling limit of the path leading to the leftmost particle in a
  branching random walk.
\newblock {\em Theory Probab. Appl.}, 59(4):567--589, 2015.

\bibitem{cmm2015arxiv}
X.~Chen, T.~Madaule, and B.~Mallein.
\newblock On the trajectory of an individual chosen according to supercritical
  gibbs measure in the branching random walk.
\newblock 2015.
\newblock {\tt arXiv:1507.04506v1 [math.PR]}.

\bibitem{cometsyoshida2006}
F.~Comets and N.~Yoshida.
\newblock Directed polymers in random environment are diffusive at weak
  disorder.
\newblock {\em Ann. Probab.}, 34(5):1746--1770, 2006.

\bibitem{derridaspohn88}
B.~Derrida and H.~Spohn.
\newblock Polymers on disordered trees, spin glasses, and traveling waves.
\newblock {\em J. Statist. Phys.}, 51(5-6):817--840, 1988.
\newblock New directions in statistical mechanics (Santa Barbara, CA, 1987).

\bibitem{doney85}
R.~A. Doney.
\newblock Conditional limit theorems for asymptotically stable random walks.
\newblock {\em Z. Wahrsch. Verw. Gebiete}, 70(3):351--360, 1985.

\bibitem{doney2012}
R.~A. Doney.
\newblock Local behaviour of first passage probabilities.
\newblock {\em Probab. Theory Related Fields}, 152(3-4):559--588, 2012.

\bibitem{feller71}
W.~Feller.
\newblock {\em An introduction to probability theory and its applications.
  {V}ol. {II}.}
\newblock Second edition. John Wiley \& Sons, Inc., New York-London-Sydney,
  1971.

\bibitem{iglehart74}
D.~L. Iglehart.
\newblock Functional central limit theorems for random walks conditioned to
  stay positive.
\newblock {\em Ann. Probability}, 2:608--619, 1974.

\bibitem{imhof84}
J.-P. Imhof.
\newblock Density factorizations for {B}rownian motion, meander and the
  three-dimensional {B}essel process, and applications.
\newblock {\em J. Appl. Probab.}, 21(3):500--510, 1984.

\bibitem{jaffuel2012}
B.~Jaffuel.
\newblock The critical barrier for the survival of branching random walk with
  absorption.
\newblock {\em Ann. Inst. Henri Poincar\'e Probab. Stat.}, 48(4):989--1009,
  2012.

\bibitem{kahanepeyriere76}
J.-P. Kahane and J.~Peyri{\`e}re.
\newblock Sur certaines martingales de {B}enoit {M}andelbrot.
\newblock {\em Advances in Math.}, 22(2):131--145, 1976.

\bibitem{kozlov76}
M.~V. Kozlov.
\newblock The asymptotic behavior of the probability of non-extinction of
  critical branching processes in a random environment.
\newblock {\em Teor. Verojatnost. i Primenen.}, 21(4):813--825, 1976.

\bibitem{lyons95}
R.~Lyons.
\newblock A simple path to {B}iggins' martingale convergence for branching
  random walk.
\newblock In {\em Classical and modern branching processes ({M}inneapolis,
  {MN}, 1994)}, volume~84 of {\em IMA Vol. Math. Appl.}, pages 217--221.
  Springer, New York, 1997.

\bibitem{madaule2015}
T.~Madaule.
\newblock Convergence in law for the branching random walk seen from its tip.
\newblock {\em Journal of Theoretical Probability}, pages 1--37, 2015.

\bibitem{madaule2016}
T.~Madaule.
\newblock First order transition for the branching random walk at the critical
  parameter.
\newblock {\em Stochastic Process. Appl.}, 126(2):470--502, 2016.

\bibitem{mallein2016arxiv-1}
B.~Mallein.
\newblock Asymptotic of the maximal displacement in a branching random walk.
\newblock 2016.
\newblock {\tt arXiv:1605.08292v1 [math.PR]}.

\bibitem{mallein2016arxiv-2}
B.~Mallein.
\newblock Genealogy of the extremal process of the branching random walk.
\newblock 2016.
\newblock {\tt arXiv:1606.01748v2 [math.PR]}.

\bibitem{rouault81}
A.~Rouault.
\newblock Lois empiriques dans les processus de branchement spatiaux
  homog\`enes supercritiques.
\newblock {\em C. R. Acad. Sci. Paris S\'er. I Math.}, 292(20):933--936, 1981.

\bibitem{sakhanenko2006}
A.~I. Sakhanenko.
\newblock Estimates in the invariance principle in terms of truncated power
  moments.
\newblock {\em Sibirsk. Mat. Zh.}, 47(6):1355--1371, 2006.

\bibitem{shi2015}
Z.~Shi.
\newblock {\em Branching random walks}, volume 2151 of {\em Lecture Notes in
  Mathematics}.
\newblock Springer, Cham, 2015.
\newblock Lecture notes from the 42nd Probability Summer School held in Saint
  Flour, 2012, {\'E}cole d'{\'E}t{\'e} de Probabilit{\'e}s de Saint-Flour.
  [Saint-Flour Probability Summer School].

\bibitem{stone65}
C.~Stone.
\newblock A local limit theorem for nonlattice multi-dimensional distribution
  functions.
\newblock {\em Ann. Math. Statist.}, 36:546--551, 1965.

\bibitem{vonbahresseen65}
B.~von Bahr and C.-G. Esseen.
\newblock Inequalities for the {$r$}th absolute moment of a sum of random
  variables, {$1\leq r\leq 2$}.
\newblock {\em Ann. Math. Statist}, 36:299--303, 1965.

\end{thebibliography}
\def\cprime{$'$}

\end{document}